\documentclass[12pt]{article}
\usepackage{amsmath,amssymb}
\usepackage{amscd}
\usepackage{comment}
\usepackage{color}
\usepackage{enumerate}
\usepackage{cases}
\usepackage{mathtools}
\usepackage[all]{xy}
\usepackage{amsthm}
\usepackage{pb-diagram}
\usepackage[top=25truemm,bottom=25truemm,left=20truemm,right=20truemm]{geometry}
\theoremstyle{definition}
\newtheorem{defi}{Definition}
\newtheorem{theo}[defi]{Theorem}
\newtheorem{lemm}[defi]{Lemma}
\newtheorem{prop}[defi]{Proposition}
\newtheorem{cor}[defi]{Corollary}
\newtheorem{ex}[defi]{Example}
\newtheorem{rem}[defi]{Remark}

\def\det{\mathrm{det}}

\def\SO{\mathrm{SO}}

\def\SL{\mathrm{SL}}

\def\SU{\mathrm{SU}}

\def\O{\mathrm{O}}

\def\End{\mathrm{End}}

\def\refe{\ast}

\def\vol{\mathrm{vol}}

\def\max{\mathrm{max}}

\def\min{\mathrm{min}}

\def\loc{\mathrm{loc}}

\def\cut{\mathrm{cut}}

\def\sh{\mathrm{sh}}

\def\sgn{\mathrm{sgn}}

\def\esssup{\mathrm{ess} \ \mathrm{sup}}

\def\T{\mathsf{T}}

\def\Ric{\mathrm{Ric}}

\def\R{{\mathbb R}}
\def\Z{{\mathbb Z}}
\def\C{{\mathbb C}}
\def\N{{\mathbb N}}

\def\K{{\mathbb K}}
\def\CP{{\mathbb C}{\mathbb P}}
\def\D{{\mathbb D}}

\def\L{{\cal L}}
\def\O{{\cal O}}

\def\inum{{\sqrt{-1}}}

\providecommand{\keywords}[1]
{{
\small 
\textbf{{Keywords---}} #1
}}
\providecommand{\MSC}[1]
{{
\small 
\textbf{{MSC---}} #1
}}

\begin{document}

\title {Complete solutions and subharmonic functions}
\author {Natsuo Miyatake}
\date{}
\maketitle

\begin{abstract}
Let $X$ be a Riemann surface, $K_X \rightarrow X$ the canonical bundle, and $T_X\coloneqq K_X^{-1}\rightarrow X$ the dual bundle of the canonical bundle. For each integer $r \geq 2$, each $q \in H^0(K_X^r)$, and each choice of the square root $K_X^{1/2}$ of the canonical bundle, we canonically obtain a Higgs bundle $(\K_r,\Phi(q))\rightarrow X$, which is called a cyclic Higgs bundle. A diagonal solution $h = (h_1, \dots, h_r)$ to the Hitchin equation on a cyclic Higgs bundle yields $r-1$ Hermitian metrics $H_1, \dots, H_{r-1}$ on $T_X\rightarrow X$, defined as $H_j \coloneqq h_j^{-1} \otimes h_{j+1}$ for each $j=1,\dots, r-1$, while $h_1$, $h_r$, and $q$ yield a Hermitian metric $H_r$ on $T_X \rightarrow X$ which is degenerate at the zero set of $q$. A diagonal solution is said to be complete if the K\"ahler metrics induced by $H_1,\dots, H_{r-1}$ are all complete. Li--Mochizuki established a theorem stating that, for any Riemann surface $X$ and any $q$ that is non-zero unless $X$ is hyperbolic, there exists a unique complete solution $h$ to the Hitchin equation for the cyclic Higgs bundle $(\K_r,\Phi(q))$ with $\det(h)=1$. The holomorphic section $q$ induces a subharmonic weight function $\phi_q=\frac{1}{r}\log|q|^2$ on $K_X\rightarrow X$, and the Hitchin equation depends on $q$ only through this weight function. In previous work, the author introduced a variant of the Hitchin equation associated with an arbitrary subharmonic weight function $\varphi$. In this paper, we extend the Li--Mochizuki theorem to complete solutions of the extended Hitchin equation associated to any subharmonic weight function $\varphi$ that is not identically $-\infty$ unless $X$ is hyperbolic. Furthermore, when $X$ is either the complex plane or the unit disc and $(\varphi_\epsilon)_{\epsilon>0}$ is a family of smooth subharmonic weight functions decreasing monotonically to $\varphi$ as $\epsilon\searrow0$, with each $\varphi_\epsilon$ defined on the disc $\D_\epsilon:=\{z\in\C\mid |z|<1-\epsilon\}$ in the unit disc case, we show that the corresponding complete solutions $h_\epsilon$, normalized by $\det(h_\epsilon)=1$, converge monotonically to the complete solution $h$ associated with $\varphi$ satisfying $\det(h)=1$. 
\end{abstract}

\MSC{35J47, 31A05, 53C07}


\section{Introduction and main theorems}\label{1}
Li--Mochizuki \cite{LM1} established a theorem stating that for any $r$-differential $q \in H^0(K_X^r)$ with $r\geq 2$, which is non-zero unless $X$ is hyperbolic, there exists a unique complete harmonic metric with determinant 1 on the cyclic Higgs bundle associated with $q$.
The purpose of this paper is to extend this theorem to complete solutions of a variant of the Hitchin equation associated with more general subharmonic weight functions, and to establish an approximation theorem for them. The objective of this section is to give precise statements of our main theorems. Before presenting the main theorems, we explain the background of cyclic Higgs bundles, harmonic metrics, and their extensions using subharmonic weight functions.

Let $X$ be a Riemann surface, $K_X \to X$ its canonical bundle, and $T_X \coloneqq K_X^{-1} \to X$ the holomorphic tangent bundle. For each integer $r \ge 2$ and a choice of a square root $K_X^{1/2}$ of the canonical bundle, Hitchin \cite{Hit1, Hit2} introduced a holomorphic vector bundle $\K_r$ of rank $r$, defined by
\begin{align*}
\K_r \coloneqq \bigoplus_{j=1}^r K_X^{\frac{r-(2j-1)}{2}}
= K_X^{\frac{r-1}{2}} \oplus K_X^{\frac{r-3}{2}} \oplus \cdots \oplus K_X^{-\frac{r-3}{2}} \oplus K_X^{-\frac{r-1}{2}} .
\end{align*}
The definition of $\K_r$ arises from the Weyl vector of type $A_{r-1}$.  
For each $q \in H^0(K_X^r)$, one associates a holomorphic section $\Phi(q) \in H^0(\End \K_r \otimes K_X)$, called a Higgs field, given by
\begin{align*}
\Phi(q) \coloneqq 
\begin{pmatrix}
0 &        &        & q \\ 
1 & \ddots &        &   \\ 
  & \ddots & \ddots &   \\ 
  &        & 1      & 0
\end{pmatrix}.
\end{align*}
The pair $(\K_r,\Phi(q)) \to X$ is called a cyclic Higgs bundle (cf. \cite{Bar1, Hit1, Hit2}). A diagonal Hermitian metric $h=(h_1,\dots,h_r)$ on $\K_r$ induces Hermitian metrics $H_1,\dots,H_{r-1}$ on $T_X \to X$, defined by $H_j \coloneqq h_j^{-1}\otimes h_{j+1}$ for $j=1,\dots,r-1$. Moreover, $h_1$, $h_r$, and a holomorphic $r$-differential $q$ together determine a degenerate Hermitian metric $H_r$ on $T_X \to X$, defined by $H_r\coloneqq h_r^{-1}\otimes h_1\otimes h_q$, where $h_q$ is a degenerate Hermitian metric on $K_X^{-r}$ given by
\begin{align*}
h_q(u,v) \coloneqq \langle q_x, u\rangle\, \overline{\langle q_x, v\rangle}
\quad \text{for } x\in X, \ u,v\in (K_X^{-r})_x.
\end{align*}
A diagonal Hermitian metric is said to be \emph{complete} if the K\"ahler metrics induced by $H_1,\dots,H_{r-1}$ are all complete. A Hermitian metric on $(\K_r,\Phi(q))\rightarrow X$ is said to be a {\it harmonic metric} if it solves the Hitchin equation. Li--Mochizuki \cite{LM1} proved that for any Riemann surface $X$ and any holomorphic section $q$ that is non-zero unless $X$ is hyperbolic, there exists a unique complete harmonic metric $h$ on $(\K_r,\Phi(q)) \to X$ with $\det(h)=1$.  

There are two extreme cases in which the Hitchin equation admits explicit solutions: the case where the holomorphic section $q$ has no zeros, and the case where $q$ is identically zero. If $q$ has no zeros, then the diagonal metric $h$ constructed from $H_1=\cdots=H_{r-1}=h_q^{1/r}$
by the inverse change of variables is a harmonic metric. In particular, if the metric $h_q^{1/r}$ is complete, then the harmonic metric $h$ is also complete. By contrast, suppose that $q$ is identically zero and that there exists a K\"ahler metric of Gaussian curvature $-2$ on $X$. Let $H$ be the Hermitian metric on $K_X^{-1}\to X$ induced by this K\"ahler metric. Then the diagonal Hermitian metric $h$ on $\K_r$ constructed from the Hermitian metrics $H_j \coloneqq j(r-j)H\ (j=1,\dots,r-1)$
by the inverse change of variables is a harmonic metric, and it is complete whenever the K\"ahler metric is complete. From this perspective, complete harmonic metrics on cyclic Higgs bundles may be viewed as interpolating between complete flat K\"ahler metrics and complete hyperbolic metrics. The zeros of $q$ bend a flat metric, while the metric approaches the hyperbolic one in the limiting case $q=0$.

The section $q$ induces a semipositive singular Hermitian metric $h_q^{-1/r}=e^{-\frac{1}{r}\log|q|^2_{h_\ast}}h_\ast$ on $K_X$, and the associated diagonal harmonic metric depends only on this metric. When we consider the harmonic metric as something intermediate between a flat K\"ahler metric and a K\"ahler metric with negative constant curvature, rather than as something that produces a harmonic map, and as an object of study in complex analysis, it is therefore natural to extend the theory to a general semipositive singular Hermitian metric $e^{-\varphi}h_\refe$, moving beyond the classical setting of holomorphic differentials.  Since subharmonic functions can be approximated by logarithms of absolute values of holomorphic functions, this extension may be interpreted as studying Higgs fields with cyclic multivalency and the associated harmonic metrics on ramified coverings, as well as limits of these metrics as the covering degrees tend to infinity.

The main significance of extending the theory of complete solutions to general semipositive singular Hermitian metrics lies in the construction of {\it entropy} and {\it free energy}, two functions introduced in \cite{Miy4, Miy5}, and in their application to the approximation theory of {\it equilibrium metrics}. The metrics arising in this approximation theory need not be induced by holomorphic differentials, and their weight functions need not be smooth. Studying entropy and free energy in this setting therefore calls for a theory of complete solutions associated with general subharmonic weight functions. 
Our main theorems provide an analytic foundation for studying approximation of equilibrium metrics through entropy and free energy. We briefly recall the definitions of these functions in Section \ref{S and F}; their applications to approximation of equilibrium metrics will be developed in \cite{Miy8}.

In this paper, we study the existence and uniqueness of complete solutions associated with more general subharmonic weight functions. Throughout this paper, when we say that ``$h$ is a complete solution associated with $\varphi$," we mean that $h$ is a solution to the extended Hitchin equation \eqref{phi} in Section \ref{2.2} associated with $\varphi$, and that it is complete in the sense defined in Section \ref{rdc}. Our main theorems are as follows:

\begin{theo}\label{main theorem 1}
{\it
Let $h=(h_1,\dots,h_r)$ and $h^\prime=(h_1^\prime,\dots,h_r^\prime)$ be complete solutions associated with $\varphi$ satisfying $\det(h)=\det(h^\prime)=1$.  
Then $h=h^\prime$. 
}
\end{theo}

\begin{theo}\label{main theorem 2}
{\it Let $\varphi$ be a subharmonic weight function on $X$ that is not identically $-\infty$ unless $X$ is hyperbolic. Then there exists a real complete solution $h=(h_1,\dots,h_r)$ associated with $\varphi$. Moreover, if $X$ is the complex plane $\C$, $h$ satisfies the following comparison property at infinity:
\begin{itemize}
\item For each $R>0$, let $\D_R$ be the open ball of radius $R$ and $h^{>R}$ the unique complete solution associated with $\varphi|_{\C\setminus\overline{\D}_R}$ satisfying $\det(h^{>R})=1$. Then for any $R_1>R_2>0$, $h$ is mutually bounded with $h^{>R_2}$ on $\C\setminus \D_{R_1}$.
\end{itemize}
}
\end{theo}

\begin{theo}\label{main theorem 3} {\it Let $X$ be either the complex plane $\C$ or the unit disc $\D$, and let $\varphi$ be a subharmonic weight function on $K_X$. Assume that $\varphi\not\equiv-\infty$ if $X=\C$. Let $(\varphi_\epsilon)_{0<\epsilon<1}$ be a family of smooth subharmonic weight functions, where each $\varphi_\epsilon$ is defined on $\C$ in the complex plane case and on the disc $\D_\epsilon\coloneqq \{z\in\C\mid |z|<1-\epsilon\}$ in the unit disc case. Suppose that $\varphi_{\epsilon^\prime}\leq\varphi_\epsilon$ on the domain of $\varphi_\epsilon$ whenever $0<\epsilon^\prime<\epsilon<1$, and that $\varphi_\epsilon$ converges pointwise to $\varphi$ as $\epsilon\searrow0$. Then the corresponding complete solutions $h_\epsilon=(h_{\epsilon,1},\ldots,h_{\epsilon,r})$, normalized by $\det(h_\epsilon)=1$, converge monotonically, as $\epsilon\searrow0$, to the unique complete solution $h=(h_1,\ldots,h_r)$ associated with $\varphi$ and satisfying $\det(h)=1$.
}
\end{theo}

The proof of Theorems \ref{main theorem 1} and \ref{main theorem 2} follows the general strategy of \cite{LM1}; however, because issues regarding the regularity of $e^\varphi$ prevent the arguments in \cite{LM1} from being applied directly, some adjustments and modifications are required. In particular, since the Omori--Yau maximum principle and the Cheng--Yau maximum principle, which play a crucial role in \cite{LM1}, cannot be applied in our setting, we provide alternative tools for this purpose in Section \ref{prelim max}. Also, the existence of a complete solution (Theorem \ref{main theorem 2}) is not proven directly for the general case; instead, it is first established for the case where $e^\varphi$ possesses good regularity, and then extended to the more general case using Theorem \ref{main theorem 3}.

Approximating families satisfying the assumptions of Theorem \ref{main theorem 3} can always be constructed by the standard theory of mollifications (cf. \cite{GZ1}). The proof of Theorem \ref{main theorem 3} draws on ideas from \cite[Proposition 4.6]{LM1}. In addition to its role in the proof of Theorem \ref{main theorem 2}, this approximation theorem is used in \cite{Miy4, Miy5} to establish results on entropy and free energy.

Standard local regularity theory for the Poisson equation (cf. \cite{Miy3}) implies that the solution $h$ in Theorem \ref{main theorem 2} is of class $C^{1,\alpha}_\loc$ for every $0<\alpha<1$. In fact, the regularity of $h$ can be further improved to $C^{1,1}_\loc$, and it is of class $C^2$ if and only if $e^\varphi$ is continuous. These regularity properties are discussed in Appendix \ref{appA}.

This paper is structured as follows: Section \ref{2} provides background. Section \ref{3} is devoted to the proof of the main theorems. 

\noindent
{\bf AI Declaration} This paper combines two preprints, arXiv:2508.12848 and arXiv:2610.23441v1, and includes a newly added appendix. In preparing arXiv:2508.12848, I developed all the mathematical arguments myself, without any AI involvement. Conversations with ChatGPT were used solely to improve the English. In preparing arXiv:2610.23441v1, I also discussed mathematical arguments with ChatGPT (OpenAI; GPT-5.6 Sol, GPT-6.1 Sol, and GPT-6 Astra models, accessed through a ChatGPT Plus subscription). The Appendix was prepared with AI assistance, but it does not affect the substance of the main theorems. As can be seen from arXiv:2610.23441v1, every use of the propositions in the Appendix within the main text can be replaced by simpler arguments based on standard results for the Poisson equation. Below, I describe in detail how AI was used in preparing arXiv:2610.23441v1, using the numbering of the present paper. I divide the aspects of the work in which AI assistance was used into the following four categories. I note at the outset that only the contribution described in item 1 is essential to the proofs of the main theorems. In items 2 and 3, AI played an essential role in developing the respective propositions. However, by weakening the statements of these propositions or instead invoking results used in Li--Mochizuki's paper, one can complete the proofs of the main theorems without the parts of these arguments to which AI contributed. My reason for seeking AI assistance in items 2 and 3 was not that it was indispensable to proving the main theorems, but that I wished to formulate these propositions in a more complete form so that I could cite them in subsequent work. In item 4, I developed the underlying argument myself, and AI identified gaps in it during a subsequent check. I now describe items 1--4 in detail. 
\begin{enumerate}
\item I used AI assistance in proving Proposition \ref{SL}. The specific assistance received and the history leading to the completed proof are as follows. In \cite{LM1}, Li--Mochizuki proved uniqueness by combining the Cheng--Yau and Omori--Yau maximum principles on open Riemann surfaces with a priori estimates derived from Simpson's inequality. For the equation considered here, however, the versions of these maximum principles used in \cite{LM1} require more regularity than is available for the relevant functions when establishing uniqueness in full generality. I therefore introduced Propositions \ref{wstoc} and \ref{SL} as substitutes for the Omori--Yau and Cheng--Yau maximum principles, respectively. In arXiv:2508.12848, I had used extensions of these principles to manifolds with boundary, but the resulting uniqueness theorem required additional assumptions. I proved Proposition \ref{wstoc} without AI assistance, whereas the proof of Proposition \ref{SL} involved AI assistance.

The starting point for Proposition \ref{SL} was \cite[Theorem 8.20]{BMPR1}, and my basic strategy was to extend that result. The initial version of Proposition \ref{SL}, which I proved without AI assistance, asserted, roughly speaking, that an upper semicontinuous function $u$ is bounded above if there exists a continuous auxiliary function $g$ to which $u$ is close in a suitable sense. 

Consequently, the initial version of Theorem 1 established uniqueness under the assumption that $e^\varphi h_\ast^{-1}\otimes H_j^{-1}$ is bounded for every $j=1,\ldots,r-1$. I also proved that this boundedness condition holds whenever $e^\varphi h_\ast^{-1}$ is continuous outside a compact set, so uniqueness follows in that case without any further assumptions.

I had obtained all of these results without AI assistance. At the time, I believed that extending \cite[Theorem 8.20]{BMPR1} to weak solutions would necessarily require some analogue of continuity to ensure that the relevant strict superlevel set is open at the final step of the contradiction argument. I therefore thought that the result described above might be the limit of what could be achieved by extending that theorem, and I prepared the initial draft with this weaker uniqueness statement. While I was preparing that draft, however, the GPT-6.1 Sol model in ChatGPT suggested the stronger formulation now stated as Proposition \ref{SL}. This enabled me to establish Theorem 1 in full generality.
\item In proving Lemma \ref{uvFG}, which is used in the proof of Proposition \ref{min-max}, I received repeated assistance from ChatGPT in correcting gaps in my extension of the argument in \cite[Proof of Lemma 1.5]{DS1}. ChatGPT also suggested introducing the truncated function $G_j^\#$ in the proof of Proposition \ref{min-max}. I had not initially recognized that the issue described in Remark \ref{comparg} could arise. ChatGPT pointed out the resulting gap in my initial draft and proposed this truncation argument to address it. 	However, this truncation argument is unnecessary if, for example, the supersolutions and subsolutions in Proposition \ref{min-max} are all continuous, and this continuity assumption suffices for the applications of Proposition \ref{min-max} in the proofs of the main theorems.

\item In the proof of Proposition \ref{coffee}, I initially intended to apply Lemma \ref{cands} on $Y_i$ to show that the solution of the Dirichlet problem lies between the subsolution and the supersolution. ChatGPT pointed out that the lower Ricci curvature bound required by Lemma \ref{cands} need not hold even though $\overline{Y_i}$ is compact, since the Ricci curvature in the weak sense need not be represented by a continuous function. ChatGPT also explained how the proposition can still be proved by modifying the comparison argument as illustrated in Remark \ref{remark50}. Note, however, that for the applications needed in the proofs of the main theorems, this proposition can be replaced by an application of \cite[Proposition 5.2]{LM1} in the case where $e^\varphi$ is smooth.

\item ChatGPT helped me correct the following minor errors in my initial draft. It brought Ambrose's theorem, discussed immediately before Proposition \ref{pistarphi}, to my attention as a way to fill a gap in the original argument. The constant $M_\varphi$ introduced at the beginning of Section \ref{udcase} was originally defined as $\sup_X e^\varphi$, but ChatGPT pointed out that this choice causes the argument to fail when $\varphi\equiv-\infty$. Finally, the constant $c$ in Lemma \ref{le2000} was originally defined by an equality. ChatGPT pointed out that, for the argument to work, $c$ needed to be chosen strictly larger than the value specified in the initial draft.
\end{enumerate}
\noindent
{\bf Acknowledgements.} I would like to thank Toshiaki Yachimura for helpful conversations. This work was supported by JSPS KAKENHI Grant Number JP24K16912.

\section{Background}\label{2}
\subsection{Real, diagonal, and complete Hermitian metrics}\label{rdc}
Let $X$ be a connected, possibly non-compact Riemann surface and $K_X\rightarrow X$ the canonical bundle. We choose a square root $K_X^{1/2}$ of the canonical bundle $K_X\rightarrow X$. Let $\K_r$ be a holomorphic vector bundle of rank $r$ defined as follows:
\begin{align*}
\K_r\coloneqq \bigoplus_{j=1}^r K_X^{\frac{r-(2j-1)}{2}} = K_X^{\frac{r-1}{2}} \oplus K_X^{\frac{r-3}{2}} \oplus \cdots \oplus K_X^{-\frac{r-3}{2}} \oplus K_X^{-\frac{r-1}{2}}.
\end{align*}
A Hermitian metric $h$ on $\K_r\rightarrow X$ is said to be {\it real} (cf. \cite{Hit2}) if the following bundle isomorphism $S$ is isometric with respect to $h$:
\begin{align*}
S\coloneqq \left(
\begin{array}{ccc}
& & 1 \\ 
& \reflectbox{$\ddots$} & \\ 
1 & &
\end{array}
\right): \K_r\rightarrow \K_r^{\vee},
\end{align*}
where $\K_r^{\vee}$ denotes the dual bundle of $\K_r$. Also, $h$ is called {\it diagonal} (cf. \cite{Bar1, Col1, DL2}) if it splits as $h=(h_1,\dots, h_r)$. Let $T_X\rightarrow X$ be the dual bundle of $K_X\rightarrow X$. By taking the difference between the adjacent components in a diagonal Hermitian metric $h=(h_1,\dots, h_r)$, we obtain $r-1$ Hermitian metrics $H_1,\dots, H_{r-1}$ on $T_X\rightarrow X$:
\begin{align*}
H_j\coloneqq h_j^{-1}\otimes h_{j+1}.
\end{align*}
A diagonal Hermitian metric $h$ is real if and only if $h_j=h_{r-j+1}^{-1}$ for all $j=1,\dots, r$. If $\det(h)=1$, this is also equivalent to the metrics $H_1,\dots, H_{r-1}$ satisfying $H_j=H_{r-j}$ for all $j=1,\dots, r-1$. We say that a diagonal Hermitian metric $h$ is {\it complete} (cf. \cite{LM1}) if the K\"ahler metrics induced by $H_1,\dots, H_{r-1}$ are all complete. 

\begin{rem} Note that the notion of completeness introduced above is considerably broader than the original notion in \cite{LM1}. In \cite{LM1}, a reference K\"ahler metric is fixed, thereby allowing diagonal Hermitian metrics on $\K_r$ to be identified with $r$-tuples of functions. The term ``complete solution'' is then used for a diagonal Hermitian metric $h$ on $\K_r$, viewed as an $r$-tuple of functions, that satisfies $\det(h)=1$ and solves the Hitchin equation for a cyclic Higgs bundle. In contrast, we use the term ``complete'' for arbitrary diagonal Hermitian metrics, without requiring them to be harmonic or imposing any condition on their determinants, although the main theorem and its proof deal exclusively with metrics satisfying $\det(h)=1$.

\end{rem} 
\begin{rem}
In this paper, we allow K\"ahler metrics on Riemann surfaces and Hermitian metrics on $\K_r$ to be continuous but not necessarily smooth. 
Whenever we say that a (possibly non-smooth) K\"ahler metric is complete, we mean that the induced distance function defines a complete metric space.
\end{rem}

\begin{rem}
Throughout this paper, unless otherwise specified, it is implicitly assumed that Riemann surfaces and manifolds have no boundary. Manifolds and Riemann surfaces with boundary will also be considered later, but in such cases the presence of boundary will always be stated explicitly.
\end{rem}

\subsection{Harmonic metrics on cyclic Higgs bundles}\label{2.1}
Let $q$ be a holomorphic section of $K_X^r\rightarrow X$. Then we associate a holomorphic section $\Phi(q)\in H^0(\End \K_r\otimes K_X)$, which is called a Higgs field, as follows:
\begin{align*}
\Phi(q)\coloneqq 
\left(
\begin{array}{cccc}
0 & && q \\ 
1 & \ddots && \\ 
& \ddots & \ddots & \\ 
& & 1 & 0
\end{array}
\right).
\end{align*}
The pair $(\K_r,\Phi(q))$ is called a cyclic Higgs bundle (cf.~\cite{Bar1, Hit1, Hit2}). The above cyclic Higgs bundle is an example of cyclotomic Higgs bundles introduced in \cite{Sim3}. For each Hermitian metric $h$ on $\K_r$, we associate a connection $D_q(h)$ defined as follows:
\begin{align*}
D_q(h)\coloneqq \nabla^h + \Phi(q) + \Phi(q)^{\ast h}.
\end{align*}
A Hermitian metric on $(\K_r,\Phi(q))\rightarrow X$ is called a {\it harmonic metric} if the connection $D_q(h)$ is a flat connection. A harmonic metric $h$ yields a harmonic map $\hat{h}: \widetilde{X}\rightarrow \SL(r,\C)/\SU(r)$, where $\widetilde{X}$ is the universal covering space. If $h$ is real, then $\hat{h}$ maps into $\SL(r,\R)/\SO(r)$. Suppose that $X$ is a compact Riemann surface of genus at least $2$. Then the Higgs bundle $(\K_r,\Phi(q))$ is stable for any $q\in H^0(K_X^r)$, and thus there uniquely exists a harmonic metric $h$ on $(\K_r,\Phi(q))$ such that $\det(h)=1$. From the uniqueness of the metric and the symmetry of the Higgs field (cf.~\cite{Bar1, Hit2}), the harmonic metric $h$ is diagonal and real. To construct a non-diagonal harmonic metric (resp. a non-real harmonic metric), one may, for example, use a solution to the Dirichlet problem for the Hitchin equations with a non-diagonal (resp. non-real) Hermitian metric prescribed on the boundary (cf. \cite{Don1, LM2}).
Let $h=(h_1,\dots, h_r)$ be a diagonal harmonic metric on a cyclic Higgs bundle $(\K_r,\Phi(q))\rightarrow X$. As mentioned in the previous subsection we obtain $r-1$ Hermitian metrics $H_1,\dots, H_{r-1}$ on $T_X\rightarrow X$ defined as
\begin{align*}
H_j\coloneqq h_j^{-1}\otimes h_{j+1} \ \text{for each $j=1,\dots,r -1$}.
\end{align*}
We also obtain a Hermitian metric $H_r$ on $T_X\rightarrow X$ which is degenerate at the zeros of $q$ as follows:
\begin{align*}
H_r\coloneqq h_r^{-1}\otimes h_1 \otimes h_q,
\end{align*}
where $h_q$ is a Hermitian metric on $K_X^{-r}\rightarrow X$ that is degenerate at the zeros of $q$ defined as follows:
\begin{align*}
(h_q)_x(u,v)\coloneqq \langle q_x,u\rangle\overline{\langle q_x,v\rangle} \ \text{for $x\in X$, $u,v\in (K_X^{-r})_x$}. 
\end{align*}
The connection $D_q(h)$ is flat if and only if the Hermitian metric $h$ satisfies the following elliptic equation, called the Hitchin equation:
\begin{align*}
F_h + [\Phi(q)\wedge\Phi(q)^{\ast h}] = 0.
\end{align*}
By using $H_1,\dots, H_r$, we can describe the Hitchin equation as follows: 
\begin{align}
\inum F_{h_j} + \vol(H_{j-1}) - \vol(H_j) = 0 \ \text{for $j=1,\dots, r-1$}, \label{FToda}
\end{align}
where for each $j=1,\dots, r$, $\vol(H_j)$ is the volume form of the metric $H_j$ which is locally described as $\vol(H_j)=\inum H_j(\frac{\partial}{\partial z}, \frac{\partial}{\partial z})\ dz\wedge d\bar{z}$, and $H_0$ is $H_r$. Equation (\ref{FToda}) is a kind of Toda equation (cf.~ \cite{AF1, Bar1, GH1, GL1, Moc0, Moc1}). By making a variable change from $h=(h_1,\dots, h_r)$ to $H_1,\dots, H_{r-1}$, the Hitchin equation becomes the following:
\begin{align}
\inum F_{H_j} + 2\vol(H_j) - \vol(H_{j-1}) - \vol(H_{j+1}) = 0 \ \text{for $j=1,\dots, r-1$.} \label{cyc}
\end{align}
In the case where $r=2$ and $q=0$, the Hitchin equation for the cyclic Higgs bundle is equivalent to the following equation:
\begin{align}
\inum F_H + 2\vol(H) = 0, \label{H}
\end{align}
where the solution $H$ to equation (\ref{H}) is a Hermitian metric $H$ on $T_X\rightarrow X$. The K\"ahler metric induced by $H$ is a constant negative Gaussian curvature K\"ahler metric (cf.~\cite{Hit1}). 
We remark on the following two points:
\begin{itemize}
\item If $q$ has no zeros, then $H_1=\cdots=H_{r-1}=h_q^{1/r}$ is a solution to equation (\ref{cyc}).
\item Suppose that $q=0$ and that there is a solution $H$ to equation (\ref{H}). We define positive constants $\lambda_1,\dots, \lambda_{r-1}$ as follows: 
\begin{align}
\lambda_j\coloneqq 2\sum_{k=1}^{r-1}(\Lambda_{r-1}^{-1})_{jk} = j(r-j), \label{lambda}
\end{align}
where $\Lambda_{r-1}^{-1}$ is the inverse matrix of the Cartan matrix of type $A_{r-1}$. Then $H_j=\lambda_j H\ (j=1,\dots, r-1)$ is a solution to equation (\ref{cyc}).
\end{itemize}

Li--Mochizuki \cite{LM1} established the following result:
\begin{theo}[\cite{LM1}]{\it Suppose that $q$ is a non-zero holomorphic section unless $X$ is hyperbolic. Then there always exists a unique complete harmonic metric $h$ on $(\K_r,\Phi(q))\rightarrow X$ satisfying $\det(h)=1$. Moreover, the complete harmonic metric $h$ is real.
}
\end{theo}
In addition to \cite{LM1}, we refer the reader to~\cite{DL3, DW1, GL1, LT1, LTW1, Li1, LM2, Moc0, Moc1, Nie1, Wan1, WA1} for the works related to complete harmonic metrics. 

\subsection{More general subharmonic weight functions}\label{2.2}
For each $q\in H^0(K_X^r)$, the metric $h_q$ induces a singular metric (cf. \cite{GZ1}) $h_q^{-1/r}$ that diverges at the zeros of $q$. We can define a curvature for $h_q^{-1/r}$ in the sense of currents, which is semipositive and has support at the zero points of $q$. We fix a smooth metric $h_\refe$ and denote by $e^{-\phi_q}h_\refe$ the metric $h_q^{-1/r}$, where the weight function $\phi_q$ is defined as $\phi_q\coloneqq \frac{1}{r}\log|q|_{h_\refe}^2$. If we choose a local flat reference metric, then the weight function defines a local subharmonic function. Also, although we will not go into details here, there is a way to identify the weight function with a global psh function on the total space of the dual line bundle excluding zero points (see \cite[Section 2.1]{BB1}). In the sense described above, we call the function $\phi_q$ a {\it subharmonic weight function}, although this is a bit of an abuse of terminology. We consider a more general singular metric $e^{-\varphi}h_\refe$ with semipositive curvature, where the weight function $\varphi$ is a function that is locally a sum of a subharmonic function and a smooth function. For each $q_N\in H^0(K_X^N)$, let $\phi_{q_N}$ be defined as $\frac{1}{N}\log|q_N|_{h_\refe}^2$. Then $e^{-\phi_{q_N}}h_\refe$ is a singular metric with semipositive curvature. Moreover, any subharmonic weight function can be approximated, at least in the $L^1_{\loc}$ sense, by a sequence $(\phi_{q_N})_{N\in\N}$, where $q_N\in H^0(K_X^N)$ (cf. \cite{GZ1, Hor1}). In \cite{Miy2, Miy3}, for each semipositive singular metric $e^{-\varphi}h_\refe$ and each $r\geq 2$, the following equation for a diagonal metric $h=(h_1,\dots, h_r)$ on $\K_r\rightarrow X$ was introduced:
\begin{align}
\inum F_{h_j} + \vol(H_{j-1}) - \vol(H_j) = 0 \ \text{for $j=1,\dots, r-1$}, \label{phi}
\end{align}
where $H_1,\dots, H_{r-1}$ are defined as $H_j\coloneqq h_j^{-1}\otimes h_{j+1}$ for each $j=1,\dots, r-1$, and $H_r=H_0$ is defined as follows:
\begin{align*}
H_r = H_0 &\coloneqq h_r^{-1}\otimes h_1 \otimes (e^{-\varphi}h_\refe)^{-r} \\ 
&= H_1^{-1}\otimes \cdots \otimes H_{r-1}^{-1} \otimes (e^{-\varphi}h_\refe)^{-r}.
\end{align*}
For $\varphi=\phi_{q_N}$, a solution of equation (\ref{phi}) pulls back to a harmonic metric on the bundle $\pi^\ast\K_r$ over a suitable cyclic branched covering $\pi:Y\to X$ on which $q_N^{1/N}$ is single-valued. The corresponding Higgs field is defined using this root and the natural map $\pi^\ast K_X\to K_Y$ (see \cite[Section 2]{Miy3}). As mentioned above, more general $\varphi$ can be approximated by a sequence $(\phi_{q_N})_{N\in\N}$ at least in the $L^1_{\loc}$-sense (cf. \cite{GZ1, Hor1}); therefore, equation (\ref{phi}) can be considered to be an equation obtained as the limit when the covering degree increases infinitely or the number of zeros increases infinitely.  By making a variable change from $h=(h_1,\dots, h_r)$ to $H_1,\dots, H_{r-1}$, the equation becomes the following:
\begin{align}
\inum F_{H_j} + 2\vol(H_j) - \vol(H_{j-1}) - \vol(H_{j+1}) = 0 \ \text{for $j=1,\dots, r-1$.}\label{cyc2}
\end{align}
We remark on the following two points:
\begin{itemize}
\item If $e^{-\varphi}h_\refe$ is flat, then $H_1=\cdots=H_{r-1}=(e^{-\varphi}h_\refe)^{-1}$ is a solution to equation (\ref{cyc2}).
\item Suppose that $\varphi=-\infty$. Then equation (\ref{cyc2}) is the same as equation (\ref{cyc}) for the case where $q=0$, and thus for a solution $H$ to equation (\ref{H}), $H_j=\lambda_j H\ (j=1,\dots, r-1)$ is a solution to equation (\ref{cyc2}), where $\lambda_1,\dots, \lambda_{r-1}$ are constants defined in (\ref{lambda}).
\end{itemize}

\subsection{Entropy and free energy}\label{S and F}
This subsection reviews the definitions of the functions called entropy and free energy given in \cite{Miy4, Miy5}. Let $h=(h_1,\dots, h_r)$ be a solution to equation (\ref{phi}) and let $H_1,\dots, H_r$ be the Hermitian metrics constructed from $h$ and $e^{-\varphi}h_\refe$ as in Section \ref{2.2}. Then we define functions, which we call entropy and free energy, as follows:

\begin{defi}\label{entropy2} 
For each $j=0,1,\dots,r-1$ and a real number $\beta$, 
we define a nonnegative function
$p_j(r,\beta,\varphi):X\to[0,1]$ as follows.
For each $x\in X\setminus\varphi^{-1}(-\infty)$, 
we define
\begin{align*}
p_j(r,\beta,\varphi)(x)\coloneqq
\frac{\vol(H_j)^\beta}{\sum_{j=0}^{r-1}\vol(H_j)^\beta}(x) \ \text{for each $j=0,\dots, r-1$,}
\end{align*}
where $\vol(H_j)^\beta/\sum_{j=0}^{r-1}\vol(H_j)^\beta$ is understood to be
$(\vol(H_j)/\vol(H_\refe))^\beta/\sum_{j=0}^{r-1}(\vol(H_j)/\vol(H_\refe))^\beta$
for some reference metric $H_\refe$, which does not affect
$p_j(r,\beta,\varphi)$. For each $x\in \varphi^{-1}(-\infty)$, we set
$p_0(r,\beta,\varphi)(x)=0$ and define
\begin{align*}
p_j(r,\beta,\varphi)(x)\coloneqq
\frac{\vol(H_j)^\beta}{\sum_{j=1}^{r-1}\vol(H_j)^\beta}(x) \ \text{for each $j=1,\dots, r-1$.}
\end{align*}
We define the {\it entropy} by
\begin{align*}
S(r,\beta,\varphi)\coloneqq
-\sum_{j=0}^{r-1} p_j(r,\beta,\varphi)\log p_j(r,\beta,\varphi).
\end{align*}
\end{defi}

\begin{defi}\label{F}
Let $\beta$ be a nonzero real number.
For each $x\in X$, if $(\beta\varphi)(x)$ does not take the value $\infty$, we set
\begin{align*}
F(r,\beta,\varphi,H_\refe)(x)
\coloneqq
-\frac{1}{\beta}
\log(
\sum_{j=0}^{r-1}
(\vol(H_j)/\vol(H_\refe))^\beta)(x),
\end{align*}
where $H_\refe$ is a fixed reference metric.
If $(\beta\varphi)(x)$ takes the value $\infty$, we define
\begin{align*}
F(r,\beta,\varphi,H_\refe)(x)
\coloneqq
-\frac{1}{\beta}
\log (
\sum_{j=1}^{r-1}
(\vol(H_j)/\vol(H_\refe))^\beta)(x).
\end{align*}
We call $F(r,\beta,\varphi,H_\refe)$ the \it{free energy}.
\end{defi}

Note that the difference between free energy functions defined using the same reference metric does not depend on the choice of the reference metric. 
The author expects these functions to provide a new scale for measuring the degree of accuracy of {\it approximations to equilibrium metrics} in various situations; see, e.g.,~\cite{BCHM1, Ber1, Ber2, Ber3, BB1, BBN1,   BL1, BL2, BL3, BL4, Dys1, FLM1, Gin1, GZ1, Ham1, Lju1, MP1, MV1, Ran1, ST1, SZ1}. One of the goals in this direction is to quantitatively establish entropy-increase and free energy-decay principles for processes in which non-equilibrium metrics approximate equilibrium metrics. See~\cite{Miy4, Miy5, Miy8} for further details and for the results obtained so far.

\section{Proof}\label{3}
\subsection{Preliminaries on approximation of non-smooth K\"ahler metrics}
The objective of this subsection is to prove the following proposition:
\begin{prop}\label{GWT}
{\it Let $(X,g_X)$ be a Riemann surface equipped with a continuous K\"ahler metric. Then the following holds:
\begin{enumerate}
\item For each $\epsilon>0$, there exists a smooth K\"ahler metric $\widetilde{g}_{X,\epsilon}$ such that
\begin{align*}
e^{-\epsilon} g_X \le \widetilde{g}_{X,\epsilon} \le e^{\epsilon} g_X.
\end{align*}
\item Suppose moreover that the Ricci curvature of $g_X$ is bounded below by $C$. 
Then for each $\epsilon>0$, there exists a smooth K\"ahler metric $\widetilde{g}_{X,\epsilon}$ such that
\begin{align}
e^{-\epsilon} g_X \le \widetilde{g}_{X,\epsilon} \le e^{\epsilon} g_X, \label{gxe}
\end{align}
and $\widetilde{g}_{X,\epsilon}$ has Ricci curvature bounded below by a constant $C_\epsilon$, where
\begin{align}
C_\epsilon =
\begin{cases}
e^{-\epsilon}C & \text{if } C > 0, \label{Cepsilon}\\
e^{\epsilon}C & \text{if } C < 0, \\
-\epsilon & \text{if } C=0.
\end{cases}
\end{align}
\end{enumerate}
}
\end{prop}
Note that assertion 1 is an immediate consequence of the well-known Whitney approximation theorem (see \cite[Theorem 6.21]{Lee1}). Therefore, the main part of the proof of Proposition \ref{GWT} is the proof of assertion 2, which relies on \cite[p.~66, Corollary 1]{GW1}. 
Before giving the proof, we review this corollary together with \cite[p.~66, Theorem 3.1]{GW1}, from which the corollary is derived. Let $(M,g_M)$ be a Riemannian manifold of real dimension $n$. We denote by $\Delta_{g_M}=-d^\ast d$ the Laplace--Beltrami operator acting on functions, and by $\gamma_p$ the Riemannian distance function from $p\in M$. Let $\O_M$ denote the collection of all open subsets of $M$.

\begin{defi}[\cite{GW1}, p.~49 Definition 1.4]
Let $\L = (\L(U))_{U \in \O_M}$ be a subsheaf of the sheaf of continuous real-valued functions on $M$. 
We say that $\L$ has the {\it semilocal $C^\infty$ approximation property} if the following holds:
Let $U \in \O_M$, let $K \subset U$ be compact, and let $f \in \L(U)$ be such that $f$ is $C^\infty$ on a neighborhood of a compact (possibly empty) subset $K_1 \subset K$. 
Then there exists an open subset $V \subset U$ containing $K$ such that for every $\epsilon > 0$ there exists a $C^\infty$ function $F \in \L(V)$ satisfying
\begin{enumerate}[(a)]
\item $\sup_{p \in K} |f(p) - F(p)| < \epsilon$
\item \label{b} $d_{K_1}(f,F) < \epsilon$
\end{enumerate}
where the definition of $d_{K_1}(\cdot,\cdot)$ can be found in \cite[p.~49]{GW1}.
\end{defi}
\begin{defi}[\cite{GW1}, p.~49]
For each $U\in\O_M$, let $C^0(U)$ be the space of all real-valued continuous functions on $U$. 
The {\it $C^0$ fine topology} on $C^0(U)$ is defined to be the topology generated by the subsets of the form
\begin{align*}
O_{f_1,f_2}\coloneqq \{f\in C^0(U)\mid |f(p)-f_1(p)|<f_2(p) \ \text{for all } p\in U\},
\end{align*}
where $f_1\in C^0(U)$, $f_2\in C^0(U)$, and $f_2$ is a positive function.
\end{defi}
\begin{defi}[\cite{GW1}, p.~65] Let $\eta:M\rightarrow\R$ be a continuous function. A continuous function $f:M\rightarrow \R$ is said to be {\it $\eta$-subharmonic} if for each point $p\in M$ there is a positive number $\delta$ such that the following function $F:M\rightarrow \R$ is subharmonic in a neighborhood of $p$:
\begin{align*}
F(q)\coloneqq f(q)-\left(\delta+\frac{1}{2n}\eta(p)\right)\gamma_p(q)^2 \ \text{for}\ q\in M.
\end{align*}
\end{defi}
We note that the square $\gamma_p^2$ of the Riemannian distance function is smooth in a neighborhood of $p$ and satisfies $\Delta_{g_M}\gamma_p^2|_p = 2n$. 
This can be verified by using normal coordinates at $p$. 
From this and the definition of $\eta$-subharmonicity, the following characterization immediately follows:

\begin{prop}[\cite{GW1}, p.~66 Lemma 3.2] {\it A continuous function $f:M\rightarrow \R$ is $\eta$-subharmonic if and only if there is a $C^\infty$ function $\eta_1$ such that $\Delta_{g_M}f\geq \eta_1$ holds in the weak sense and $\eta_1>\eta$ on $M$.}
\end{prop}
\begin{rem}
In \cite[p.~66 Lemma 3.2]{GW1}, the range of $\eta_1$ is stated to be $\R^+$, which appears to be a typo; it should simply be $\R$.
\end{rem}
We denote by $\L_{\eta\sh}$ the sheaf of (continuous) $\eta$-subharmonic functions. The following holds:
\begin{theo}[\cite{GW1}, p.~66 Theorem 3.1]\label{sla}{\it For any continuous $\eta:M\rightarrow \R$, $\L_{\eta\sh}$ has the semilocal $C^\infty$ approximation property.
}
\end{theo}
Since the sheaf $\L_{\eta \sh}$ has the $C^\infty$ stability property and the maximum-closure property (cf. \cite[Definition 1.1, Definition 1.3, and the arguments immediately below Lemma 3.2]{GW1}), it follows from \cite[Theorem 1.1]{GW1} that the following holds:
\begin{cor}[\cite{GW1}, p.~66 Corollary 1]\label{dense}
{\it For any open set $U \subset M$, smooth $\eta$-subharmonic functions on $U$ are dense in $\L_{\eta \sh}(U)$ with respect to the $C^0$ fine topology.}
\end{cor}
We now prove Proposition \ref{GWT}.
\begin{proof}[Proof of Proposition \ref{GWT}]
We prove assertion 2. It suffices to consider $0<\epsilon\leq1$. Fix such an $\epsilon$. Let $H$ be the continuous Hermitian metric on $K_X^{-1} \to X$ induced by $g_X$. 
By assumption, there exists a constant $C$ such that
\begin{align}
\inum F_H \geq C \vol(H). \label{CH}
\end{align}
Let $H_\ast$ be a smooth reference metric on $K_X^{-1}$ and write $H = e^{-f} H_\ast$ for some continuous function $f$ on $X$. Let $g_\ast$ be the smooth K\"ahler metric induced by $H_\ast$. We apply Corollary \ref{dense} on $(X,2g_\ast)$, for which $\Delta_{2g_\ast}=\inum\Lambda_{H_\ast}\partial\bar\partial$.
It follows from \eqref{CH} that
\begin{align*}
\inum F_{H_\ast} + \inum \partial\bar{\partial} f \geq C e^{-f} \vol(H_\ast).
\end{align*}
Set $\eta \coloneqq C_{\epsilon/2}e^{-f} - \inum \Lambda_{H_\ast} F_{H_\ast}$, where $\Lambda_{H_\ast}$ denotes the adjoint of wedging with $\vol(H_\ast)$ and where $C_\epsilon$ is defined in (\ref{Cepsilon}). By applying Corollary \ref{dense} with domain $X$, for any positive continuous function $f'$, there exists a smooth $\eta$-subharmonic function $u$ such that $|f-u| < f'$ on $X$.  In particular, there exists a smooth $\eta$-subharmonic function $u_\epsilon$ such that $|f-u_\epsilon| < \epsilon/2$ on $X$. Set $\widetilde{H}_\epsilon \coloneqq e^{-u_\epsilon} H_\ast$, and let $\widetilde{g}_{X,\epsilon}$ be the K\"ahler metric induced by $\widetilde{H}_\epsilon$. 
By construction, $\widetilde{g}_{X,\epsilon}$ satisfies \eqref{gxe}. Moreover, the curvature satisfies
\begin{align*}
\inum F_{\widetilde{H}_\epsilon}
&= \inum F_{H_\ast} + \inum \partial\bar{\partial} u_\epsilon \\
&\geq \inum F_{H_\ast} + \eta \vol(H_\ast) \\
&= C_{\epsilon/2} e^{-f} \vol(H_\ast) \\
&\geq C_{\epsilon} e^{-u_\epsilon} \vol(H_\ast)
= C_{\epsilon} \vol(\widetilde{H}_\epsilon).
\end{align*}
This completes the proof.
\end{proof}
\subsection{Preliminaries on maximum principles}\label{prelim max}
This subsection provides some preliminaries on maximum principles. Let $(M,g_M)$ be a connected Riemannian manifold. We denote by $\Delta_{g_M}=-d^\ast d$ the Laplace--Beltrami operator acting on functions. Let $p_{g_M}(x,y,t)$ be the positive minimal heat kernel (cf.~\cite{Gri1, PLi1}). Throughout this section, we denote by $L^1(+\infty)$ the space of functions on $\R$ that belong to $L^1([c,\infty))$ for some $c\in\R$.
We first review the definition of stochastic completeness and several equivalent conditions.
\begin{defi}[cf.~\cite{AMR1}] We say that $(M,g_M)$ is {\it stochastically complete} if the following holds for some (hence any) $(x,t)\in M\times \R_{>0}$:
\begin{align*}
\int_Mp_{g_M}(x,y,t)d\mu_{g_M}(y)=1,
\end{align*}
where we denote by $d\mu_{g_M}$ the volume measure associated with $g_M$.
\end{defi}

\begin{prop}[cf.~\cite{AMR1}]\label{conditions for sc}{\it The following conditions are equivalent:
\begin{enumerate}[(i)]
\item $(M,g_M)$ is stochastically complete.
\item\label{iistoc} For every $\lambda>0$, the only nonnegative bounded $C^2$ function satisfying $\Delta_{g_M}u\geq \lambda u$ on $M$ is $u\equiv 0$.
\item For every $\lambda>0$, the only nonnegative bounded $C^2$ solution of $\Delta_{g_M}u=\lambda u$ on $M$ is $u\equiv 0$.
\item For every $T>0$, the only bounded solution on $M\times (0,T)$ of the Cauchy problem
\begin{align*}
\begin{cases}
\frac{du_t}{dt}=\frac{1}{2}\Delta_{g_M}u_t \\
\lim_{t\to0}u_t=0 \ \text{in the $L^1_{\loc}(M)$ sense}
\end{cases}
\end{align*}
is $u\equiv 0$.
\end{enumerate}
}
\end{prop}

The following characterization of stochastic completeness was established by Pigola--Rigoli--Setti \cite{PRS1}:
\begin{prop}[\cite{PRS1}]\label{stoc}{\it The following are equivalent:
\begin{enumerate}[(i)]
\item\label{stoci} $(M,g_M)$ is stochastically complete.
\item\label{u} For every $C^2$ function $u$ satisfying $u^\ast=\sup_M u < \infty$ and every $\epsilon > 0$, we have
\begin{align*}
\inf_{\Omega_\epsilon}\Delta_{g_M}u\leq 0,
\end{align*}
where $\Omega_\epsilon=\{x\in M\mid u(x)>u^\ast-\epsilon\}$.
\item For every $C^2$ function $u$ with $u^\ast=\sup_M u<\infty$, there exists a sequence of points $(x_k)_{k\geq 1}$ such that
\begin{align*}
&u(x_k)>u^\ast-1/k, \\
&(\Delta_{g_M}u)(x_k)<1/k 
\end{align*}
for each $k\geq 1$.
\item For every $C^2$ function $u$ with $u^\ast=\sup_Mu<\infty$ and every continuous function $f$, if $\Delta_{g_M}u\geq f(u)$ on the subset $\Omega_\epsilon=\{x\in M\mid u(x)>u^\ast-\epsilon\}$ for some $\epsilon>0$, then $f(u^\ast)\leq 0$.
\end{enumerate}
}
\end{prop}

The following is a sufficient condition for a Riemannian manifold $(M,g_M)$ to be stochastically complete:
\begin{prop}[cf.~\cite{AMR1}]\label{rx}
{\it Suppose that $(M,g_M)$ is a complete Riemannian manifold of dimension $m$. 
Let $p\in M$ be a fixed origin and denote by $\gamma_p(x)$ the Riemannian distance function from $p$, and by $\cut(p)$ its cut locus.
Assume that the Ricci curvature satisfies
\begin{align*}
\Ric(\nabla \gamma_p,\nabla \gamma_p)\geq -(m-1)G^2(\gamma_p)
\end{align*}
on $M\setminus (\cut(p)\cup\{p\})$ for some positive nondecreasing continuous function $G$ with
\begin{align*}
\frac{1}{G}\notin L^1(+\infty).
\end{align*}
Then $(M,g_M)$ is stochastically complete.
}
\end{prop}

In particular, it follows from Proposition \ref{rx} that any complete Riemannian manifold with Ricci curvature bounded below is stochastically complete. Indeed, if $\Ric \ge -(m-1)K$ for some constant $K\ge 0$, then we may take $G\equiv \sqrt{K+1}$. For later use, we record the following variant of Proposition \ref{stoc}:
\begin{prop}\label{wstoc}
{\it
Suppose that $(M,g_M)$ is stochastically complete. 
Then for every $u\in C^0(M)\cap W^{1,2}_{\loc}(M)$ with 
$u^\ast=\sup_M u<\infty$, the following cannot occur:
\begin{itemize}
\item
There exist $\epsilon>0$, $c>0$, and a continuous function $g:U\rightarrow \R$ defined on a nonempty open subset $U\subseteq M$ such that
\begin{align}
\Delta_{g_M} u&\ge g \ \text{in the weak sense on $U$}, \label{gMg} \\
g &\ge 2c \ \text{on} \ \Omega_\epsilon, \label{Oeps} \\ 
\overline{\Omega}_\epsilon&\subseteq U, \label{OmegaU}
\end{align}
where $\Omega_\epsilon=\{x\in M \mid u(x)>u^\ast-\epsilon\}$.
\end{itemize}
}
\end{prop}
\begin{proof} The proof is the same as that of the implication (i)$\Rightarrow$(ii) of Proposition \ref{stoc}. Suppose that (\ref{gMg}), (\ref{Oeps}), and (\ref{OmegaU}) hold for some $\epsilon>0$, $c>0$, and a continuous function $g:U\rightarrow \R$. We let $\Omega^\ast=\{x\in U\mid g(x)>c\}$, so that $\overline{\Omega}_\epsilon\subseteq\Omega^\ast$. Setting $\lambda=c/\epsilon$, at each $x\in \Omega^\ast$ we have
\begin{align*}
g(x)>c\geq c+\lambda(u(x)-u^\ast)=\lambda(u(x)+\epsilon-u^\ast),
\end{align*}
so that $u+\epsilon-u^\ast$ is a weak subsolution of
\begin{align}
\Delta_{g_M}u-\lambda u=0 \label{lambdau}
\end{align}
on $\Omega^\ast$. Since the constant function $0$ is obviously a subsolution of (\ref{lambdau}) on $M$, we see that $u_\epsilon\coloneqq \max\{u+\epsilon-u^\ast,0\}$ is also a subsolution on $M$ (see, e.g., \cite[Lemma 1.5]{DS1}). 
Also, since $\max:\R^2\to \R$ is Lipschitz continuous, hence a $W^{1,\infty}_{\loc}$ function (cf. \cite[Theorem 4.5]{EG1}), $u_\epsilon$ is again a continuous $W^{1,2}_{\loc}$ function. Furthermore, $u_\epsilon\not\equiv 0$ and $0\leq u_\epsilon \leq \epsilon<\infty$. Note that any positive constant is a supersolution of (\ref{lambdau}). Choose a constant $u_+>\epsilon$. Applying the monotone iteration scheme (see \cite[Proposition 2.4]{RRV1} for the formulation needed here) yields a smooth solution $v$ of (\ref{lambdau}) on $M$ such that $u_\epsilon \leq v\leq u_+$. Now, since $u_\epsilon$ does not vanish identically, the same holds for $v$, and this contradicts condition (iii) in Proposition \ref{conditions for sc} and, equivalently, condition (i). 
\end{proof}

Following the proof of \cite[Lemma 3.2]{LM1}, we extend Proposition \ref{wstoc} to manifolds with smooth compact boundary, under a stronger assumption than stochastic completeness:
\begin{prop}\label{wstocb}
{\it
Let $(M,g_M)$ be a Riemannian manifold with smooth compact boundary $\partial M$. Suppose that $(M,g_M)$ is complete and that the Ricci curvature is bounded below. Then, for every $u\in C^0(M)\cap W^{1,2}_{\loc}(M\setminus \partial M)$ with 
$u^\ast=\sup_M u<\infty$, if $\max_{\partial M}u<u^\ast$, the following cannot occur:
\begin{itemize}
\item
There exist $\epsilon>0$, $c>0$, and a continuous function $g:U\to \R$ defined on a nonempty open subset $U\subseteq M\setminus \partial M$ such that
\begin{align}
\Delta_{g_M} u &\ge g \quad \text{in the weak sense on $U$}, \label{gmu}\\
g &\ge 2c \quad \text{on} \ \Omega_\epsilon, \label{ome}\\ 
\overline{\Omega}_\epsilon&\subseteq U, \label{bOmU} 
\end{align}
where $\Omega_\epsilon=\{x\in M \mid u(x)>u^\ast-\epsilon\}$. 
\end{itemize}
}
\end{prop}
\begin{proof}
Suppose, for contradiction, that there exist $\epsilon>0$, $c>0$, and a continuous function $g:U\to \R$ satisfying \eqref{gmu}, \eqref{ome}, and \eqref{bOmU}. There exists a compact neighborhood $K \subset M$ of $\partial M$, together with a diffeomorphism $\Psi: K \to \partial M \times [0,1]$ such that $\Psi(\partial M)=\partial M\times\{0\}$. By shrinking $K$ if necessary, we may assume that $\max_K u < u^\ast$. There exists a complete Riemannian metric $\widetilde{g}_M$ on $M\setminus \partial M$ such that $\widetilde{g}_M = g_M$ on $M\setminus K$ and $\widetilde{g}_M$ has Ricci curvature bounded below. After possibly shrinking $\epsilon$ and $U$, we may assume that $U \subseteq M \setminus K$. Then we have
\begin{align*}
\Delta_{\widetilde{g}_M} u &\ge g \quad \text{in the weak sense on $U$}, \\
g &\ge 2c \quad \text{on } \Omega_\epsilon, \\ 
\overline{\Omega}_\epsilon &\subset U.
\end{align*}
This contradicts Proposition \ref{wstoc}, completing the proof.
\end{proof}

In order to extend the role of the Cheng--Yau maximum principle in \cite{LM1} to equations with weaker regularity of solutions, following \cite[Theorem 8.20]{BMPR1}, we prove the following proposition.

\begin{prop}\label{SL}{\it
Let $(X,g_X)$ be a Riemann surface with a complete continuous K\"ahler metric $g_X$ whose Ricci curvature is bounded below.
Let $f:\R\to\R$ be continuous and define
\begin{align*}
F(t):=\int_0^t f(s)ds.
\end{align*}
Assume that $f(0)=0$ and $f(t)>0$ for $t>0$ and that there exists a positive constant $C$ such that
for all $t>0$ the following holds (cf. \cite[Definition 1.9]{BMPR1}):
\begin{align}
\sup_{s\in [0,t]}f(s)\leq Cf(t) \label{Cinc}
\end{align}
Suppose moreover that 
\begin{align}
1/\sqrt{F}\in L^1(+\infty). \label{KOs}
\end{align}
Let $u$ be a function belonging to $L^1_\loc(X)$ such that $f(u)\in L^1_\loc(X)$ and $u$ satisfies the following inequality in the sense of distributions:
\begin{align}
\Delta_{g_X}u\geq f(u). \label{ugfu}
\end{align}
Then $\esssup_Xu<\infty$. Furthermore, under the stronger assumption that $u$ is continuous on $X\setminus Y$ for some relatively compact open set $Y$ with smooth boundary $\partial Y$, and that $u\in W^{1,2}_{\loc}(X\setminus \overline{Y})$, either $f\bigl(\sup_{X\setminus Y} u\bigr)\leq 0$ or $\sup_{X\setminus Y} u=\max_{\partial Y} u$ holds.
}
\end{prop}
For $f(t)=e^t-1$, all the above conditions are satisfied. Moreover, $f(u)=e^u-1$ is locally bounded, and hence locally integrable, whenever $u$ is upper semicontinuous. For the proof of Proposition \ref{SL}, we prepare the following lemma, which is a simplified version of \cite[Proposition 8.17]{BMPR1}:
\begin{lemm}[cf.~\cite{BMPR1}]\label{fRR}
{\it Let $f:\R\to\R$ be a continuous function that satisfies the conditions in Proposition \ref{SL}. Then for each $\epsilon>0$, $\theta>0$, $0<\delta<\lambda$, and $0\leq r_0<r_1$, there exist $R_1>r_1$ and a $C^2$ function $w:[r_0,R_1)\to[\delta,\infty)$ that solves
\begin{align}
\begin{cases}
w''(r)+\theta w'(r)\le \epsilon f(w(r)) & \text{on } [r_0,R_1),\\
w'(r)>0 & \text{on} \ [r_0,R_1),\\
w(r)\to\infty & \text{as } r\to R_1,\\
\delta\le w(r)\le \lambda & \text{on } [r_0,r_1]. 
\end{cases}
\label{wwf}
\end{align}
}
\end{lemm}

\begin{proof}
For a given $0<\sigma\leq 1$ to be specified later, set
\begin{align*}
C_\sigma\coloneqq \frac{1}{\sqrt{2\sigma}}\int_\delta^\infty\frac{1}{\sqrt{F(s)}}ds,
\end{align*}
which is well defined since we assume condition (\ref{KOs}). Since $\frac{1}{1+s}\notin L^1(+\infty)$, there exists $R_\sigma>r_0$ such that
\begin{align*}
C_\sigma=\int_{r_0}^{R_\sigma}\frac{1}{1+s}ds.
\end{align*}
Since $C_\sigma\to\infty$ as $\sigma\to 0$, we can choose $\sigma$ small enough that $R_\sigma>r_1$. We let $w:[r_0,R_\sigma)\rightarrow [\delta,\infty)$ be implicitly defined by the equation
\begin{align}
\int_{r}^{R_\sigma}\frac{1}{1+s}ds=\frac{1}{\sqrt{2\sigma}}\int_{w(r)}^\infty\frac{1}{\sqrt{F(s)}}ds, \label{wr}
\end{align}
so that, by definition, 
\begin{align*}
w(r_0)=\delta, w(r)\to\infty \ \text{as} \ r\to R_\sigma.
\end{align*}
Differentiating (\ref{wr}) yields
\begin{align}
w^\prime(r)=\frac{1}{1+r}\sqrt{2\sigma F(w(r))}, \label{wpr}
\end{align}
so that $w^\prime>0$ and 
\begin{align*}
2\sigma F(w(r))=((1+r)w^\prime(r))^2.
\end{align*}
Differentiating once more, we obtain
\begin{align*}
\sigma f(w(r))w^\prime(r)=(1+r)w^\prime(r)((1+r)w^\prime(r))^\prime,
\end{align*}
that is,
\begin{align}
\sigma \frac{f(w(r))}{1+r}=((1+r)w^\prime(r))^\prime. \label{that is}
\end{align}
Since $f>0$ on $(0,\infty)$, we infer that $(1+r)w^\prime(r)$ is increasing. Moreover, we deduce
\begin{align*}
(w^\prime(r)(1+r))^\prime &=w^{\prime\prime}(r)(1+r)+w^\prime(r) \\
&\geq w^{\prime\prime}(r)(1+r).
\end{align*}
Inserting this into (\ref{that is}), we obtain
\begin{align}
w^{\prime\prime}(r)\leq \sigma \frac{f(w(r))}{(1+r)^2}. \label{wpp}
\end{align}
We next integrate (\ref{that is}) on $[r_0,r]$, and we use (\ref{Cinc}) to deduce 
\begin{align*}
w^\prime(r)(1+r)&=w^\prime(r_0)(1+r_0)+\sigma\int_{r_0}^r\frac{f(w(s))}{1+s}ds \\
&\leq w^\prime(r_0)(1+r_0)+\sigma Cf(w(r))\int_{r_0}^r\frac{1}{1+s}ds,
\end{align*}
and thus, using $w(r_0)=\delta$, we obtain
\begin{align}
w^\prime(r)&\leq w^\prime(r)(1+r) \notag \\
&\leq w^\prime(r_0)(1+r_0)+\sigma Cf(w(r))\int_{r_0}^r\frac{1}{1+s}ds \notag \\
&\leq \left(\frac{w^\prime(r_0)(1+r_0)}{f(w(r))}+\sigma C\int_{r_0}^r\frac{1}{1+s}ds\right)f(w(r)) \notag \\
&\leq \left(\frac{Cw^\prime(r_0)(1+r_0)}{f(\delta)}+\sigma C\int_{r_0}^r\frac{1}{1+s}ds\right)f(w(r)) \notag \\
&=\left(\frac{Cw^\prime(r_0)(1+r_0)}{f(\delta)}+\frac{\sigma C}{\sqrt{2\sigma}}\int_\delta^{w(r)}\frac{1}{\sqrt{F(s)}}ds\right)f(w(r)) \label{mistake}\\
&\leq \left(\frac{Cw^\prime(r_0)(1+r_0)}{f(\delta)}+\frac{\sigma C}{\sqrt{2\sigma}}\int_\delta^\infty\frac{1}{\sqrt{F(s)}}ds\right)f(w(r)), \notag
\end{align}
where we have used (\ref{wr}) in (\ref{mistake}). Then we use (\ref{wpr}) with $r=r_0$ to obtain
\begin{align}
w^\prime(r)&\leq \left(\frac{C\sqrt{2\sigma F(w(r_0))}}{f(\delta)}+\frac{\sigma C}{\sqrt{2\sigma}}\int_\delta^\infty\frac{1}{\sqrt{F(s)}}ds\right)f(w(r)) \notag \\
&\leq \hat{C}_\sigma f(w(r)),
 \label{Csq}
\end{align}
where we set
\begin{align*}
\hat{C}_\sigma\coloneqq \frac{C\sqrt{2\sigma F(w(r_0))}}{f(\delta)}+\frac{\sigma C}{\sqrt{2\sigma}}\int_\delta^\infty\frac{1}{\sqrt{F(s)}}ds.
\end{align*}
Putting together (\ref{wpp}) and (\ref{Csq}), we obtain
\begin{align*}
w^{\prime\prime}(r)+\theta w^\prime(r) \leq \left(\frac{\sigma}{(1+r_0)^2}+\theta\hat{C}_{\sigma}\right)f(w(r)).
\end{align*}
Since $\hat{C}_{\sigma}\to 0$ as $\sigma\to 0$, we can choose $\sigma$ small enough that the differential inequality in (\ref{wwf}) is satisfied. To prove the last condition in (\ref{wwf}), simply observe that by (\ref{wr})
\begin{align*}
\int_{r_0}^{r_1}\frac{1}{1+s}ds=\frac{1}{\sqrt{2\sigma}}\int_\delta^{w(r_1)}\frac{1}{\sqrt{F(s)}}ds,
\end{align*}
and it is enough to choose $\sigma$ in such a way that 
\begin{align*}
\int_{r_0}^{r_1}\frac{1}{1+s}ds<\frac{1}{\sqrt{2\sigma}}\int_{\delta}^\lambda \frac{1}{\sqrt{F(s)}} ds.
\end{align*}
\end{proof}
\begin{lemm}{\it Under the hypotheses of Proposition \ref{SL}, set $u_+\coloneqq \max\{u,0\}$. Then $u_+$ also satisfies the following inequality in the sense of distributions:
\begin{align}
\Delta_{g_X}u_+\geq f(u_+). \label{fuplus}
\end{align}
Moreover, $u_+$ belongs to $W^{1,2}_\loc(X)\cap L^\infty_\loc(X)$.
}
\end{lemm}
\begin{proof}Let $U\subseteq X$ be a coordinate neighborhood identified with the unit disc and let $g_E$ be the Euclidean metric on $U$. Writing $g_X=e^{F_U}g_E$ on $U$, where $F_U\in C^0(U)$, inequality (\ref{ugfu}) reduces to the following on $U$:
\begin{align}
\Delta_{g_E}u\geq e^{F_U}f(u). \label{gEfu}
\end{align}
Since $f(0)=0$, the constant function 0 also satisfies inequality (\ref{gEfu}), and hence the claim in the first part of the lemma follows by applying Lemma \ref{uvFG} (ii), which we will prove later, to $u$ and the constant function 0, with $G(x,t)=e^{F_U(x)}f(t)$.  Since $f(t)\geq 0$ for $t\geq 0$, we observe that $u_+$ agrees almost everywhere with a subharmonic function. Since $u_+\ge0$ and subharmonic functions are locally bounded above, $u_+$ belongs to $L^\infty_\loc(X)$. Moreover, it is well-known that every locally bounded subharmonic function belongs to $W^{1,2}_\loc$ (cf. \cite[Section 1.1.1]{DGZ1}), and therefore $u_+\in W^{1,2}_\loc(X)$. This completes the proof. 
\end{proof}
\begin{proof}[Proof of Proposition \ref{SL}]

We suppose that the conclusion does not hold, i.e., for each positive real number $\lambda$, the set
\begin{align*}
E_\lambda\coloneqq \{x\in X\mid u_+(x)>\lambda\}
\end{align*}
has positive measure with respect to every smooth positive measure. 
Fix $\epsilon_0 > 0$. By Proposition \ref{GWT}, there exists a smooth K\"ahler metric $\widetilde{g}_{X,\epsilon_0}$ such that
\begin{align}
e^{-\epsilon_0} g_X \le \widetilde{g}_{X,\epsilon_0} \le e^{\epsilon_0} g_X, \label{gXe0}
\end{align}
and the Ricci curvature of $\widetilde{g}_{X,\epsilon_0}$ is bounded below. Since $g_X$ is complete, $\widetilde{g}_{X,\epsilon_0}$ is also complete. From here on, geodesic balls and the Riemannian distance function are taken with respect to this smooth K\"ahler metric. 
Let $B_{r_0}$ be an open geodesic ball centered at $o$ with radius $r_0 > 0$. Set 
\begin{align*}
M_0\coloneqq \esssup_{B_{2r_0}}u_+.
\end{align*}
Since $u_+$ belongs to $L^\infty_\loc(X)$ and $B_{2r_0}$ is relatively compact, $M_0$ is finite. Fix positive real numbers $\delta$ and $\lambda$ so that $M_0< \delta < \lambda$. Also, choose $r_1>2r_0$ large enough so that
\begin{align*}
E_\lambda\cap (B_{r_1}\setminus B_{r_0})
\end{align*}
has positive measure. 
Since the Ricci curvature of $\widetilde{g}_{X,\epsilon_0}$ is bounded below, the Laplacian comparison theorem (see, e.g., \cite[Theorem 1.2]{AMR1}) yields
\begin{align}
\Delta_{\widetilde{g}_{X,\epsilon_0}} r \leq A  \quad \text{weakly on } X \setminus B_{r_0/2}, \label{Ar}
\end{align}
for some constant $A > 0$, where $\Delta_{\widetilde{g}_{X,\epsilon_0}} = -d^\ast d$ denotes the Laplace--Beltrami operator, and $r$ is the Riemannian distance function from $o$. Applying Lemma \ref{fRR} with
\begin{align*}
&\theta=A, \\
&\epsilon=\frac{e^{-\epsilon_0}}{C},
\end{align*}
where $C$ is the constant in condition (\ref{Cinc}), we deduce the existence of $w$ satisfying (\ref{wwf}). Setting $\bar{w}(x)=w(r(x))$ and taking into account $w^\prime>0$, (\ref{Ar}), and the well-known formula $|\nabla^{\widetilde{g}_{X,\epsilon_0}}r|_{\widetilde{g}_{X,\epsilon_0}}=1$ almost everywhere, $\bar{w}$ solves
\begin{align}
\begin{cases}
\Delta_{\widetilde{g}_{X,\epsilon_0}}\bar{w}\leq \frac{e^{-\epsilon_0}}{C}f(\bar{w}) \ \text{on}\ B_{R_1}\setminus B_{r_0}\ \text{in the weak sense}, \\
w^\prime(r)>0 \ \text{on} \ B_{R_1}\setminus B_{r_0}, \ \bar{w}\to\infty \ \text{as} \ x\to\partial B_{R_1}\\
\delta\leq \bar{w}\leq \lambda \ \text{on} \ B_{r_1}\setminus B_{r_0}.
\end{cases}
\label{partialB1}
\end{align}
We compare $u_+$ and $\bar{w}$ on $U\coloneqq B_{R_1}\setminus \overline{B}_{r_0}$. We set 
\begin{align*}
v\coloneqq\max\{u_+-\overline{w},0\}.
\end{align*}
Observe that $v$ vanishes almost everywhere outside a compact subset of $U$.  Indeed, on $B_{2r_0}\setminus \overline{B}_{r_0}$, by the assumption $M_0<\delta$, $v$ is almost everywhere 0, and near the boundary of $B_{R_1}$, $v$ is also almost everywhere 0 since $u_+\in L^\infty_\loc(X)$ and $\overline{w}$ diverges at the boundary of $B_{R_1}$. Since $X$ is a Riemann surface, the metrics $g_X$ and $\widetilde{g}_{X,\epsilon_0}$ are conformally equivalent, and we may write $\widetilde{g}_{X,\epsilon_0}=\rho g_X$ with $e^{-\epsilon_0}\leq \rho \leq e^{\epsilon_0}$. 
By (\ref{fuplus}), we obtain
\begin{align}
\Delta_{\widetilde{g}_{X,\epsilon_0}}u_+&\geq \rho^{-1}f(u_+) \geq e^{-\epsilon_0}f(u_+). \label{fuplusplus}
\end{align}
Subtracting the first inequality in (\ref{partialB1}) from (\ref{fuplusplus}), we obtain the following:
\begin{align*}
\Delta_{\widetilde{g}_{X,\epsilon_0}}(u_+-\overline{w})\geq e^{-\epsilon_0}(f(u_+)-\frac{1}{C}f(\overline{w})).
\end{align*}
Note that the function $v$ belongs to $W^{1,2}_0(U)$ since $u_+\in W^{1,2}_\loc$ and $\overline{w}\in W^{1,\infty}_\loc$ and since $v$ vanishes almost everywhere outside a compact subset of $U$. Therefore, $v$ can be used as a test function, and multiplying both sides by $v$ and integrating with respect to the measure induced by $\widetilde{g}_{X,\epsilon_0}$, we have
\begin{align}
-\int_U(d(u_+-\overline{w}), dv)_{\widetilde{g}_{X,\epsilon_0}}d\mu_{\widetilde{g}_{X,\epsilon_0}}\geq \int_Ue^{-\epsilon_0}(f(u_+)-\frac{1}{C}f(\overline{w}))vd\mu_{\widetilde{g}_{X,\epsilon_0}}, \label{integralU}
\end{align}
where we denote by $(\cdot, \cdot)_{\widetilde{g}_{X,\epsilon_0}}$ the inner product induced by $\widetilde{g}_{X,\epsilon_0}$, and by $d\mu_{\widetilde{g}_{X,\epsilon_0}}$ the volume measure of $\widetilde{g}_{X,\epsilon_0}$. From the weak differential formula for the function $\max\{\cdot, 0\}$ and the chain rule for Sobolev functions, we obtain $dv=\mathbf{1}_{\{u_+>\overline{w}\}}d(u_+-\overline{w})$ almost everywhere on $U$, where we denote by $\mathbf{1}_{\{u_+>\overline{w}\}}$ the indicator function of the set $\{u_+>\overline{w}\}$. Since $\mathbf{1}_{\{u_+>\overline{w}\}}^2=\mathbf{1}_{\{u_+>\overline{w}\}}$, from (\ref{integralU}), we have
\begin{align}
-\int_U(dv, dv)_{\widetilde{g}_{X,\epsilon_0}}d\mu_{\widetilde{g}_{X,\epsilon_0}}\geq \int_Ue^{-\epsilon_0}(f(u_+)-\frac{1}{C}f(\overline{w}))vd\mu_{\widetilde{g}_{X,\epsilon_0}}. \label{integralU2}
\end{align}
The right-hand side of inequality (\ref{integralU2}) is nonnegative. Indeed, for almost every point $x\in U$ such that $v(x)>0$, since $u_+(x)>\bar{w}(x)$, by (\ref{Cinc}), we obtain $f(u_+(x))\geq \frac{1}{C}f(\overline{w}(x))$. This implies that $\int_U(dv, dv)_{\widetilde{g}_{X,\epsilon_0}}d\mu_{\widetilde{g}_{X,\epsilon_0}}=0$. Since $v$ vanishes almost everywhere outside a compact subset of $U$, there exists a function $\widetilde{v}$ on $X$ that belongs to $W^{1,2}_\loc(X)$ and such that
\begin{align*}
\widetilde{v}=
\begin{cases}
v &\ \text{on $U$}\\
0 &\ \text{on $X\setminus U$}.
\end{cases}
\end{align*}
Since $dv=0$ almost everywhere on $U$, we have $d\widetilde{v}=0$ almost everywhere on $X$. Hence $\widetilde{v}$ is almost everywhere constant on the connected surface $X$. Since $\widetilde{v}=0$ on $B_{r_0}$, this constant is 0. Therefore,  $v=0$ almost everywhere on $U$. Since $\overline{w}$ satisfies $\overline{w}\leq \lambda$ on  $B_{r_1}\setminus B_{r_0}$, this in particular implies that $
E_\lambda\cap (B_{r_1}\setminus B_{r_0})$ has measure 0, which contradicts the assumption. 

Finally, we prove the last assertion of the proposition. 
Suppose that $u$ is continuous on $X\setminus Y$ for some relatively compact open set $Y$ with smooth boundary $\partial Y$, and that $u \in W^{1,2}_{\loc}(X\setminus \overline{Y})$. 
Assume further that $\sup_{X\setminus Y} u > \max_{\partial Y} u$. Fix $\epsilon_0>0$. By Proposition \ref{GWT}, there exists a complete K\"ahler metric $\widetilde{g}_{X,\epsilon_0}$ satisfying (\ref{gXe0}) and whose Ricci curvature is bounded below. By (\ref{ugfu}), we obtain
\begin{align*}
\Delta_{\widetilde{g}_{X,\epsilon_0}}u&\geq \rho^{-1}f(u), 
\end{align*}
where $\rho$ is defined by $\widetilde{g}_{X,\epsilon_0}=\rho g_X$.
By Proposition \ref{wstocb}, for each integer $k \ge 1$, there exists a point 
\begin{align*}
x_k \in \Omega_{1/k} = \{ x \in X\setminus Y \mid u(x) > \sup_{X\setminus Y} u - 1/k \}
\end{align*}
such that $\rho(x_k)^{-1}f(u(x_k)) < 1/k$.  Since $\rho$ is bounded, letting $k \to \infty$, we obtain $f(\sup_{X\setminus Y} u) \leq 0$.
\end{proof}

\subsection{Preliminaries on linear algebra}
The purpose of this subsection is to reformulate the linear-algebraic lemmas established in \cite{LM1} (see also \cite{Sim2}). All the proofs can be found in \cite{LM1}, except for Lemma \ref{lem4}.
Let $h_1,\dots, h_r$ be positive real numbers satisfying $\Pi_{i=1}^r h_i=h_1\cdot h_2\cdots h_{r-1}\cdot h_r=1$. We set $H_j\coloneqq h_j^{-1}\cdot h_{j+1}$ for each $j=1,\dots, r-1$. Let $\phi$ be a real number or $-\infty$. We set a non-negative real number $H_r$ as $H_r\coloneqq H_1^{-1}\cdot H_2^{-1}\cdots H_{r-1}^{-1}\cdot e^{r\phi}=h_r^{-1}\cdot h_1\cdot e^{r\phi}$. We also set $H_0\coloneqq H_r$. The following Lemma \ref{lem1} and Lemma \ref{lem2} are special cases of well-known inequalities which are formulated in more general settings (cf. \cite{LM1, Sim2}).
\begin{lemm}[cf.~\cite{LM1, Sim2}]\label{lem1}{\it There exists a positive constant $C$ which depends only on $r$ such that the following holds for any $h_1,\dots, h_r$ satisfying $\Pi_{i=1}^r h_i=h_1\cdot h_2\cdots h_{r-1}\cdot h_r=1$ and $\phi$:
\begin{align}
\sum_{j=1}^r(H_{j-1}-H_j)^2\geq C(\sum_{j=1}^rH_j-re^\phi)^2. \label{ephi1}
\end{align}
}
\end{lemm}
\begin{proof} We give a proof directly without using the general theory. Since the variable transformation from $h_1,\dots, h_{r-1}$ to $H_j=h_j^{-1}\cdot h_{j+1}$ is invertible, it is sufficient to state that inequality (\ref{ephi1}) holds for any positive real numbers $H_1,\dots, H_{r-1}$. We first suppose that $\phi$ is not $-\infty$.  Set $\widetilde{H}_j\coloneqq e^{-\phi} H_j\ (j=1,\dots, r-1)$ and $\widetilde{H}_r\coloneqq \widetilde{H}_1^{-1}\cdots \widetilde{H}_{r-1}^{-1}=e^{-\phi}H_r$. Then, from inequality (\ref{ephi1}) for $\widetilde{H}_1, \dots, \widetilde{H}_r$ in the case where $\phi=0$, we obtain inequality (\ref{ephi1}) for $H_1, \dots, H_r$ with respect to $\phi$. Therefore, from the beginning, we may assume that $\phi=0$. We arrange the method given in \cite[Proof of Proposition 3.8]{LM1} for our proof. We set $c_\max\coloneqq \max\{H_j\mid 1\leq j\leq r\}$ and $c_\min\coloneqq \min \{H_j\mid 1\leq j\leq r\}$. We take $j_\max$ and $j_\min$ so that $c_\max=H_{j_\max}$ and $c_\min=H_{j_\min}$. 
In what follows, for simplicity we will only consider the case where they are 1 and $r$, respectively. Even if this were not the case, a simple consideration shows that the following estimates would hold just as well, with a slight modification of the argument. The following holds:
\begin{align*}
c_\max-c_\min&=\sum_{j=2}^r(H_{j-1}-H_j)\\
&\leq (\sum_{j=2}^r(H_{j-1}-H_j)^2)^{1/2}(\sum_{j=2}^r1^2)^{1/2}\\
&\leq \sqrt{r-1}(\sum_{j=1}^r(H_{j-1}-H_j)^2)^{1/2},
\end{align*}
where we have used the Cauchy--Schwarz inequality. Therefore we have 
\begin{align}
\sum_{j=1}^r(H_{j-1}-H_j)^2&\geq \frac{1}{r-1}(c_\max-c_\min)^2 \notag \\
&=\frac{1}{r-1}c_\max^2(1-c_\min/c_\max)^2. \label{cmax}
\end{align}
On the other hand, we have
\begin{align*}
\sum_{j=1}^rH_j-r\leq rc_\max-r.
\end{align*}
Since $\Pi_{i=1}^rH_i=H_1\cdot H_2\cdots H_r=1$, it holds that $\sum_{j=1}^rH_j-r\geq 0$. Hence we have
\begin{align}
(\sum_{j=1}^rH_j-r)^2&\leq r^2(c_\max-1)^2 \notag \\
&=r^2c_\max^2(1-1/c_\max)^2 \label{cmin}.
\end{align}
We compare (\ref{cmax}) and (\ref{cmin}). Since $c_\min\leq 1$, we have
\begin{align*}
1-c_\min/c_\max\geq 1-1/c_\max.
\end{align*}
Furthermore, since $1-1/c_\max\geq 0$, we have
\begin{align}
(1-c_\min/c_\max)^2\geq (1-1/c_\max)^2. \label{cmaxmin}
\end{align}
Then from (\ref{cmax}), (\ref{cmin}), and (\ref{cmaxmin}), for the case where $\phi\neq -\infty$, we have
\begin{align}
\sum_{j=1}^r(H_{j-1}-H_j)^2\geq \frac{1}{r^2(r-1)}(\sum_{j=1}^rH_j-re^\phi)^2. \label{rsquare}
\end{align}
We next consider the case where $\phi=-\infty$, i.e., $e^\phi=0$. Let $u_1,\dots, u_{r-1}$ be the canonical basis of $\R^{r-1}$. We define an invertible $(r-1)\times (r-1)$-matrix $A$ as follows:
\begin{align*}
A u_j&\coloneqq u_j-u_{j+1} \ (j=1,\dots, r-2),  \\
A u_{r-1}&\coloneqq u_{r-1}, 
\end{align*}
where $u_1,\dots, u_{r-1}$ are considered to be column vectors. We set $\hat{H}\coloneqq (H_1,\dots, H_{r-1})\in\R^{r-1}$. Then we have
\begin{align*}
\sum_{j=1}^r(H_{j-1}-H_j)^2&=H_1^2+\sum_{j=2}^{r-1}(H_{j-1}-H_j)^2+H_{r-1}^2 \\
&\geq H_1^2+\sum_{j=2}^{r-1}(H_{j-1}-H_j)^2 \\
&=|A\hat{H}|^2,
\end{align*}
where the norm $|\cdot|$ denotes the Hilbert--Schmidt norm. On the other hand, by the Cauchy--Schwarz inequality, we have
\begin{align*}
(\sum_{j=1}^{r-1}H_j)^2&\leq (r-1)|\hat{H}|^2 \\
&\leq (r-1)|A^{-1}\cdot A\hat{H}|^2 \\
&\leq (r-1)|A^{-1}|^2|A\hat{H}|^2 \\
&=(r-1)\cdot \frac{1}{2}r(r-1) |A\hat{H}|^2\\
&=\frac{1}{2}r(r-1)^2|A\hat{H}|^2.
\end{align*}
Therefore we have
\begin{align}
\sum_{j=1}^r(H_{j-1}-H_j)^2\geq \frac{2}{r^2(r-1)}(\sum_{j=1}^{r-1}H_j)^2 \label{rsquare2}
\end{align}
for the case where $\phi=-\infty$. From inequalities (\ref{rsquare}) and (\ref{rsquare2}), we have inequality (\ref{ephi1}) with $C=\frac{1}{r^2(r-1)}$.
\end{proof}

\begin{lemm}[cf.~\cite{LM1, Sim2}]\label{lem2} {\it Let $C$ be a positive constant as in Lemma \ref{lem1}. Suppose that there exists a positive real number $H>0$ such that $e^\phi\leq H$ and that $\sum_{j=1}^rH_j\geq 2r H$. Then the following holds:
\begin{align*}
\sum_{j=1}^r(H_{j-1}-H_j)^2\geq \frac{C}{4}(\sum_{j=1}^rH_j)^2.
\end{align*}
}
\end{lemm}
\begin{proof} The proof is the same as in the proof of \cite[Lemma 3.7]{LM1}. From the assumption, we have
\begin{align*}
\sum_{j=1}^rH_j-re^\phi\geq rH \geq re^\phi.
\end{align*}
This implies
\begin{align*}
\frac{1}{2}\sum_{j=1}^r H_j \geq re^\phi,
\end{align*}
and thus 
\begin{align*}
\sum_{j=1}^rH_j-re^\phi\geq \frac{1}{2}\sum_{j=1}^r H_j.
\end{align*}
From this we have the desired inequality.
\end{proof}

For the proof of the following Lemma \ref{lem3}, we refer the reader to \cite[Proposition 3.8]{LM1}.
\begin{lemm}[\cite{LM1}]\label{lem3}{\it For any $0<\epsilon<1$, there exists $\delta>0$ that depends only on $\epsilon$ and $r$, such that the following holds:
\begin{itemize}
\item If there exists $1\leq j\leq r-1$ such that $H_j\leq \epsilon e^\phi$, then we obtain $\sum_{j=1}^r(H_{j-1}-H_j)^2\geq \delta (\sum_{j=1}^r H_j)^2$.
\end{itemize}
}
\end{lemm}
Let $h_1^\prime,\dots, h_r^\prime$ be positive real numbers satisfying $\Pi_{i=1}^rh_i^\prime=h_1^\prime\cdot h_2^\prime\cdots h_{r-1}^\prime\cdot h_r^\prime=1$. We set $s_j\coloneqq h_j^{-1}\cdot h_j^\prime$ for each $j=1,\dots, r$. We also set $s\coloneqq (s_1,\dots, s_r)\in (\R_{>0})^r$. We assume that there exist positive constants $B$ and $C$ such that the following conditions hold:
\begin{itemize}
\item $H_j^{-1}\leq B$ for each $j=1,\dots, r-1$.
\item $\sum_{j=1}^rs_j\leq C$.
\end{itemize}
The following lemma is a variant of \cite[Lemma 3.10]{LM1}:
\begin{lemm}\label{lem4}{\it For each $\epsilon>0$, if $\sum_{j=1}^{r-1}s_j^{-1}(s_{j+1}-s_j)^2H_j \leq \epsilon^2$, then 
\begin{align*}
\sum_{j=1}^r|s_j-1|\leq r\sqrt{(r-1)BC}\cdot \epsilon.
\end{align*}
}
\end{lemm}
\begin{proof} By assumption, we have
\begin{align*}
\frac{1}{BC}\sum_{j=1}^{r-1}(s_{j+1}-s_j)^2\leq \sum_{j=1}^{r-1}s_j^{-1}(s_{j+1}-s_j)^2H_j\leq \epsilon^2.
\end{align*}
For each $i<k$, we have
\begin{align*}
s_k-s_i=\sum_{j=i}^{k-1}(s_{j+1}-s_j),
\end{align*}
and thus, by the Cauchy-Schwarz inequality, we obtain
\begin{align}
|s_k-s_i|\leq \sqrt{k-i}\left(\sum_{j=i}^{k-1}(s_{j+1}-s_j)^2\right)^{1/2}\leq \sqrt{r-1}\sqrt{BC}\epsilon.\label{sk-sj}
\end{align}
Since $\Pi_{j=1}^rs_j=1$, $\min_{1\leq j\leq r}s_j\leq 1\leq \max_{1\leq j\leq r}s_j$. Therefore, we obtain
\begin{align}
|s_j-1|\leq \max_{1\leq j\leq r} s_j-\min_{1\leq j\leq r}s_j.\label{sj-1}
\end{align}
From (\ref{sk-sj}) and (\ref{sj-1}), we have
\begin{align*}
\sum_{j=1}^r|s_j-1|\leq r\sqrt{(r-1)BC}\epsilon.
\end{align*}
This implies the claim.
\end{proof}
\begin{rem} The consistency between the notation $|[f,s]s^{-1/2}|_h$ used in \cite[Lemma 3.10]{LM1} and the notation we used above is as follows (see also \cite[Lemma 3.1]{Sim1} and \cite[Proof of Lemma 3.10]{LM1}): Let $h=(h_1,\dots, h_r)$ and $h^\prime=(h_1^\prime,\dots, h_r^{\prime})$ be diagonal Hermitian metrics on $(\K_r,\Phi(q))\rightarrow X$. We set $s_j\coloneqq h_j^{-1}\otimes h_j^\prime$ for each $j=1,\dots, r$, $s_0=s_r$, and $s\coloneqq (s_1,\dots, s_r)$. We also define Hermitian metrics $H_1,\dots, H_{r-1}, H_r=H_0$ on $K_X^{-1}\rightarrow X$ in the same way as in Section \ref{2.1}. Let $u_1$ be a local frame around $x\in X$ of $K_X^{(r-1)/2}\rightarrow X$. Let $z$ be a local holomorphic chart around $x$ and we define a local frame $u_j$ of $K_X^{\frac{r-(2j-1)}{2}}\rightarrow X$ as $u_j\coloneqq u_1\otimes (\frac{\partial}{\partial z})^{(j-1)}$ for each $j=2,\dots, r$. Then we can check that $\Phi(q)^{\ast h}u_{j+1}=H_j(\frac{\partial}{\partial z}, \frac{\partial}{\partial z})u_j\otimes d\bar{z}$ and the following:
\begin{align*}
&\inum[\Phi(q),s]s^{-1/2}\wedge ([\Phi(q),s]s^{-1/2})^{\ast h} u_j \\
=&\inum[\Phi(q),s]s^{-1}[s,\Phi(q)^{\ast h}]u_j \\
=&s_{j-1}^{-1}(s_j-s_{j-1})^2H_{j-1}\left(\frac{\partial}{\partial z},\frac{\partial}{\partial z}\right)u_j\inum dz\wedge d\bar{z}
\end{align*}
for each $j=1,\dots,r$. From this, we can verify the consistency between the notation used in \cite{LM1} and our notation.
\end{rem}
Throughout the remainder of the paper, we write $n:=\lfloor r/2\rfloor$, where $\lfloor x\rfloor$ denotes the greatest integer not exceeding $x$.
\begin{lemm}[\cite{LM1}, part of Theorem 4.4]\label{lemma33}{\it Suppose that $H_j=H_{r-j}$ for all $j=1,\dots, r$ and $H_1,\dots, H_r$ satisfy
\begin{align}
H_{j-1}\cdot H_j^{-1} < 1 \ \text{for all $1\leq j\leq n$}. \label{2jn}
\end{align}
Then the following holds:
\begin{align*}
h_j<e^{-\frac{1}{2}(r-(2j-1))\phi}\ \text{for all $j=1,\dots, n$.}
\end{align*}
}
\end{lemm}
\begin{proof} We give an alternative proof based on a slightly different perspective from that in \cite[p.~27]{LM1}. It is enough to consider the case where $\phi\in\R$. Note that from the assumption, we have $h_j=h_{r-j+1}^{-1}$ for all $j=1,\ldots,r$ (see Section \ref{rdc}). We set 
\begin{align*}
&v_j\coloneqq \log (h_j)+\frac{1}{2}(r-(2j-1))\phi \ \text{for $j=1,\dots, r$}, \\
&A_j\coloneqq \log(H_j)-\phi\ \text{for $j=0,1,\dots, r-1$}.
\end{align*}
Then for each $j=1,\dots, n$, by a direct calculation, we have
\begin{align*}
2v_j=A_0+2\sum_{k=1}^{j-1}A_k.
\end{align*}
It is enough to show that $v_j<0$ for all $j=1,\dots,n$. Suppose, for contradiction, that $v_j\geq 0$ for some $1\leq j\leq n$. By (\ref{2jn}), we have
$A_0<A_1<\cdots<A_n$.
Hence $A_{j-1}\geq0$; otherwise, $A_k<0$ for all $0\leq k\leq j-1$, and the above identity would imply $v_j<0$, a contradiction. From (\ref{2jn}), $A_k>0$ for all $j\leq k\leq n$. This implies 
\begin{align*}
\sum_{k=0}^{r-1}A_k&=A_0+\sum_{k=1}^{j-1}A_k+\sum_{k=j}^nA_k+\sum_{k=n+1}^{r-1}A_k \\
&=A_0+\sum_{k=1}^{j-1}A_k+\sum_{k=j}^nA_k+\sum_{k=1}^{r-(n+1)}A_k \\
&=2v_j+2\sum_{k=j}^{n-1}A_k+(r+1-2n)A_n>0,
\end{align*}
which contradicts the fact $\sum_{j=0}^{r-1}A_j=0$.
\end{proof}
\subsection{Proof of Theorem \ref{main theorem 1}}
Let $h=(h_1,\dots, h_r)$ be a weak solution to equation (\ref{phi}) satisfying $\det(h)=1$. By Proposition \ref{C11}, $h$ is of class $C^{1,1}_\loc$. We denote by $H_1,\dots, H_r$ the Hermitian metrics on $K_X^{-1}\rightarrow X$ induced by $h$ as in Section \ref{2.2}.
The following is an extension of the well-known inequality for Higgs bundles established by Simpson \cite[Lemma 10.1]{Sim1} (for the proof, see \cite[Proposition 11]{Miy3} and \cite[Proposition 23]{Miy5}):

\begin{lemm}\label{free energy}{\it The following inequality holds in the sense of distributions: 
\begin{align}
\inum\partial\bar{\partial} \log(\sum_{j=1}^r\vol(H_j)/\vol(H_\refe)) \geq \frac{\sum_{j=1}^r(\vol(H_{j-1})-\vol(H_j))^2}{\sum_{j=1}^r\vol(H_j)}+\inum F_{H_\refe}. \label{well-known}
\end{align}
}
\end{lemm}
By using the above Lemma \ref{free energy}, we can show the following as in \cite[Lemma 3.13]{LM1}:
\begin{lemm}\label{ephi}{\it Suppose that there exists a Hermitian metric $H$ on $K_X^{-1}\rightarrow X$ such that
\begin{align*}
e^{\varphi}h_\refe^{-1}\leq H.
\end{align*}
Then there exist positive constants $C_1$ and $C_2$ independent of $H$ such that
\begin{align}
\inum\Lambda_H\partial\bar{\partial}\log(\sum_{j=1}^r\vol(H_j)/\vol(H))\geq C_1(\sum_{j=1}^r\vol(H_j)/\vol(H))-C_2+\inum \Lambda_HF_H, \label{C1C2}
\end{align}
where we denote by $\Lambda_H=(\vol(H)\wedge)^{\ast}$ the contraction operator induced by $H$.
}
\end{lemm}
\begin{proof}
In what follows, we will identify the Hermitian metrics on $K_X^{-1}\rightarrow X$ with positive real numbers at each point by fixing one trivialization for each point. Note, however, that all the following discussions are independent of the choice of trivialization. From Lemma \ref{lem1}, there exists a positive constant $C$ depending only on $r$ such that
\begin{align*}
\sqrt{\sum_{j=1}^r(H_{j-1}-H_j)^2}\geq \sqrt{C}(\sum_{j=1}^rH_j-re^{\varphi}h_\refe^{-1}).
\end{align*}
From the assumption $e^\varphi h_\refe^{-1}\leq H$, we have
\begin{align}
\sqrt{\sum_{j=1}^r(H_{j-1}-H_j)^2}\geq \sqrt{C}(\sum_{j=1}^rH_j-re^{\varphi}h_\refe^{-1})\geq \sqrt{C}(\sum_{j=1}^rH_j-2r H). \label{sqrtc1}
\end{align}
On the other hand, from Lemma \ref{lem2}, under the assumption of $\sum_{j=1}^rH_j\geq 2rH$, we have
\begin{align}
\sqrt{\sum_{j=1}^r(H_{j-1}-H_j)^2}\geq \frac{\sqrt{C}}{2}\sum_{j=1}^rH_j. \label{sqrtc2}
\end{align}
From inequalities (\ref{sqrtc1}) and (\ref{sqrtc2}), we have
\begin{align}
\sum_{j=1}^r(H_{j-1}-H_j)^2\geq \frac{C}{2}(\sum_{j=1}^rH_j)(\sum_{j=1}^rH_j-2rH).
\end{align}
Therefore, the following holds globally over $X$:
\begin{align}
\frac{\sum_{j=1}^r(\vol(H_{j-1})-\vol(H_j))^2}{\sum_{j=1}^r\vol(H_j)}\geq \frac{C}{2}(\sum_{j=1}^r\vol(H_j)-2r\vol(H)). \label{-2rH}
\end{align}
From the above inequality (\ref{-2rH}) and inequality (\ref{well-known}), we have the desired conclusion.
\end{proof}
Throughout this section, we follow the notation in Lemma \ref{ephi} and denote by $\Lambda_H$ the contraction operator induced by a Hermitian metric $H$ on $K_X^{-1} \to X$. For the K\"ahler metric $g_H$ induced by $H$, we have $\Delta_{2g_H}=\inum\Lambda_H\partial\bar\partial$. In this subsection, we apply Propositions \ref{wstoc} and \ref{SL} to inequalities involving this operator with respect to $2g_H$. The following Corollary \ref{CYM} is an extension of the crucial insight discovered by Li--Mochizuki \cite[Section 3.3]{LM1} to more general subharmonic weight functions:

\begin{cor}\label{CYM}{\it Let $g_X$ be a complete continuous K\"ahler metric on $X$ whose Ricci curvature is bounded below and $H$ the Hermitian metric on $K_X^{-1}\rightarrow X$ induced by $g_X$. 
Then for the K\"ahler metric $g_X$, the following (i) and (ii) are equivalent:
\begin{enumerate}[(i)]
\item The function $e^{\varphi}h_\refe^{-1}\otimes H^{-1}$ is a bounded function.
\item For each $j=1,\dots, r$, $H_j\otimes H^{-1}$ is a bounded function. 
\end{enumerate}
}
\end{cor}
\begin{proof} It is clear that (ii) implies (i). We prove the converse. 
We suppose that (i) holds. For simplicity, we assume that we can take 1 as an upper bound on the function $e^\varphi h_\refe^{-1}\otimes H^{-1}$. Since the Ricci curvature of $g_X$ is bounded below, there exists a positive constant $C_3$ such that $\inum \Lambda_H F_H\geq -C_3$. We set 
\begin{align*}
C_4&\coloneqq \frac{C_1}{C_2+C_3}, \\
u&\coloneqq \log(\sum_{j=1}^r\vol(H_j)/\vol(H))+\log C_4.
\end{align*}
From inequality (\ref{C1C2}), we obtain
\begin{align}
\inum\Lambda_H\partial\bar{\partial}u\geq (C_2+C_3)(e^u-1). \label{C2C3}
\end{align}
By applying Proposition \ref{SL} to inequality (\ref{C2C3}), we obtain the claim.
\end{proof}
 Following \cite[Lemma 3.22]{LM1}, we can show the following:
\begin{lemm}\label{log}{\it There exist positive constants $C_1$ and $C_2$ depending only on $r$ such that the following holds for each $i=1,\dots, r-1$:
\begin{align*}
\inum\Lambda_{H_i}\partial\bar{\partial} \log(\sum_{j=1}^r(\vol(H_j)/\vol(H_i))) \geq C_1(\sum_{j=1}^r\vol(H_j)/\vol(H_i))-C_2+\inum \Lambda_{H_i}F_{H_i}. 
\end{align*}
}
\end{lemm}
\begin{proof} As in the proof of Lemma \ref{ephi}, we identify the Hermitian metrics on $K_X^{-1}\rightarrow X$ with positive real numbers at each point by fixing one trivialization for each point. For each $i=1,\dots, r-1$, if $2H_i\geq e^{\varphi}h_\refe^{-1}$, then by the same argument as in the proof of Lemma \ref{ephi}, we have
\begin{align*}
\frac{\sum_{j=1}^r(H_{j-1}-H_j)^2}{\sum_{j=1}^rH_j}\geq \frac{C}{2}(\sum_{j=1}^rH_j-4rH_i),
\end{align*}
where $C$ is a positive constant as in Lemma \ref{lem1} and Lemma \ref{lem2}. On the other hand, if $2H_i<e^{\varphi}h_\refe^{-1}$, then from Lemma \ref{lem3}, there exists a positive constant $\delta$ depending only on $r$ (and the choice of $\epsilon=\frac{1}{2}$) such that
\begin{align*}
\frac{\sum_{j=1}^r(H_{j-1}-H_j)^2}{\sum_{j=1}^rH_j}\geq \delta (\sum_{j=1}^rH_j).
\end{align*}
From the above argument, the following holds globally on $X$:
\begin{align}
\frac{\sum_{j=1}^r(\vol(H_{j-1})-\vol(H_j))^2}{\sum_{j=1}^r\vol(H_j)}\geq \min\{C/2,\delta\}(\sum_{j=1}^r\vol(H_j))-2rC\vol(H_i). \label{c1}
\end{align}
From inequality (\ref{c1}) and inequality (\ref{well-known}), we have the desired conclusion.
\end{proof}
From Lemma \ref{log}, we have the following:
\begin{cor}\label{ji}{\it Suppose that an $H_i$ among $H_1, \dots, H_{r-1}$ is complete. Then $\sum_{j=1}^rH_j\otimes H_i^{-1}$ is a bounded function. If moreover $e^\varphi$ is continuous, then $\sum_{j=1}^rH_j\otimes H_i^{-1}\leq \frac{C_2+2}{C_1}$ with constants $C_1$ and $C_2$ in Lemma \ref{log}.
}
\end{cor}

\begin{proof} 
From equation (\ref{cyc2}), we have $\inum\Lambda_{H_i}F_{H_i}\geq -2$ (cf. \cite[Lemma 2.4]{LM1}). 
We set
\begin{align*}
&C\coloneqq C_2+2, \\
&u\coloneqq \log(\sum_{j=1}^r(\vol(H_j)/\vol(H_i)))+\log(C_1/(C_2+2)).
\end{align*}
Then we have
\begin{align}
\inum\Lambda_{H_i}\partial\bar{\partial}u\geq C(e^u-1). \label{eu-1}
\end{align}
By applying Proposition \ref{SL} to inequality (\ref{eu-1}), it follows that $u$ is essentially bounded above and hence so is $\sum_{j=1}^rH_j\otimes H_i^{-1}$. Since $\sum_{j=1}^rH_j\otimes H_i^{-1}$ is locally a sum of a smooth function and a subharmonic function, a local mean value inequality implies that it is indeed a bounded function. Suppose moreover that $e^\varphi$ is continuous. Since a locally bounded subharmonic function belongs to $W^{1,2}_\loc$ (cf. \cite[Section 1.1.1]{DGZ1}), $e^\varphi$ belongs to $W^{1,2}_\loc$. Therefore $u$ is a continuous function belonging to $W^{1,2}_\loc$, and thus by applying Proposition \ref{SL} to inequality (\ref{eu-1}), we have $u\leq 0$. This implies that $\sum_{j=1}^rH_j\otimes H_i^{-1}\leq \frac{C_2+2}{C_1}.$
\end{proof}
Let $h^\prime=(h_1^\prime,\dots, h_r^\prime)$ be another solution to equation (\ref{phi}) satisfying $\det(h^\prime)=1$. We denote by $s=(s_1,\dots, s_r)$ the endomorphism determined by $h(s\cdot,\cdot)=h^\prime(\cdot,\cdot)$. We set $s_0\coloneqq s_r$. 
\begin{lemm}\label{ddbars}{\it The endomorphism $s$ satisfies the following inequality in the sense of distributions:
\begin{align}
\inum\partial\bar{\partial} \sum_{j=1}^rs_j\geq \sum_{j=1}^rs_{j-1}^{-1}(s_{j-1}-s_j)^2\vol(H_{j-1}). \label{s}
\end{align}
}
\end{lemm}
\begin{proof} By a direct computation, we have
\begin{align*}
\inum\partial \bar{\partial} \sum_{j=1}^rs_j&= \inum\partial \bar{\partial}\sum_{j=1}^re^{\log s_j} \\
&= \inum \partial \sum_{j=1}^r(\bar{\partial}\log s_j) e^{\log s_j} \\
&= \inum \sum_{j=1}^r\partial (\log s_j)\wedge\bar{\partial}(\log s_j) e^{\log s_j}+\inum\sum_{j=1}^r(\partial \bar{\partial} \log s_j )e^{\log s_j} \\
&\geq \inum \sum_{j=1}^rs_j\partial \bar{\partial}(\log s_j) \\
&=\sum_{j=1}^rs_j(\inum F_{h_j}-\inum F_{h_j^\prime}) \\
&=\sum_{j=1}^rs_j(-\vol(H_{j-1})+\vol(H_j)+\vol(H_{j-1}^\prime)-\vol(H_j^\prime)). 
\end{align*}
On the other hand, since we have $s_j\cdot s_{j-1}^{-1}=H_{j-1}^{-1}\otimes H_{j-1}^\prime$ for each $j=1,\dots, r$, it holds that
\begin{align*}
s_{j-1}^{-1}(s_j-s_{j-1})^2\vol(H_{j-1})&=(s_j-s_{j-1})(s_js_{j-1}^{-1}-1)\vol(H_{j-1}) \\
&=(s_j-s_{j-1})(\vol(H_{j-1}^\prime)-\vol(H_{j-1})).
\end{align*}
Therefore, we have
\begin{align*}
\sum_{j=1}^rs_{j-1}^{-1}(s_{j-1}-s_j)^2\vol(H_{j-1}) 
&=\sum_{j=1}^r(s_j-s_{j-1})(\vol(H_{j-1}^\prime)-\vol(H_{j-1})) \\
&=\sum_{j=1}^r(s_j\vol(H_{j-1}^\prime)+s_{j-1}\vol(H_{j-1})-s_{j-1}\vol(H_{j-1}^\prime)-s_j\vol (H_{j-1})) \\
&=\sum_{j=1}^rs_j(\vol(H_{j-1}^\prime)+\vol(H_j)-\vol(H_j^\prime)-\vol(H_{j-1})). \\
&=\sum_{j=1}^rs_j(-\vol(H_{j-1})+\vol(H_j)+\vol(H_{j-1}^\prime)-\vol(H_j^\prime)). 
\end{align*}
This implies the claim.
\end{proof}
Then we prove Theorem \ref{main theorem 1}.
\begin{proof}[Proof of Theorem \ref{main theorem 1}] Let $H_1^\prime, \dots, H_r^\prime$ be the Hermitian metrics constructed from $h^\prime$ as in Section \ref{2.2}. From Corollaries \ref{CYM} and \ref{ji}, it follows that $H_1,\dots, H_{r-1}$ and $H_1^\prime,\dots, H_{r-1}^\prime$ are mutually bounded. In particular, $s_1,\dots, s_r$ are bounded functions on $X$. We choose one of $H_1,\dots, H_{r-1}, H_1^\prime,\dots, H_{r-1}^\prime$ and denote it by $H$. Since $H_1,\dots, H_{r-1}$ and $H_1^\prime,\dots, H_{r-1}^\prime$ are solutions to equation (\ref{cyc2}), the Gaussian curvature associated with $H$ is bounded below (cf. \cite[Lemma 2.4]{LM1}). Fix $\epsilon>0$. By Proposition \ref{GWT}, there exists a smooth Hermitian metric $\widetilde{H}$ on $K_X^{-1}\rightarrow X$ such that $e^{-\epsilon}H \leq \widetilde{H}\leq e^\epsilon H$ and the Ricci curvature of $\widetilde{H}$ is bounded below. From inequality (\ref{s}), we obtain
\begin{align}
\inum\Lambda_{\widetilde{H}}\partial\bar{\partial} \sum_{j=1}^rs_j\geq \sum_{j=1}^rs_{j-1}^{-1}(s_{j-1}-s_j)^2H_{j-1}\otimes \widetilde{H}^{-1}\geq \sum_{j=1}^{r-1}s_j^{-1}(s_j-s_{j+1})^2H_j\otimes \widetilde{H}^{-1}. \label{sH}
\end{align}
Applying Proposition \ref{wstoc} to inequality (\ref{sH}), we obtain, for each $m\in\Z_{>0}$, a point $x_m\in X$ such that 
\begin{align*}
&(\sum_{j=1}^rs_j)(x_m)\geq \sup_X(\sum_{j=1}^rs_j)-1/m, \\
&\sum_{j=1}^{r-1}(s_j^{-1}(s_j-s_{j+1})^2)(x_m)(H_j\otimes \widetilde{H}^{-1})(x_m)\leq 1/m.
\end{align*}
Around each $x_m$, we choose a local coordinate $z$ such that $\widetilde{H}(\frac{\partial}{\partial z}, \frac{\partial }{\partial z})(x_m)=1$. 
Since $H_1,\dots, H_{r-1}, \widetilde{H}$ are mutually bounded, there exists a constant $B>0$ such that $H_j(\frac{\partial}{\partial z}, \frac{\partial}{\partial z})(x_m)^{-1}\leq B$ for all $j=1,\dots, r-1$ and all $m\in \Z_{>0}$. By Lemma \ref{lem4}, there exists a positive constant $C_1>0$ such that
\begin{align*}
\sum_{j=1}^r|s_j(x_m)-1|\leq C_1/m^{1/2}
\end{align*}
for all $m\in \Z_{>0}$. This implies that $\sup_X\sum_{j=1}^rs_j\leq r$. On the other hand, since $s_1\cdots s_r=1$, we have $\sum_{j=1}^rs_j\geq r$. Therefore, $\sum_{j=1}^r s_j = r$ on $X$. By the arithmetic mean--geometric mean inequality, we conclude that $s_1=\cdots=s_r=1$. Hence $h=h^\prime$.
\end{proof}
\subsection{Pullback}\label{pb}
Extending \cite[Section 3.4.7]{LM1}, this section establishes some properties of pullbacks of solutions to equation (\ref{phi}) by holomorphic maps. Let $\pi:Y\rightarrow X$ be a  holomorphic map whose derivative is nowhere vanishing on $Y$. Since the derivative is nowhere vanishing on $Y$, we have a natural isomorphism $\pi^\ast K_X\simeq K_Y$. Choose a square root $K_X^{1/2}$ of the holomorphic vector bundle $K_X\rightarrow X$. Since we have
$\pi^\ast K_X^{1/2}\otimes \pi^\ast K_X^{1/2}\simeq \pi^\ast K_X\simeq K_Y$, by setting $K_Y^{1/2}\coloneqq \pi^\ast K_X^{1/2}$, there is a natural isomorphism
\begin{align*}
\pi^\ast\left(\K_{X,r}\right)\simeq \K_{Y,r},
\end{align*}
where we have used the notation $\K_{X,r}\coloneqq \bigoplus_{j=1}^rK_X^{\frac{r-(2j-1)}{2}}$ and $\K_{Y,r}\coloneqq \bigoplus_{j=1}^rK_Y^{\frac{r-(2j-1)}{2}}$ to explicitly denote the underlying Riemann surfaces. This notation is adopted solely for this subsection and Section \ref{proof of mc}.
For each diagonal Hermitian metric $h=(h_1,\dots, h_r)$ on $\K_{X,r}$, we denote by $\pi^\ast h=(\pi^\ast h_1,\dots, \pi^\ast h_r)$ the Hermitian metric on $\K_{Y,r}$ induced by the above isomorphism. 
Let $\varphi$ be a subharmonic weight function on $K_X$. We denote by $\pi^\ast\varphi$ the induced weight function on $K_Y$. Let $H_1, \dots, H_r$ be the Hermitian metrics constructed from a diagonal Hermitian metric $h=(h_1,\dots, h_r)$ and the semipositive metric $e^{-\varphi}h_\refe$ as in Section \ref{2.2}. Pulling back equation (\ref{phi}) for $h$ by $\pi$, we obtain
\begin{align*}
\inum F_{\pi^\ast h_j} + \vol(\pi^\ast H_{j-1}) - \vol(\pi^\ast H_j) = 0 \ \text{for $j=1,\dots, r-1$}.
\end{align*}
This implies
\begin{prop}\label{hsolves}{\it If $h$ solves equation (\ref{phi}) associated to $\varphi$, then $\pi^\ast h$ solves equation (\ref{phi}) associated to $\pi^\ast\varphi$. If $\pi$ is surjective, the converse also holds.}
\end{prop} 
We next discuss completeness. By Ambrose's theorem \cite{Amb1} (see also \cite{Pet1}), a local isometry from a complete Riemannian manifold to a connected Riemannian manifold is a surjective covering map, and the target manifold is complete. Conversely, completeness is preserved under Riemannian coverings. These facts also apply to pullbacks of continuous Hermitian metrics on Riemann surfaces: for each continuous metric $g$ on $X$, choose a smooth metric $\widetilde{g}$ satisfying $\frac{1}{2}g\leq\widetilde{g}\leq2g$. These comparisons preserve completeness and remain valid after pullback. Applying this criterion to the metrics induced by $H_1,\ldots,H_{r-1}$ yields the following:
\begin{prop}\label{pistarphi}{\it For each diagonal Hermitian metric $h$ on $\K_{X,r}$, the pullback $\pi^\ast h$ is complete if and only if $\pi$ is a surjective covering map and $h$ is complete. In particular, if $\pi$ is a surjective covering map and $h$ is a solution to equation (\ref{phi}) associated to $\varphi$, then $h$ is a complete solution if and only if $\pi^\ast h$ is a complete solution associated to $\pi^\ast\varphi$.  
}
\end{prop}
\subsection{Supersolution and subsolution to equation (\ref{phi})}
In this subsection, we establish some properties of subsolutions and supersolutions to equation (\ref{phi}). 
Let $e^{-\varphi}h_\refe$ be a semipositive singular Hermitian metric on $K_X\rightarrow X$ with a smooth reference metric $h_\refe$. For each $j=1,\dots, r$, we denote by $h_{\refe,j}\coloneqq h_\refe^{\frac{r-(2j-1)}{2}}$ the real diagonal Hermitian metric on $K_X^{\frac{r-(2j-1)}{2}}$ induced by $h_\refe$. 
Let $h = (h_1, \dots, h_r)=(e^{-\phi_1}h_{\refe,1}, e^{-\phi_2}h_{\refe,2}, \dots, e^{-\phi_r}h_{\refe,r})$ be a real diagonal Hermitian metric on $\K_r \rightarrow X$. We set $H_j\coloneqq h_j^{-1}\otimes h_{j+1}$ for each $j=1,\dots, r-1$ and $H_0=H_r\coloneqq e^{r\varphi}h_\refe^{-r}\otimes h_r^{-1}\otimes h_1$. We suppose that $\phi_1,\dots, \phi_n, H_0,\dots, H_n$ belong to $L^1_\loc$. The notions of supersolutions and subsolutions to equation (\ref{phi}) are then defined as follows (see also \cite[Definition 5.1]{LM1}):
\begin{defi}We say that $h$ is a {\it supersolution} (resp. {\it subsolution}) to equation \eqref{phi} if it satisfies the following differential inequalities in the weak sense:
\begin{align*}
\inum F_{h_j} + \vol(H_{j-1}) - \vol(H_j) \geq 0 \quad (\text{resp. } \leq 0) \quad \text{for all } j = 1, \dots, n.
\end{align*}
\end{defi}
\begin{ex}\label{H1Hr} Define a diagonal Hermitian metric $h^{(\varphi)}\coloneqq (h_1^{(\varphi)},\dots, h_r^{(\varphi)})$ by $h_j^{(\varphi)}\coloneqq (e^{-\varphi}h_\refe)^{\frac{r-(2j-1)}{2}}$ for each $j=1,\dots, r$. For each $j=1,\dots, r-1$, define $H_j^{(\varphi)}\coloneqq h_j^{(\varphi)-1}\otimes h_{j+1}^{(\varphi)}$ and define $H_r^{(\varphi)}\coloneqq (e^{-\varphi}h_\refe)^{-r}\otimes H_1^{(\varphi)-1}\otimes \cdots \otimes H_{r-1}^{(\varphi)-1}$. Then $H_1^{(\varphi)}=\cdots =H_r^{(\varphi)}=e^{\varphi}h_\refe^{-1}$.  Since each $h_j^{(\varphi)},$ $j=1,\dots, n$, has semipositive curvature, $h^{(\varphi)}$ is a supersolution to equation (\ref{phi}).
\end{ex}

\begin{ex}\label{subandsuper} Let $h_+$ and $h_-$ be real diagonal Hermitian metrics such that $h_+$ is a supersolution and $h_-$ is a subsolution to equation (\ref{phi}) associated with $\varphi$. Suppose that $\varphi \leq \varphi^\prime$. Then $h_+$ is also a supersolution to equation (\ref{phi}) associated with $\varphi^\prime$. Similarly, if $\varphi \geq \varphi^\prime$, then $h_-$ is also a subsolution to equation (\ref{phi}) associated with $\varphi^\prime$.
\end{ex}
\begin{ex}\label{genkidase} Let $h_+=(h_{+,1},\dots, h_{+,r})$ be a real diagonal supersolution to equation (\ref{phi}), and let $h_-=(h_{-,1},\dots, h_{-,r})$ be a real diagonal subsolution to equation (\ref{phi}). Let $\phi_+$ be a nonpositive bounded subharmonic function, and let $\phi_-$ be a nonnegative bounded superharmonic function. For each $j=1,\dots, n$, define
\begin{align*}
&\widetilde{h}_{+,j}\coloneqq e^{-\phi_+}h_{+,j}, \\
&\widetilde{h}_{-,j}\coloneqq e^{-\phi_-}h_{-,j}.
\end{align*}
For the remaining indices $j=n+1,\dots, r$, we define $\widetilde{h}_{+,j}$ and $\widetilde{h}_{-,j}$ so that $\widetilde{h}_{+}\coloneqq (\widetilde{h}_{+,1},\dots, \widetilde{h}_{+,r})$ and $\widetilde{h}_{-}\coloneqq (\widetilde{h}_{-,1},\dots, \widetilde{h}_{-,r})$ are real diagonal Hermitian metrics. Then $\widetilde{h}+$ and $\widetilde{h}-$ are a supersolution and a subsolution to equation (\ref{phi}), respectively (cf.~\cite[Proof of Proposition 5.9]{LM1}). 
\end{ex}
For the proof of Proposition \ref{min-max}, following the argument in
\cite[Proof of Lemma 1.5]{DS1}, we first establish the following lemma.
\begin{lemm}\label{uvFG}{\it Let $U\subseteq \R^d$ be an open subset, and let $F, G:U\times \R\rightarrow \R$ be Borel measurable functions. We denote by $g_E$ the Euclidean metric and by $\Delta_{g_E}\coloneqq -d^\ast d$ the Laplacian. Then the following holds:
\begin{enumerate}[(i)]
\item Let $u_0,u_1:U\rightarrow \R$ be locally integrable functions. Suppose that $F(\cdot, u_i(\cdot))$ is locally integrable for $i=0,1$, and $u_0$ and $u_1$ satisfy the following in the sense of distributions for $i=0,1$:
\begin{align*}
\Delta_{g_E}u_i \leq F(\cdot ,u_i(\cdot))  
\end{align*}
Then $u_\min\coloneqq\min\{u_0,u_1\}$ and $F(\cdot, u_\min(\cdot))$ are also locally integrable functions and satisfy 
\begin{align}
\Delta_{g_E}u_\min\leq F(\cdot,u_\min(\cdot)) \label{umin}
\end{align}
\item Let $v_0,v_1:U\rightarrow \R$ be locally integrable functions. Suppose that $G(\cdot, v_i(\cdot))$ is locally integrable for $i=0,1$, and $v_0$ and $v_1$ satisfy the following in the sense of distributions for $i=0,1$:
\begin{align*}
\Delta_{g_E}v_i\geq G(\cdot,v_i(\cdot))
\end{align*}
Then $v_\max\coloneqq\max\{v_0,v_1\}$ and $G(\cdot,v_\max(\cdot))$ are also locally integrable functions and satisfy
\begin{align*}
\Delta_{g_E}v_\max\geq G(\cdot,v_\max(\cdot)).
\end{align*}
\end{enumerate}
}
\end{lemm}
\begin{proof} 
We first note that assertion (ii) follows immediately from assertion (i).
Indeed, assuming assertion (i), set
\begin{align*}
u_0\coloneqq -v_0,\ u_1\coloneqq -v_1,\ F(x,t)\coloneqq -G(x,-t).
\end{align*}
Then assertion (ii) follows by applying assertion (i) to $u_0$ and $u_1$.
Therefore, it suffices to prove assertion (i). We begin by proving $u_\min$ and $F(\cdot,u_\min(\cdot))$ are both locally integrable. Indeed, since 
\begin{align*}
|u_{\min}|\leq |u_0|+|u_1|,
\end{align*}
we have $u_\min\in L^1_\loc(U)$. Also, since $u_\min(x)$ coincides with either $u_0(x)$ or $u_1(x)$
for each $x\in U$, we have 
\begin{align*}
|F(x,u_\min(x))|\leq |F(x,u_0(x))|+|F(x,u_1(x))|,
\end{align*}
which proves the local integrability of $F(\cdot, u_\min(\cdot))$. We next prove (\ref{umin}). Let $V\subseteq U$ be an arbitrary relatively compact open subset. Let $(\rho_\epsilon)_{\epsilon>0}$ be a standard mollifier. For each $\epsilon>0$ and each $i=0,1$, we set
\begin{align*}
&u_{i,\epsilon}\coloneqq u_i\ast\rho_\epsilon, \\
&f_{i,\epsilon}\coloneqq F(\cdot,u_i(\cdot))\ast \rho_\epsilon,
\end{align*}
where $\ast$ denotes the convolution. We take $\epsilon$ small enough so that $u_{i,\epsilon}$ and $f_{i,\epsilon}$ are defined on $V$. Note that, by the standard properties of mollification, for each
$i=0,1$, $u_{i,\epsilon}\rightarrow u_i$ and $f_{i,\epsilon}\rightarrow f_i$
in $L^1(V)$ as $\epsilon\searrow 0$, where
$f_i\coloneqq F(\cdot,u_i(\cdot))$. Moreover, $u_{i,\epsilon}(x)\rightarrow u_i(x),\ i=0,1$, at every common Lebesgue point $x\in V$ of $u_0$ and $u_1$, and hence
for almost every $x\in V$. Also, since $\Delta_{g_E}u_i\leq f_i$
in the sense of distributions and $\rho_\epsilon$ is nonnegative,
convolution with $\rho_\epsilon$ yields $\Delta_{g_E}u_{i,\epsilon}\leq f_{i,\epsilon}$
pointwise on $V$. We define
\begin{align*}
u_\epsilon\coloneqq \min\{u_{0,\epsilon}, u_{1,\epsilon}\}=\frac{1}{2}\left(u_{0,\epsilon}+u_{1,\epsilon}-|u_{0,\epsilon}-u_{1,\epsilon}|\right).
\end{align*}
Applying the inequality $|\min\{a,b\}-\min\{c,d\}|\leq |a-c|+|b-d|$ for arbitrary real numbers $a,b,c,d$, we obtain
\begin{align*}
|u_\epsilon-u_\min|\leq |u_{0,\epsilon}-u_0|+|u_{1,\epsilon}-u_1|
\end{align*}
and hence, $u_\epsilon$ converges to $u_\min$ in $L^1(V)$. Set
\begin{align*}
w_\epsilon\coloneqq u_{0,\epsilon}-u_{1,\epsilon}.
\end{align*}
By Kato's inequality,
\begin{align*}
\Delta_{g_E}|w_\epsilon|
\geq
\sgn(w_\epsilon)\Delta_{g_E}w_\epsilon
\end{align*}
in the sense of distributions on $V$, where $\sgn(t)\ (t\in\R)$ is defined by
\begin{align*}
\sgn(t)&\coloneqq 
\begin{cases}
\frac{t}{|t|}\ &\text{if $t\neq 0$} \\
0\ &\text{if $t=0$}.
\end{cases}
\end{align*}
Hence,
\begin{align*}
\Delta_{g_E}u_\epsilon
\leq
\alpha_\epsilon\Delta_{g_E}u_{0,\epsilon}
+(1-\alpha_\epsilon)\Delta_{g_E}u_{1,\epsilon} 
\end{align*}
in the sense of distributions on $V$, where $\alpha_\epsilon\coloneqq
\frac{1}{2}(1-\sgn(w_\epsilon))$.
On the other hand, since
\begin{align*}
\Delta_{g_E}u_{i,\epsilon}\leq f_{i,\epsilon}
\end{align*}
pointwise on $V$ and $0\leq\alpha_\epsilon\leq1$, we have
\begin{align*}
\alpha_\epsilon\Delta_{g_E}u_{0,\epsilon}
+(1-\alpha_\epsilon)\Delta_{g_E}u_{1,\epsilon}
\leq
\alpha_\epsilon f_{0,\epsilon}
+(1-\alpha_\epsilon)f_{1,\epsilon}
\end{align*}
everywhere on $V$. Therefore,
\begin{align}
\Delta_{g_E}u_\epsilon
\leq
\alpha_\epsilon f_{0,\epsilon}
+(1-\alpha_\epsilon)f_{1,\epsilon} \label{alpha01}
\end{align}
in the sense of distributions on $V$.
We prove that $\alpha_\epsilon f_{0,\epsilon}+(1-\alpha_\epsilon)f_{1,\epsilon}$ converges to $F(\cdot, u_\min(\cdot))$ in $L^1(V)$. A simple calculation yields
\begin{align*}
\alpha_\epsilon f_{0,\epsilon}+(1-\alpha_\epsilon)f_{1,\epsilon}=\alpha_\epsilon(f_{0,\epsilon}-f_0)+(1-\alpha_\epsilon)(f_{1,\epsilon}-f_1)+\alpha_\epsilon f_0+(1-\alpha_\epsilon)f_1.
\end{align*}
Since $0\leq\alpha_\epsilon\leq1$ and
$f_{i,\epsilon}\to f_i$ in $L^1(V)$ for $i=0,1$, the first two
terms converge to $0$ in $L^1(V)$. Therefore, it is enough to prove that $g_\epsilon\coloneqq \alpha_\epsilon f_0+(1-\alpha_\epsilon)f_1$ converges to $F(\cdot,u_\min(\cdot))$ in $L^1(V)$. Let $x\in V$ be a common Lebesgue point of $u_0$ and $u_1$.  If $u_0(x)>u_1(x)$, then we obtain $\alpha_\epsilon(x)=0$ for sufficiently small $\epsilon$ and thus $g_\epsilon(x)$ converges to $f_1(x)=F(x,u_\min(x))$. Similarly, if $u_0(x)<u_1(x)$, then $g_\epsilon(x)$ converges to $f_0(x)=F(x,u_\min(x))$. Finally, if $u_0(x)=u_1(x)$, then 
\begin{align*}
g_\epsilon(x)&=\alpha_\epsilon(x)f_0(x)+(1-\alpha_\epsilon(x))f_1(x) \\
&=\alpha_\epsilon(x)f_0(x)+(1-\alpha_\epsilon(x))f_0(x) \\
&=f_0(x) \\
&=F(x,u_\min(x)).
\end{align*}
 Therefore, we conclude that $g_\epsilon(x)$ converges to $F(x,u_\min(x))$ in every case, and hence, $g_\epsilon$ converges to $F(\cdot, u_\min(\cdot))$ almost everywhere. Moreover, we have
\begin{align*}
|g_\epsilon|\leq |f_0|+|f_1|.
\end{align*}
Therefore, by the dominated convergence theorem, $g_\epsilon$ converges to $F(\cdot, u_\min(\cdot))$ in $L^1(V)$. Let $\phi:U\rightarrow \R_{\geq 0}$ be a nonnegative smooth function supported in $V$. From (\ref{alpha01}), we obtain
\begin{align*}
\int_Vu_\epsilon\Delta_{g_E}\phi\leq \int_V (\alpha_\epsilon f_{0,\epsilon}+(1-\alpha_\epsilon)f_{1,\epsilon})\phi.
\end{align*}
By taking the limit $\epsilon\searrow 0$, we have
\begin{align*}
\int_Vu_\min\Delta_{g_E}\phi\leq \int_V F(\cdot, u_\min(\cdot))\phi.
\end{align*}
Since $V$ is an arbitrary relatively compact open subset, this proves the assertion. 
\end{proof}
\begin{rem}
In Lemma \ref{uvFG}, the inequalities are understood in the sense of
distributions. In particular, we do not assume that the functions under
consideration belong to $W^{1,2}_{\loc}(U)$. We emphasize this point
because, in some literature on Poisson equations, weak solutions are
defined from the outset as functions belonging to
$W^{1,2}_{\loc}(U)$.
\end{rem}
The following proposition is an extension of \cite[Lemma 5.4]{LM1}:
\begin{prop}\label{min-max} {\it Let $h_+=(h_{+,1},\dots, h_{+,r})$ and $h_+^\prime=(h_{+,1}^\prime, \dots, h_{+,r}^\prime)$ be two supersolutions to equation (\ref{phi}), and let $h_-=(h_{-,1},\dots, h_{-,r})$ and $h_-^\prime=(h_{-,1}^\prime, \dots, h_{-,r}^\prime)$ be two subsolutions to equation (\ref{phi}). For each $j=1,\dots, n$, define
\begin{align*} 
\widetilde{h}_{+,j}&\coloneqq \min\{h_{+,j}, h_{+,j}^\prime\}, \\
\widetilde{h}_{-,j}&\coloneqq \max\{h_{-,j}, h_{-,j}^\prime\}. 
\end{align*}
For the remaining indices $j=n+1,\dots, r$, we define $\widetilde{h}_{+,j}$ and $\widetilde{h}_{-,j}$ so that $\widetilde{h}_{+}\coloneqq (\widetilde{h}_{+,1},\dots, \widetilde{h}_{+,r})$ and $\widetilde{h}_{-}\coloneqq (\widetilde{h}_{-,1},\dots, \widetilde{h}_{-,r})$ are real diagonal Hermitian metrics. Suppose that for each $j=1,\dots, n$, $\widetilde{h}_{+,j}^{-1}\otimes \widetilde{h}_{+,j+1}$ and $\widetilde{h}_{-,j}^{-1}\otimes \widetilde{h}_{-,j+1}$ belong to $L^1_\loc$. Then $\widetilde{h}_+$ and $\widetilde{h}_-$ are a supersolution and a subsolution to equation (\ref{phi}), respectively.}
\end{prop}
\begin{proof} The proof is based on the argument in \cite[Proof of Lemma 5.4]{LM1}. For simplicity, we only prove that $\widetilde{h}_+$ is a supersolution since the proof for showing that $\widetilde{h}_-$ is a subsolution is similar. Since the assertion is local, we can assume from the outset that $X$ is the unit disc. Let $g_E$ be the Euclidean metric of the unit disc. Let $h_\refe$ be the reference metric on $K_X$ induced by the Euclidean metric and denote by $h_{\refe,j}\coloneqq h_\refe^{\frac{r-(2j-1)}{2}}$ the Hermitian metric on $K_X^{\frac{r-(2j-1)}{2}}$ induced by $h_\refe$. Using $h_{\refe,j}$, we denote the $j$-th components of $h_+$ and $h_+^\prime$ as follows: 
\begin{align*}
&h_{+,j}=e^{-\phi_{+,j}}h_{\refe,j}, \\
&h_{+,j}^\prime=e^{-\phi_{+,j}^\prime}h_{\refe,j}.
\end{align*}
We also set $\hat{\phi}_{+,j}\coloneqq \max\{\phi_{+,j},\phi_{+,j}^\prime\}$ for each $j=1,\dots, n$. Then by the definition of supersolutions, we obtain
\begin{align}
\begin{cases}
\frac{1}{2}\Delta_{g_E}\phi_{+,j}+e^{\phi_{+,j-1}-\phi_{+,j}}-e^{\phi_{+, j}-\phi_{+,j+1}}\geq 0 \ \text{for $j=2,\dots, n-1$}\\
\frac{1}{2}\Delta_{g_E}\phi_{+,n}+e^{\phi_{+,n-1}-\phi_{+,n}}-e^{(2n+2-r)\phi_{+, n}}\geq 0 \\
\frac{1}{2}\Delta_{g_E}\phi_{+,1}+e^{r\varphi-2\phi_{+,1}}-e^{\phi_{+, 1}-\phi_{+,2}}\geq 0, 
\end{cases}
\label{phi+}\\
\begin{cases}
\frac{1}{2}\Delta_{g_E}\phi_{+,j}^\prime+e^{\phi_{+,j-1}^\prime-\phi_{+,j}^\prime}-e^{\phi_{+, j}^\prime-\phi_{+,j+1}^\prime}\geq 0 \ \text{for $j=2,\dots, n-1$}\\
\frac{1}{2}\Delta_{g_E}\phi_{+,n}^\prime+e^{\phi_{+,n-1}^\prime-\phi_{+,n}^\prime}-e^{(2n+2-r)\phi_{+, n}^\prime}\geq 0 \\
\frac{1}{2}\Delta_{g_E}\phi_{+,1}^\prime+e^{r\varphi-2\phi_{+,1}^\prime}-e^{\phi_{+, 1}^\prime-\phi_{+,2}^\prime}\geq 0.
\end{cases}
\label{phi+p}
\end{align}
Strictly speaking, the case $n=1$ requires separate consideration; however, since the following discussion can easily be modified or extended to cover that case, we assume $n \geq 2$ for simplicity. For each $j=2,\dots,n-1$, we set
\begin{align*}
&G_j(x,t)\coloneqq  -2(e^{\hat{\phi}_{+,j-1}-t}-e^{t-\hat{\phi}_{+,j+1}}),\\
&G_j^\#(x,t)\coloneqq \max\{G_j(x,t), B_j(x)\},
\end{align*}
where $B_j$ is defined by
\begin{align}
B_j\coloneqq -2|e^{\phi_{+,j-1}-\phi_{+,j}}-e^{\phi_{+, j}-\phi_{+,j+1}}|-2|e^{\phi_{+,j-1}^\prime-\phi_{+,j}^\prime}-e^{\phi_{+, j}^\prime-\phi_{+,j+1}^\prime}|.\label{defBj}
\end{align}
Clearly, 
\begin{align}
&G_j(x,t)\leq G_j^\#(x,t), \notag\\
&G_j(x,\phi_{+,j})\leq -2(e^{\phi_{+,j-1}-\phi_{+,j}}-e^{\phi_{+, j}-\phi_{+,j+1}}), \label{random}\\
&G_j(x,\phi_{+,j}^\prime)\leq -2(e^{\phi_{+,j-1}^\prime-\phi_{+,j}^\prime}-e^{\phi_{+, j}^\prime-\phi_{+,j+1}^\prime})\label{random2}. 
\end{align}
Moreover, from (\ref{defBj}), (\ref{random}), and (\ref{random2}), we deduce
\begin{align}
&G_j^\#(x,\phi_{+,j})\leq -2(e^{\phi_{+,j-1}-\phi_{+,j}}-e^{\phi_{+, j}-\phi_{+,j+1}}), \label{Random}\\
&G_j^\#(x,\phi_{+,j}^\prime)\leq -2(e^{\phi_{+,j-1}^\prime-\phi_{+,j}^\prime}-e^{\phi_{+, j}^\prime-\phi^\prime_{+,j+1}})\label{Random2}. 
\end{align}
This in particular implies that $G_j^\#(x,\phi_{+,j})$ and 
$G_j^\#(x,\phi_{+,j}^\prime)$ are locally integrable, since they are bounded below by a locally integrable function $B_j$.
From (\ref{phi+}), (\ref{phi+p}), (\ref{Random}), and (\ref{Random2}), we obtain
\begin{align*}
&\Delta_{g_E}\phi_{+,j}\geq G_j^\#(x,\phi_{+,j}), \\
&\Delta_{g_E}\phi_{+,j}^\prime\geq G_j^\#(x,\phi_{+,j}^\prime).
\end{align*}
By Lemma \ref{uvFG} (ii), we have
\begin{align*}
\Delta_{g_E}\hat{\phi}_{+,j}\geq G_j^\#(x,\hat{\phi}_{+,j})\geq G_j(x,\hat{\phi}_{+,j}).
\end{align*}
The above discussion can also be easily extended to the cases $j=1,n$. Since $\widetilde{h}_{+,j}=\min\{h_{+,j}, h_{+,j}^\prime\}=e^{-\hat{\phi}_{+,j}}h_{\refe,j}$ for each $j=1,\dots, n$, this implies the claim. 
\end{proof}
\begin{rem}\label{comparg}
Following the comparison argument in \cite[Proof of Lemma 5.4]{LM1}, inequalities (\ref{phi+}) and (\ref{phi+p}) also yield the following (\ref{phi++}) and (\ref{phi+p+}):
\begin{align}
&\frac{1}{2}\Delta_{g_E}\phi_{+,j}+e^{\hat{\phi}_{+,j-1}-\phi_{+,j}}-e^{\phi_{+, j}-\hat{\phi}_{+,j+1}}\geq 0, \label{phi++}\\
&\frac{1}{2}\Delta_{g_E}\phi_{+,j}^\prime+e^{\hat{\phi}_{+,j-1}-\phi_{+,j}^\prime}-e^{\phi_{+, j}^\prime-\hat{\phi}_{+,j+1}}\geq 0,  \label{phi+p+}
\end{align}
for each $j=2,\dots, n-1$. In the present setting, these inequalities are interpreted after testing against nonnegative smooth compactly supported functions, allowing the left-hand sides to take the value $+\infty$. This interpretation is well defined because the exponential terms with a minus sign are locally integrable, whereas those with a plus sign need not be. However, the local integrability required by Lemma \ref{uvFG} is not guaranteed, so the lemma cannot be applied directly. The truncation argument above, based on (\ref{phi+}) and (\ref{phi+p}), provides inequalities with locally integrable right-hand sides to which Lemma \ref{uvFG} applies.
\end{rem}
The following lemma is an extension of \cite[Proposition 4.6]{LM1}:
\begin{lemm}\label{cands}{\it Suppose that $e^\varphi$ is continuous. Let $X$ be a Riemann surface with a smooth compact boundary $\partial X$, possibly empty. 
Let $h=(h_1,\dots, h_r)$ be a continuous real diagonal complete subsolution to equation (\ref{phi}), and let $h^\prime=(h_1^\prime,\dots, h_r^\prime)$ be a continuous real diagonal supersolution to equation (\ref{phi}). For each $j=1,\dots, r-1$, set $H_j\coloneqq h_j^{-1}\otimes h_{j+1}, \ H_j^\prime\coloneqq h_j^{\prime-1}\otimes h_{j+1}^\prime$ and $H_0=H_r\coloneqq h_r^{-1}\otimes h_1\otimes (e^{-\varphi}h_\refe)^{-r}, \ H_0^\prime=H_r^\prime\coloneqq h_r^{\prime-1}\otimes h_1^\prime\otimes (e^{-\varphi}h_\refe)^{-r}$. Suppose that the curvatures of $H_1,\dots, H_n$ are bounded below and that $H_{j-1}\otimes H_j^{-1}$ is bounded for every $j=2,\dots, n$. Suppose moreover that $\log(h_j \otimes h_j^{\prime-1})$ is bounded above and belongs to $W^{1,2}_{\loc}(X\setminus \partial X)$ for each $j=1,\dots, n$. Then one of the following holds:
\begin{enumerate}[(i)]
\item
\begin{align*}
\max_{1\leq j\leq n}\sup_X(h_j\otimes h_j^{\prime-1})\leq 1;
\end{align*}
\item The boundary $\partial X$ is nonempty and 
\begin{align*}
\max_{1\leq j\leq n}\sup_X(h_j\otimes h_j^{\prime-1})\leq \max_{1\leq j\leq n}\max_{\partial X}(h_j\otimes h_j^{\prime-1}).
\end{align*}
\end{enumerate}
Moreover, if $\log(h_j\otimes h_j^{\prime-1})$ is of class $C^2$ for all $j=1,\dots, n$, and if (ii) does not occur, then either $h_j\otimes h_j^{\prime-1}<1$ for all $j=1,\dots, n$ or $h_j= h_j^{\prime}$ for all $j=1,\dots, n$. 
}
\end{lemm}
\begin{proof}
For each $j=1,\dots, n$, we define $\xi_j\coloneqq \log(h_j\otimes h_j^{\prime-1})$, $M_j\coloneqq \sup_Xe^{\xi_j}$, $M\coloneqq \max_{1\leq j\leq n}M_j$.  Assume that (ii) does not occur. 
A direct calculation yields
\begin{align}
\inum\partial \bar{\partial} \xi_j\geq \vol(H_{j-1})-\vol(H_j)-\vol(H_{j-1}^\prime)+\vol(H_j^\prime) \ \text{for $j=1,\dots, n$}. \label{xij}
\end{align}
Suppose that $n=1$. From (\ref{xij}), we have
\begin{align}
\inum \Lambda_{H_1}\partial \bar{\partial} \xi_1&\geq (1-e^{-2\xi_1})H_0\otimes H_1^{-1}-(1-e^{(4-r)\xi_1}). \label{xi01}
\end{align}
By assumption, the curvature of $H_1$ is bounded below. Fix $\epsilon>0$. From Proposition \ref{GWT}, there exists a smooth Hermitian metric $\widetilde{H}_1$ on $K_X^{-1}\rightarrow X\setminus\partial X$ such that 
\begin{align}
e^{-\epsilon} H_1 \le \widetilde{H}_1 \le e^{\epsilon} H_1\label{label68}
\end{align}
and the curvature of $\widetilde{H}_1$ is bounded below. If $\partial X\neq\emptyset$, we choose a compact neighbourhood $K$ of $\partial X$ satisfying $\max_K\xi_1<\sup_X\xi_1$, together with a diffeomorphism $\Psi:K\to\partial X\times[0,1]$ such that $\Psi(\partial X)=\partial X\times\{0\}$. Set $Y\coloneqq X\setminus\Psi^{-1}(\partial X\times[0,1))$. Then $\partial Y=\Psi^{-1}(\partial X\times\{1\})\subseteq X\setminus\partial X$, and $\sup_Y\xi_1=\sup_X\xi_1$. We apply Proposition \ref{wstocb} on $Y$ with the metric $2g_{\widetilde{H_1}}$, where $g_{\widetilde{H_1}}$ denotes the K\"ahler metric induced by $\widetilde H_1$. This metric is smooth up to $\partial Y$ and complete by (\ref{label68}), with Ricci curvature bounded below.
From (\ref{xi01}), we have
\begin{align}
\inum \Lambda_{\widetilde{H}_1}\partial \bar{\partial} \xi_1&\geq (1-e^{-2\xi_1})H_0\otimes \widetilde{H}_1^{-1}-(1-e^{(4-r)\xi_1}) H_1\otimes \widetilde{H}_1^{-1}. \label{xi01*}
\end{align}
By assumption, $\xi_1$ is bounded above, continuous, and belongs to $W^{1,2}_{\loc}$. Moreover, the right-hand side of inequality (\ref{xi01*}) is continuous, and, by assumption, $\sup_X\xi_1>\max_{\partial X}\xi_1$ if $\partial X$ is nonempty. Therefore, by applying Propositions \ref{wstoc} and \ref{wstocb} to (\ref{xi01*}), we obtain, for each $m\in \Z_{>0}$, a point $x_m\in X$ such that
\begin{align}
\xi_1(x_m)&\geq \sup_X\xi_1-1/m, \label{mmm}\\
1/m&\geq (1-e^{-2\xi_1(x_m)})(H_0\otimes \widetilde{H}_1^{-1})(x_m)-(1-e^{(4-r)\xi_1(x_m)})(H_1\otimes \widetilde{H}_1^{-1})(x_m). \label{1mxm}
\end{align}
Suppose that $\sup_X\xi_1>0$. From (\ref{mmm}), for all sufficiently large $m$, we have $\xi_1(x_m)\geq \sup_X\xi_1-1/m>0$. Therefore, for all sufficiently large $m$, inequality (\ref{1mxm}) yields
\begin{align*}
1/m&\geq (1-e^{-2\xi_1(x_m)})(H_0\otimes \widetilde{H}_1^{-1})(x_m)-(1-e^{(4-r)\xi_1(x_m)})(H_1\otimes \widetilde{H}_1^{-1})(x_m) \\
&\geq -(1-e^{(4-r)(\sup_X\xi_1-1/m)})e^{-\epsilon}.
\end{align*}
Taking the limit inferior as $m\to\infty$, we obtain $\sup_X\xi_1\leq 0$, contradicting the assumption $\sup_X\xi_1>0$. Therefore we obtain $\sup_X\xi_1\leq 0$, and thus we have assertion (i) for the case where $n=1$. Suppose that $\xi_1$ is of class $C^2$ and (ii) does not occur. Then inequality (\ref{xi01}), the strong maximum principle, and the assumption that (ii) does not occur immediately yield $\xi_1<0$ or $\xi_1=0$. We next consider the case where $n\geq 2$. From (\ref{xij}), we have
\begin{align}
\inum\Lambda_{H_1}\partial \bar{\partial} \xi_1&\geq (1-e^{-2\xi_1})H_0\otimes H_1^{-1}-(1-e^{\xi_1-\xi_2}), \label{xi12}\\
\inum\Lambda_{H_j} \partial \bar{\partial} \xi_j &\geq (1-e^{\xi_{j-1}-\xi_j})H_{j-1}\otimes H_j^{-1}-(1-e^{\xi_j-\xi_{j+1}}) \ \text{for $j=2,\dots, n-1$}, \label{xi34}\\
\inum \Lambda_{H_n}\partial \bar{\partial} \xi_n&\geq (1-e^{\xi_{n-1}-\xi_n})H_{n-1}\otimes H_n^{-1}-(1-e^{(2n+2-r)\xi_n}). \label{xin}
\end{align}
By assumption, the curvatures of $H_1,\dots, H_n$ are bounded below. Fix $\epsilon>0$. From Proposition \ref{GWT}, there exist smooth Hermitian metrics $\widetilde{H}_1,\dots, \widetilde{H}_n$ on $K_X^{-1}\rightarrow X\setminus \partial X$ such that 
\begin{align}
e^{-\epsilon} H_j \le \widetilde{H}_j \le e^{\epsilon} H_j\ \text{for} \ j=1,\dots, n \label{label75}
\end{align}
and the curvatures of $\widetilde{H}_1,\dots, \widetilde{H}_n$ are bounded below. If $\partial X\neq\emptyset$, then $M>\max_{1\leq j\leq n}\max_{\partial X}e^{\xi_j}$. For each component with $M_j=M$, we therefore use the same construction as above, replacing $\xi_1$ and $\widetilde{H}_1$ by $\xi_j$ and $\widetilde{H}_j$, respectively, and using (\ref{label75}) in place of (\ref{label68}). From inequalities (\ref{xi12}), (\ref{xi34}), and (\ref{xin}), we have
\begin{align}
\inum\Lambda_{\widetilde{H}_1}\partial \bar{\partial} \xi_1&\geq (1-e^{-2\xi_1})H_0\otimes \widetilde{H}_1^{-1}-(1-e^{\xi_1-\xi_2})H_1\otimes \widetilde{H}_1^{-1}, \label{xi12*}\\
\inum\Lambda_{\widetilde{H}_j} \partial \bar{\partial} \xi_j &\geq (1-e^{\xi_{j-1}-\xi_j})H_{j-1}\otimes \widetilde{H}_j^{-1}-(1-e^{\xi_j-\xi_{j+1}})H_j\otimes \widetilde{H}_j^{-1} \ \text{for $j=2,\dots, n-1$}, \label{xi34*}\\
\inum \Lambda_{\widetilde{H}_n}\partial \bar{\partial} \xi_n&\geq (1-e^{\xi_{n-1}-\xi_n})H_{n-1}\otimes \widetilde{H}_n^{-1}-(1-e^{(2n+2-r)\xi_n})H_n\otimes \widetilde{H}_n^{-1}. \label{xin*}
\end{align}
By assumption, $\xi_1,\dots, \xi_n$ are bounded above, continuous, and belong to $W^{1,2}_{\loc}$. Moreover, the right-hand sides of inequalities (\ref{xi12*}), (\ref{xi34*}), and (\ref{xin*}) are all continuous. We set  $k\coloneqq \max\{j\mid M_j=M\}$. Assume $M=M_k>1$ and derive a contradiction. 
Suppose that $k=1$. By applying Proposition \ref{wstoc} when the boundary is empty, and Proposition \ref{wstocb} together with the assumption that (ii) does not occur when the boundary is nonempty, to (\ref{xi12*}), we obtain, for each $m\in \Z_{>0}$, a point $x_m\in X$ such that
\begin{align}
\xi_1(x_m)&\geq \sup_X\xi_1-1/m, \label{mmm2}\\
\frac{1}{m}&\geq (1-e^{-2\xi_1(x_m)})(H_0\otimes \widetilde{H}_1^{-1})(x_m)-(1-e^{\xi_1(x_m)-\xi_2(x_m)})(H_1\otimes \widetilde{H}_1^{-1})(x_m) \notag \\
&\geq (1-e^{-2\xi_1(x_m)})(H_0\otimes \widetilde{H}_1^{-1})(x_m)-(1-M_2^{-1}e^{\xi_1(x_m)})(H_1\otimes \widetilde{H}_1^{-1})(x_m). \label{1mxm2}
\end{align}
By the assumptions $k=1$ and $M_k=M>1$, the first term is nonnegative for all sufficiently large $m$. Hence, from (\ref{label75}), (\ref{mmm2}), and (\ref{1mxm2}),  taking the limit inferior as $m\to\infty$, we obtain $M_1\leq M_2$, which contradicts the assumption $k=1$. Suppose that $2\leq k\leq n-1$. By applying Proposition \ref{wstoc} when the boundary is empty, and Proposition \ref{wstocb} together with the assumption that (ii) does not occur when the boundary is nonempty, to (\ref{xi34*}), we obtain, for each $m\in \Z_{>0}$, a point $x_m\in X$ such that
\begin{align*}
\xi_k(x_m)&\geq \sup_X\xi_k-1/m, \\
\frac{1}{m}&\geq (1-e^{\xi_{k-1}(x_m)-\xi_k(x_m)})(H_{k-1}\otimes \widetilde{H}_k^{-1})(x_m)-(1-e^{\xi_k(x_m)-\xi_{k+1}(x_m)})(H_k\otimes \widetilde{H}_k^{-1})(x_m). 
\end{align*}
By taking the limit inferior as $m\to\infty$, together with the assumption that $H_{k-1}\otimes H_k^{-1}$ is bounded and with (\ref{label75}), we obtain $M_k\leq M_{k+1}$, which contradicts the definition of $k$. Therefore, $k=n$.  By applying Proposition \ref{wstoc} when the boundary is empty, and Proposition \ref{wstocb} together with the assumption that (ii) does not occur when the boundary is nonempty, to (\ref{xin*}), we obtain, for each $m\in \Z_{>0}$, a point $x_m\in X$ such that
\begin{align*}
\xi_n(x_m)&\geq \sup_X\xi_n-1/m \\
\frac{1}{m}&\geq (1-e^{\xi_{n-1}(x_m)-\xi_n(x_m)})(H_{n-1}\otimes \widetilde{H}_n^{-1})(x_m)-(1-e^{(2n+2-r)\xi_n(x_m)})(H_n\otimes \widetilde{H}_n^{-1})(x_m).
\end{align*}
By taking the limit inferior as $m\to\infty$, together with the assumption $H_{k-1}\otimes H_k^{-1}$ is bounded and with (\ref{label75}), we obtain $M=M_n\leq 1$, which contradicts the assumption $M>1$. Therefore, we obtain $M\leq 1$, and thus assertion (i) follows. We prove the latter half of the statement. From (\ref{xi12}), (\ref{xi34}), and (\ref{xin}), we have
\begin{align}
\inum\Lambda_{\widetilde{H}_1}\partial \bar{\partial} \xi_1&\geq (1-e^{-2\xi_1})H_0\otimes \widetilde{H}_1^{-1}-(1-e^{\xi_1})H_1\otimes \widetilde{H}_1^{-1}, \label{xi12+}\\
\inum\Lambda_{\widetilde{H}_j} \partial \bar{\partial} \xi_j &\geq (1-e^{-\xi_j})H_{j-1}\otimes \widetilde{H}_j^{-1}-(1-e^{\xi_j})H_j\otimes \widetilde{H}_j^{-1} \ \text{for $j=2,\dots, n-1$}, \label{xi34+}\\
\inum \Lambda_{\widetilde{H}_n}\partial \bar{\partial} \xi_n&\geq (1-e^{-\xi_n})H_{n-1}\otimes \widetilde{H}_n^{-1}-(1-e^{(2n+2-r)\xi_n})H_n\otimes \widetilde{H}_n^{-1}. \label{xin+}
\end{align}
Suppose that $\xi_1,\dots, \xi_n$ are all of class $C^2$. Assume that (ii) does not occur. Then by this assumption and by the strong maximum principle applied to (\ref{xi12+}), (\ref{xi34+}), and (\ref{xin+}), we obtain either $\xi_j<0$ or $\xi_j=0$ for each $j=1,\dots, n$. If $\xi_1=0$, then from (\ref{xi12}) and $M_2\leq 1$, we have $\xi_2=0$. If $\xi_j=0$ for some $j=2,\dots, n-1$, then from (\ref{xi34}), $M_{j-1}\leq 1$, and $M_{j+1}\leq 1$, we have $\xi_{j-1}=\xi_{j+1}=0$. If $\xi_n=0$,  then from (\ref{xin}) and $M_{n-1}\leq 1$, we have $\xi_{n-1}=0$. We conclude that either $e^{\xi_j}<1$ for all $j=1,\dots, n$, or $\xi_j=0$ for all $j=1,\dots, n$.
\end{proof}

At the end of this subsection, we prove the following proposition:
\begin{prop}\label{coffee}
{\it Suppose that $e^\varphi$ is continuous. Let
$h_+=(h_{+,1},\dots, h_{+,r})$ be a continuous supersolution to equation (\ref{phi}),
and let $h_-=(h_{-,1},\dots, h_{-,r})$ be a continuous subsolution to equation (\ref{phi})
such that $h_{-,j}\leq h_{+,j}$ for each $j=1,\dots, n$.
Then there exists a real solution $h=(h_1,\dots, h_r)$ to equation (\ref{phi})
such that $h_{-,j}\leq h_j\leq h_{+,j}$ for all $j=1,\dots, n$. 
}
\end{prop}
\begin{proof}We first suppose that $X$ is noncompact. Let $Y_1\subsetneq Y_2\subsetneq \cdots \subsetneq Y_i\subsetneq Y_{i+1}\subsetneq \cdots \subsetneq X$ be a family of relatively compact open subsets with smooth boundaries $\partial Y_i$ such that $\bigcup_{i=1}^\infty Y_i=X$. For each $i$, let $\partial h^{(i)}=(\partial h_1^{(i)},\dots, \partial h_r^{(i)})$ be a real diagonal Hermitian metric on the boundary $\partial Y_i$ of $Y_i$ such that for each $j=1,\dots, n$, $h_{-,j}\left.\right|_{\partial Y_i}\leq \partial h_j^{(i)}\leq h_{+,j}\left.\right|_{\partial Y_i}$ on $\partial Y_i$. Let $h^{(i)}=(h_1^{(i)},\dots, h_r^{(i)})$ be the solution to the Dirichlet problem (cf. \cite{Miy3}) for equation (\ref{phi}) with boundary condition $h^{(i)}\left.\right|_{\partial Y_i}=\partial h^{(i)}$ satisfying $\det(h^{(i)})=1$. Although the curvature lower bounds assumed in Lemma \ref{cands} need not hold here, the compactness of $\overline{Y}_i$ allows us to apply the same comparison argument at an attained maximum, using the local maximum principle for subharmonic functions (see Remark \ref{remark50} for more details). Hence $h_{-,j}\le h_j^{(i)}\le h_{+,j}$ for all $j=1,\ldots,n$. We obtain $h_{-,j}\leq h_j^{(i)}\leq h_{+,j}$ for each $i\geq 1$ and each $j=1,\dots, n$.  
By the interior $L^p$-estimate (cf. \cite[Theorem 9.11]{GT1}), for each $i\geq 1$, $1<p<\infty$, and $l\geq 1$, there exist positive constants $C_1$ and $C_2$ independent of $l$ such that
\begin{align*}
|\log h^{(i+l)}|_{L^p_2(Y_i)}\leq C_1(|\log h^{(i+l)}|_{L^p(Y_{i+1})}+|F_{h^{(i+l)}}|_{L^p(Y_{i+1})})\leq C_2.
\end{align*}
Since bounded subsets of reflexive Banach spaces are weakly relatively compact, for each $i\geq 2$, we can find a subsequence of $(h^{(i+l)})_{l\geq 1}$ that converges weakly in $L^p_2(Y_i)$ to a Hermitian metric $h$ on $Y_i$. Using the diagonal trick, we can find a subsequence of $(h^{(k)})_{k\geq 1}$ that converges weakly in $L^p_{2,\loc}$ to a Hermitian metric $h=(h_1,\dots, h_r)$ on $X$ belonging to $L^p_{2, \loc}$. After passing to a further subsequence if necessary, we may assume that $h_j^{(k)}$ converges to $h_j$ almost everywhere for every $j=1,\ldots,r$. Since the $h_j^{(k)}$ are uniformly bounded above and away from zero on compact subsets, Lebesgue's dominated convergence theorem allows us to pass to the limit in the nonlinear terms of equation (\ref{phi}). Therefore, the Hermitian metric $h$ solves equation (\ref{phi}) in the weak sense. Suppose now that $X$ is compact. Fix $p\in X$ and apply the noncompact case to $X\setminus\{p\}$ to obtain a real solution $h$ satisfying the required bounds. These bounds and the continuity of $e^\varphi$ imply that the local logarithmic weights of $h$ and the right-hand sides of their Poisson equations are bounded near $p$. Subtracting local Newtonian potentials of these right-hand sides yields bounded harmonic functions on a punctured coordinate disc. Their singularities are removable, so $h$ extends to a real weak solution on $X$. The required bounds hold at $p$ by continuity.
\end{proof}
\begin{rem}\label{remark50} We illustrate the comparison used in Proposition \ref{coffee} when $n=1$. Let $h$ be a continuous subsolution and $h^\prime$ a continuous supersolution on $\overline{Y_i}$, satisfying $h_1\leq h_1^\prime$ on $\partial Y_i$. Using the notation of Lemma \ref{cands}, set $\xi:=\log(h_1\otimes h_1'^{-1})$. Then, in the sense of distributions,
\begin{align*}
\inum\partial\bar\partial\xi
\geq
(1-e^{-2\xi})\vol(H_0)
+(e^{(4-r)\xi}-1)\vol(H_1).
\end{align*}
If $\max_{\overline{Y_i}}\xi>0$, this maximum is attained at some $p\in Y_i$. The right-hand side is strictly positive near $p$, so $\xi$ is subharmonic there. The maximum principle for subharmonic functions forces $\xi$ to be constant near $p$, contradicting the displayed inequality. Thus $\xi\leq0$ on $\overline{Y_i}$. Applying this comparison to $(h_-,h^{(i)})$ and $(h^{(i)},h_+)$ yields the inequalities used in Proposition \ref{coffee}.
\end{rem}

\subsection{Proof of Theorems \ref{main theorem 2} and \ref{main theorem 3} for the unit disc case}\label{udcase}
Suppose that $X$ is the unit disc $\D\coloneqq \{z\in\C\mid |z|<1\}$. Let $g_X$ be the Poincar\'e metric, i.e., the complete K\"ahler metric whose Gaussian curvature equals $-1$. We denote by $h_X$ the corresponding Hermitian metric on $K_X\rightarrow X$ and by $\omega_X$ the K\"ahler form. Note that $h_X$ and $\omega_X$ satisfy the following:
\begin{align*}
\inum F_{h_X}=\omega_X.
\end{align*}
We choose the reference metric $h_\refe$ as $h_X$. First, following \cite[Section 5]{LM1}, we prove the existence of a complete solution to equation (\ref{phi}) under the assumption that $e^\varphi$ is continuous. We set a constant $M_\varphi$, which may be infinite, as follows:
\begin{align*}
M_\varphi\coloneqq \max\{\sup_X e^{\varphi},1\}.
\end{align*}
For each $0<\epsilon<1$, let $\D_\epsilon\coloneqq \{z\in\C\mid |z|<1-\epsilon\}$ be the open disc of radius $1-\epsilon$ centered at $0$. Then we define a family of open subsets $(Y_i)_{i\geq 2}$ of $X=\D$ by $Y_i\coloneqq \D_{1/i}$ for $i=2,3,\dots$. For each $i=2,3,\dots$, we denote by $\partial Y_i$ the boundary of $Y_i$. Given each $i=2,3,\dots$ and each diagonal continuous Hermitian metric $\partial h^{(i)}=(\partial h_1^{(i)},\dots, \partial h_r^{(i)})$ on $\partial Y_i$ with $\det(\partial h^{(i)})=1$, there exists a solution $h^{(i)}=(h_1^{(i)},\dots, h_r^{(i)})$ to equation (\ref{phi}) such that $\lim_{z\to\zeta}h^{(i)}(z)=\partial h^{(i)}(\zeta)$ for any $\zeta\in \partial Y_i$ and that $\det(h^{(i)})=1$ (cf. \cite{Miy3}). Moreover, if $\partial h^{(i)}$ is real, then $h^{(i)}$ is also real (see \cite[Remark 5]{Miy3}).
From $h^{(i)}$, we define Hermitian metrics $H_j^{(i)}$ on $K_{Y_i}^{-1} \to Y_i$ for each $j=1,\dots, r-1$ by
\begin{align*}
H_j^{(i)}\coloneqq h_j^{(i)-1}\otimes h_{j+1}^{(i)},
\end{align*}
and define $H_0^{(i)}=H_r^{(i)}$ by
\begin{align*}
H_0^{(i)} = H_r^{(i)} \coloneqq e^{r\varphi} h_\refe^{-r}|_{Y_i} \otimes h_r^{(i)-1} \otimes h_1^{(i)}.
\end{align*}

\begin{lemm}\label{cmphi1}{\it Suppose that $e^\varphi$ is continuous and a bounded function. Then there exists a positive constant $C_{M_\varphi} $ depending only on $M_\varphi$ and $r$, such that for each $i = 2, 3, \dots$, by choosing the boundary condition appropriately, a real solution $h^{(i)} = (h_1^{(i)}, \dots, h_r^{(i)})$ to the Dirichlet problem for equation (\ref{phi}) on $Y_i$ satisfies the following estimate:
\begin{align}
\frac{1}{2}\omega_X \leq \vol(H_j^{(i)}) \leq C_{M_\varphi} \omega_X \ \text{for each $j = 1, \dots, n$}. \label{cmphi}
\end{align}
}
\end{lemm}
\begin{proof} 
Let $C_1$ and $C_2$ be constants as in inequality (\ref{C1C2}). We set $C_{M_\varphi}$ as follows:
\begin{align*}
C_{M_\varphi}\coloneqq \max\left\{\frac{C_2M_\varphi+1}{C_1}, \frac{r-1}{2}+2^{r-1}M_\varphi^r\right\}. 
\end{align*}
For each $i = 2, 3, \dots$, we define a Hermitian metric $\partial h^{(i)} \coloneqq (\partial h^{(i)}_1, \dots, \partial h^{(i)}_r)$ on the boundary $\partial Y_i$ of $Y_i$ as the unique one such that $\partial h_j^{(i)\,-1} \otimes \partial h_{j+1}^{(i)} =\frac{1}{2}h_X^{-1}\left.\right|_{\partial Y_i}$ holds for all $j = 1, \dots, r - 1$ and that $\det(\partial h^{(i)})=1$. Since $\partial h^{(i)}$ is real, $h^{(i)}$ is also real. We show that the Hermitian metrics \( H_1^{(i)}, \dots, H_{r-1}^{(i)} \), constructed from the solution \( h^{(i)} = (h_1^{(i)}, \dots, h_r^{(i)}) \) to the Dirichlet problem for equation (\ref{phi}) with boundary condition \( \lim_{z \to \zeta} h^{(i)}(z) = \partial h^{(i)}(\zeta) \), satisfy the estimate~(\ref{cmphi}). We first prove the estimate $\frac{1}{2}\omega_X \leq \vol(H_j^{(i)})$ for each $j=1,\dots,n$. For each $j=1,\dots, n$, $H_j^{(i)}$ satisfies the following:
\begin{align*}
\inum\partial\bar{\partial}\log(H_j^{(i)}\otimes h_X)&=-\inum F_{H_j^{(i)}}-\inum F_{h_X} \notag\\
&=2\vol(H_j^{(i)})-\vol(H_{j-1}^{(i)})-\vol(H_{j+1}^{(i)})-\omega_X \notag \\
&\leq 2\vol(H_j^{(i)})-\omega_X. 
\end{align*}
Therefore we obtain the following:
\begin{align}
\inum\Lambda_{h_X^{-1}}\partial\bar{\partial}\log(H_j^{(i)}\otimes h_X)\leq 2H_j^{(i)}\otimes h_X-1.\label{2Hji}
\end{align}
Note that by Proposition \ref{C11} and by the assumption that $e^\varphi$ is continuous, $H_j^{(i)}\ (j=1,\dots, n)$ are of class $C^2$ on $Y_i$, and thus we can apply the strong maximum principle to inequality (\ref{2Hji}). By the boundary condition and the strong maximum principle, we have $\frac{1}{2}\omega_X \leq \vol(H_j^{(i)})$ on $Y_i$. We next prove the estimate $\vol(H_j^{(i)}) \leq C_{M_\varphi}\omega_X$.  Since a locally bounded subharmonic function belongs to $W^{1,2}_\loc$ (cf. \cite[Section 1.1.1]{DGZ1}), for each smooth Hermitian metric $H$ on $K_{Y_i}^{-1}$, $\log(\sum_{j=1}^rH_j^{(i)}\otimes H^{-1})$ is continuous and belongs to $W^{1,2}_\loc$. Then by applying Proposition \ref{wstocb} to inequality (\ref{C1C2}) for $H_j=H_j^{(i)}\ (j=1,\dots, r)$ and $H=M_\varphi h_X^{-1}$, one of the following holds:
\begin{align*}
&0\geq C_1\sum_{j=1}^rH_j^{(i)}\otimes H^{-1}-C_2-M_\varphi^{-1}, \\
&\frac{r-1}{2}M_\varphi^{-1}+2^{r-1}(\max_{\partial Y_i}e^{r\varphi})M_\varphi^{-1}\geq \sum_{j=1}^rH_j^{(i)}\otimes H^{-1}.
\end{align*}
In either case, we have $\vol(H_j^{(i)})\leq \sum_{j=1}^r\vol(H_j^{(i)})\leq C_{M\varphi}\omega_X$ for each $j=1,\dots, n$. This implies the claim.
\end{proof}
Then following \cite[Proposition 5.6]{LM1}, we show the following:
\begin{lemm}\label{sandb}{\it Suppose that $e^\varphi$ is continuous and bounded. Then there exists a real $C^2$ complete solution to equation (\ref{phi}).
}
\end{lemm}
\begin{proof} 
By Lemma \ref{cmphi1} and the interior $L^p$-estimate (cf. \cite[Theorem 9.11]{GT1}), for each $i\geq 2$, $1<p<\infty$, and $l\geq 0$, there exist positive constants $C_1$ and $C_2$ that are independent of $l$ such that
\begin{align*}
|\log h^{(i+1+l)}|_{L^p_2(Y_i)}\leq C_1(|\log h^{(i+1+l)}|_{L^p(Y_{i+1})}+|F_{h^{(i+1+l)}}|_{L^p(Y_{i+1})})\leq C_2,
\end{align*}
where the $L^p_2$-norm and the $L^p$-norm are measured by the Poincar\'e metric. Since $L^p_2(Y_i)$ is reflexive, every bounded sequence in $L^p_2(Y_i)$ admits a weakly convergent subsequence. Hence, for each $i\geq 2$, we can extract a subsequence of $(h^{(i+1+l)})_{l\geq 0}$ that converges weakly in $L^p_2(Y_i)$ to a Hermitian metric on $Y_i$. By a diagonal argument, we can then extract a subsequence of $(h^{(k)})_{k\geq 2}$ that converges weakly in $L^p_{2,\loc}$ on $X$ to an $L^p_{2,\loc}$-Hermitian metric $h=(h_1,\dots,h_r)$.
The Hermitian metric $h$ solves equation (\ref{phi}) in the weak sense. By Proposition \ref{C11}, $h$ has a $C^2$ representative and solves equation (\ref{phi}) in the strong sense. Since $h^{(i)}\ (i=2,3,\dots,)$ are real, $h$ is also real. By Lemma \ref{cmphi1}, $h$ satisfies the following estimate:
\begin{align}
\frac{1}{2}\omega_X \leq \vol(H_j) \leq C_{M_\varphi} \omega_X \ \text{for each $j=1,\dots, r-1$}, \label{1/2omega}
\end{align}
where for each $j=1,\dots, r-1$, we have defined $H_j$ as $H_j\coloneqq h_j^{-1}\otimes h_{j+1}$. In particular, $h$ is complete. 
\end{proof}
Following \cite[Proposition 5.6]{LM1}, we drop the assumption that $e^\varphi$ is a bounded function.
\begin{lemm}\label{unbounded}{\it Suppose that $e^\varphi$ is continuous. Then there exists a real $C^2$ complete solution $h$ to equation (\ref{phi}).
}
\end{lemm}
\begin{proof} 
For each $i\geq 2$, let $h_c^{(i)}=(h_{c,1}^{(i)},\dots, h_{c,r}^{(i)})$ be the unique complete solution to equation (\ref{phi}) on $Y_i$ associated with $e^{-\varphi}h_\refe|_{Y_i}$. 
By (\ref{1/2omega}), $h_c^{(i)}$ is dominated by the complete hyperbolic metric as follows:
\begin{align}
h^{(i)}_{c, k}\leq (2h_{Y_i})^\frac{r+1-2k}{2}\ \text{for each $k=1,\dots, n$}, \label{cn-k}
\end{align} 
where we denote by $h_{Y_i}$ the Hermitian metric on $K_{Y_i}\rightarrow Y_i$ induced by the complete hyperbolic metric on $Y_i$. Since $h_{Y_i}$ tends to zero as $z$ approaches $\partial Y_i$, it follows that $h^{(k)-1}_{c,j}\otimes h^{(i)}_{c,j}(z)$ goes to 0
as $z$ approaches $\partial Y_i$, for each $j=1,\dots,n$. In particular, $h_{c,j}^{(k)-1} \otimes h_{c,j}^{(i)}$ is bounded above for each $j=1,\dots, n$. Then by Lemma \ref{cands}, we have $h_{c,j}^{(i)}<h_{c,j}^{(k)}$ for each $i<k$ and each $j=1,\dots, n$. By (\ref{cn-k}), we obtain the following uniform estimate: 
\begin{align*}
&h^{(i)}_{c, k}\leq (2h_X)^\frac{r+1-2k}{2}\ \text{for each $k=1,\dots, n$}.
\end{align*}
Therefore in particular we see that $(h_c^{(k)})_{k\geq i+1}$ is uniformly bounded above and below on $Y_i$. Then by the interior $L^p$-estimate (\cite[Theorem 9.11]{GT1}), for each $i\geq 2$, $1<p<\infty$, and $l\geq 0$, there exist positive constants $C_1$ and $C_2$ that are independent of $l$ such that
\begin{align*}
|\log h_c^{(i+2+l)}|_{L^p_2(Y_i)}\leq C_1(|\log h_c^{(i+2+l)}|_{L^p(Y_{i+1})}+|F_{h_c^{(i+2+l)}}|_{L^p(Y_{i+1})})\leq C_2,
\end{align*}
where the $L^p$-norm and the $L^p_2$-norms are measured by the Poincar\'e metric on $X$. Since every bounded sequence in a reflexive Banach space has a weakly convergent subsequence, for each $i\geq 2$, we can find a subsequence of $(h_c^{(i+2+l)})_{l\geq 0}$ that converges to a Hermitian metric $h$ on $Y_i$ in the $L^p_2(Y_i)$-weak sense. By using the diagonal trick, we can find a subsequence of $(h_c^{(k)})_{k\geq 2}$ that converges to an $L^p_{2,\loc}$-Hermitian metric $h=(h_1,\dots, h_r)$ on $X$ in the $L^p_{2,\loc}$-weak sense. The Hermitian metric $h$ solves equation (\ref{phi}) in the weak sense. By Proposition \ref{C11}, $h$ is $C^2$ and solves equation (\ref{phi}) in the strong sense. Since $h_c^{(i)} (i=2,3,\dots)$ are real, $h$ is also real.  By (\ref{1/2omega}), for each $i$, the complete solution $h_c^{(i)}$ satisfies the following estimate:
\begin{align*}
h_{c,j}^{(i)-1}\otimes h_{c,j+1}^{(i)}\geq \frac{1}{2}h_{Y_i}^{-1}\ \text{for each $j=1,\dots, n$}.
\end{align*}
Therefore, the solution $h=(h_1,\dots, h_r)$ satisfies the following estimate:
\begin{align}
h_j^{-1}\otimes h_{j+1}\geq \frac{1}{2}h_X^{-1}\ \text{for each $j=1,\dots, n$}. \label{hX-1}
\end{align}
In particular, $h$ is complete. 
\end{proof}
By the proof of Lemma \ref{unbounded}, we directly obtain the following lemma:
\begin{lemm}\label{khn}{\it Suppose that $e^\varphi$ is continuous. Then the unique complete solution $h=(h_1,\dots, h_r)$ to equation (\ref{phi}) associated with $\varphi$ satisfying $\det(h)=1$ satisfies the following estimate:
\begin{align*}
h_j^{-1}\otimes h_{j+1}\geq \frac{1}{2}h_X^{-1}\ \text{for each $j=1,\dots, n$}.
\end{align*}
In particular, the following holds:
\begin{align*}
h_j\leq (2h_X)^\frac{r+1-2j}{2}\ \text{for each $j=1,\dots, n$}.
\end{align*}
}
\end{lemm}
Let $(\varphi_\epsilon)_{0<\epsilon<1}$ be a family of smooth subharmonic weight functions satisfying the condition in Theorem \ref{main theorem 3}, where each $\varphi_\epsilon$ is defined on $\D_\epsilon$. By Theorem \ref{main theorem 1} and Lemma \ref{unbounded}, for each $0<\epsilon<1$, there exists a unique real complete solution $h_\epsilon$ associated with $\varphi_\epsilon$. At the end of this subsection, we prove
\begin{lemm}\label{lastlemm}{\it As $\epsilon\searrow 0$, $h_\epsilon$ converges monotonically to a real complete solution $h$ associated with $\varphi$.
}
\end{lemm}
\begin{proof}
As we remarked in Example \ref{subandsuper}, for each $\epsilon^\prime<\epsilon$, $h_{\epsilon^\prime}$ is a supersolution to equation (\ref{phi}) associated with $\varphi_\epsilon$. Since $\overline{\D_\epsilon}\subseteq \D_{\epsilon^\prime}$, the metric $h_{\epsilon^\prime}$ is smooth and positive on a neighborhood of $\overline{\D_\epsilon}$. On the other hand, by Lemma \ref{khn}, $h_\epsilon$ is dominated by the complete hyperbolic metric on $\D_\epsilon$, and hence $h_{\epsilon,j}\otimes h_{\epsilon',j}^{-1}(z)$ goes to 0 as $z$ approaches $\partial \D_\epsilon$, for each $j=1,\dots,n$. In particular, $h_{\epsilon,j}\otimes h_{\epsilon',j}^{-1}$ is bounded above on $\D_\epsilon$. Therefore by Lemma \ref{cands}, we obtain $h_{\epsilon,j}\otimes h_{\epsilon^\prime,j}^{-1}<1$ for each $j=1,\dots, n$. For each $i=2,3,\dots$, we denote by $h^{(i)}=(h^{(i)}_1,\dots, h^{(i)}_r)$ the complete solution $h_{1/i}$ defined on $Y_i=\D_{1/i}$. Then by monotonicity, Lemma \ref{khn}, and the interior $L^p$-estimate (\cite[Theorem 9.11]{GT1}), for each $i\geq 2$, $1<p<\infty$, and $l\geq 0$, there exist positive constants $C_1$ and $C_2$ that are independent of $l$ such that
\begin{align*}
|\log h^{(i+2+l)}|_{L^p_2(Y_i)}\leq C_1(|\log h^{(i+2+l)}|_{L^p(Y_{i+1})}+|F_{h^{(i+2+l)}}|_{L^p(Y_{i+1})})\leq C_2,
\end{align*}
where the $L^p$-norm and the $L^p_2$-norms are measured by the Poincar\'e metric on $X$. Then by compact Sobolev embedding, for each $i\geq 2$, we can find a subsequence of $(h^{(i+2+l)})_{l\geq 0}$ that converges to a Hermitian metric $h$ on $\overline{Y}_i$ in $C^0(\overline{Y}_i)$. By using the diagonal trick, we can find a subsequence of $(h^{(k)})_{k\geq 2}$ that locally uniformly converges to a continuous Hermitian metric $h=(h_1,\dots, h_r)$ on $X$. 
Monotonicity then implies that the whole family $(h_\epsilon)_{0<\epsilon<1}$ converges pointwise to $h$ as $\epsilon\searrow0$. By dominated convergence, $h$ solves equation (\ref{phi}) associated with $\varphi$ in the weak sense. Since $h^{(i)}$ is real for each $i$, $h$ is also real. 
By (\ref{hX-1}) in the proof of Lemma \ref{unbounded}, for each $0<\epsilon<1$, we have the following estimate:
\begin{align*}
h_{\epsilon,j}^{-1}\otimes h_{\epsilon,j+1}\geq \frac{1}{2}h_{\D_\epsilon}^{-1}\ \text{for each $j=1,\dots, n$},
\end{align*}
where $h_{\D_\epsilon}$ is the Hermitian metric induced by the hyperbolic metric. Therefore in the limit $\epsilon\searrow 0$, we obtain 
\begin{align*}
h_j^{-1}\otimes h_{j+1}\geq \frac{1}{2}h_X^{-1}\ \text{for each $j=1,\dots, n$}.
\end{align*}
In particular, the metric $h$ is complete. 
\end{proof}
As a corollary, we obtain
\begin{cor}\label{cor100}{\it On the unit disc, for any subharmonic weight function $\varphi$, there exists a unique real complete solution $h$ associated with $\varphi$. Moreover, for each smooth approximation $(\varphi_\epsilon)_{0<\epsilon<1}$ of $\varphi$ satisfying the condition in Theorem \ref{main theorem 3}, where each $\varphi_\epsilon$ is defined on the disc $\D_\epsilon$, the corresponding family $(h_\epsilon)_{0<\epsilon<1}$ of real complete solutions monotonically converges to the complete solution $h$. 
}
\end{cor}
\subsection{Proof of Theorems \ref{main theorem 2} and \ref{main theorem 3} for the complex plane case}
Throughout this subsection, we assume that the base Riemann surface is the complex plane $\C$. We also assume that the weight function $\varphi$ is not identically $-\infty$. By Lemma \ref{lastlemm}, on the unit disc, there exists a complete solution to equation (\ref{phi}). Therefore, by combining Theorem \ref{main theorem 1} with the standard argument in \cite[Proposition 5.7]{LM1}, on every hyperbolic Riemann surface, there exists a unique real complete solution to equation (\ref{phi}); a detailed proof will be given in Section \ref{proof of mc}. The estimate in Lemma \ref{khn} also holds on these surfaces for arbitrary subharmonic weight functions, by Lemma \ref{lastlemm} and the same covering argument. We denote by $\D_R=\{z\in\C\mid |z|<R\}$ the open ball of radius $R$ centered at the origin. We take positive real numbers $0<R_2<R_1$. Let $h^{<R_1}=(h^{<R_1}_1,\dots, h^{<R_1}_r)$ be the real complete solution on $\D_{R_1}$ associated with $\varphi|_{\D_{R_1}}$, and let $h^{>R_2}=(h^{>R_2}_1,\dots, h^{>R_2}_r)$ be the real complete solution on $\C\setminus \overline{\D}_{R_2}$ associated with $\varphi|_{\C\setminus \overline{\D}_{R_2}}$. Following the argument in \cite[Proposition 5.9]{LM1}, we first establish the following lemma:
\begin{lemm}\label{R1R2}{\it For each $0<R_2<R_1$, define a diagonal Hermitian metric
$h_-^{(R_1,R_2)}=(h_{-,1},\dots,h_{-,r})$ as follows. For $j=1,\dots,n$, set
\begin{align*} h_{-,j}^{(R_1,R_2)}\coloneqq 
\begin{cases} h^{<R_1}_j \quad \text{on } \overline{\D}_{R_2},\\ 
\max\{h_j^{>R_2}, h_j^{<R_1}\} \quad \text{on } \{z\in\C\mid R_2<|z|<R_1\},\\ 
h^{>R_2}_j \quad \text{on } \C\setminus \D_{R_1}.
\end{cases} 
\end{align*} 
The remaining components $h_{-,j}^{(R_1,R_2)}$, $j=n+1,\dots,r$, are defined so that $h_-^{(R_1,R_2)}$ is real. Then $h_-^{(R_1,R_2)}$ is a continuous subsolution to equation (\ref{phi}) and belongs to $W^{1,2}_{\loc}$. } 
\end{lemm}
\begin{proof}By Proposition \ref{min-max}, $h_-^{(R_1,R_2)}$ is a subsolution to equation (\ref{phi}) on $\{z\in\C\mid R_2< |z|< R_1\}$. Since $\max:\R^2\to \R$ is Lipschitz continuous, it belongs to $W^{1,\infty}_{\loc}$ (cf. \cite[Theorem 4.5]{EG1}). Therefore, since $h_{>R_2}$ and $h_{<R_1}$ belong to $W^{1,2}_\loc$, by the chain rule for Sobolev functions, $h_-^{(R_1,R_2)}$ belongs to $W^{1,2}_{\loc}$ on $\{z\in\C\mid R_2<|z|<R_1\}$. Moreover, $h_j^{<R_1}$ and $h_j^{>R_2}$, $j=1,\dots,n$, satisfy the estimates in terms of the hyperbolic metrics given in Lemma \ref{khn} as noted at the beginning of this subsection. Although that lemma assumes that $e^\varphi$ is continuous, we obtain the same estimates for the corresponding limiting complete solutions by passing to the limit along the smooth approximations in Lemma \ref{lastlemm}. Therefore, $h_-^{(R_1,R_2)}$ coincides with $h^{<R_1}$ near the boundary of $\D_{R_2}$ and with $h^{>R_2}$ near the boundary of $\D_{R_1}$. Hence, $h_-^{(R_1,R_2)}$ is continuous and belongs to $W^{1,2}_{\loc}$. This completes the proof. 
\end{proof}
In addition to $R_1$ and $R_2$, we take positive real numbers $R_1^\prime, R_2^\prime$ satisfying $0<R_2<R_2^\prime<R_1^\prime<R_1$. Let $h^{(\varphi)}=(h_1^{(\varphi)},\dots, h_r^{(\varphi)})$ be the supersolution constructed from $e^{-\varphi}h_\refe$ as in Example \ref{H1Hr}. For each positive constant $c$, define a diagonal Hermitian metric $h_+^{(R_2^\prime, R_2)}=(h_{+,1}^{(R_2^\prime, R_2)},\dots, h_{+,r}^{(R_2^\prime, R_2)})$ as follows. For $j=1,\dots,n$, set
\begin{align*} h_{+,j}^{(R_2^\prime, R_2)}\coloneqq 
\begin{cases} h^{(\varphi)}_j \ \text{on} \ \D_{R_2^\prime},\\ 
\min\{h_j^{(\varphi)}, e^ch_j^{>R_2}\} \quad \text{on } \C\setminus\D_{R_2^\prime}.
\end{cases} 
\end{align*}
The remaining components $h_{+,j}^{(R_2^\prime, R_2)}$, $j=n+1,\dots,r$, are defined so that $h_+^{(R_2^\prime,R_2)}$ is real. By Examples \ref{H1Hr} and \ref{genkidase} and Proposition \ref{min-max}, outside the boundary of $\D_{R_2^\prime}$, $h_+^{(R_2^\prime, R_2)}$ is a supersolution to equation (\ref{phi}). 
Following the argument in \cite[Proposition 5.9]{LM1}, we prove the following:
\begin{lemm}\label{lemma51}{\it 
The following holds:
\begin{enumerate}[(i)]
\item Suppose that the following conditions hold:
\begin{itemize}
\item On $\overline{\D}_{R_1^\prime}$, $e^{-\varphi}$ is finite and continuous.
\item On $\C\setminus\varphi^{-1}(-\infty)$, $\varphi$ is continuous and belongs to $W^{1,2}_\loc$.
\item The positive constant $c$ is large enough so that the following holds on $\{z\in\C\mid R^\prime_2\leq |z|\leq R_1^\prime\}$:
\begin{align}
h_j^{(\varphi)}<e^ch_j^{>R_2}\ \text{for $j=1,\dots, n$}. \label{echj}
\end{align} 
\end{itemize}
Then $h_+^{(R_2^\prime, R_2)}$ is continuous everywhere and belongs to $W^{1,2}_\loc(\C)$. In particular, $h_+^{(R_2^\prime, R_2)}$ is a supersolution on the entire complex plane, including the boundary of $\D_{R_2^\prime}$.
\item Suppose that $c$ is large enough so that the following holds on $\{z\in\C\mid R_2^\prime\leq |z|<R_1\}$:
\begin{align}
h^{<R_1}_j<e^ch^{>R_2}_j\ \text{for $j=1,\dots, n$.} \label{echR}
\end{align}
Then we obtain
\begin{align}
h_{-,j}^{(R_1,R_2)}<h_{+,j}^{(R_2^\prime, R_2)}\ \text{for $j=1,\dots, n$}. \label{74}
\end{align}
\end{enumerate}
}
\end{lemm}
\begin{proof}We first prove assertion (i). We begin by noting the following two points: 
\begin{enumerate}
\item By the subharmonicity of $\varphi$ and its continuity outside the polar set $\varphi^{-1}(-\infty)$, the polar set is a closed set. 
\item By the upper semicontinuity of $\varphi$ and continuity of $h_j^{>R_2}$, for each $x\in \varphi^{-1}(-\infty)$, there exists an open neighborhood $U$ of $x$ such that $\min\{h_j^{(\varphi)}, e^ch_j^{>R_2}\}=e^ch_j^{>R_2}$ on $U$. 
\end{enumerate}
Therefore, since $e^ch_j^{>R_2}$ is continuous and belongs to $W^{1,2}_\loc$, it is enough to prove, on the open set $\C\setminus \varphi^{-1}(-\infty)$, $h_+^{(R_2^\prime, R_2)}$ is continuous and belongs to $W^{1,2}_\loc$. We then prove the global continuity of $h_+^{(R_2^\prime, R_2)}$. It remains only to check continuity across the boundary of $\D_{R_2^\prime}$. Since $e^{-\varphi}$ is continuous on $\overline{\D}_{R_1^\prime}$, so is $h^{(\varphi)}$. Moreover, by (\ref{echj}), $h_+^{(R'_2,R_2)}$ coincides with $h^{(\varphi)}$ in a neighborhood of $\partial \D_{R^\prime_2}\coloneqq\{z\in\C\mid |z|=R_2^\prime\}$. This implies the continuity of $h^{(R'_2,R_2)}_+$. We then prove that $h^{(R'_2,R_2)}_+$ belongs to $W^{1,2}_\loc$. Since $\varphi$ is continuous on $\C\setminus \varphi^{-1}(-\infty)$, it is locally bounded on $\C\setminus \varphi^{-1}(-\infty)$. Therefore, the assumption that $\varphi$ belongs to $W^{1,2}_\loc(\C\setminus \varphi^{-1}(-\infty))$ implies that $h^{(\varphi)}$ also belongs to $W^{1,2}_\loc(\C\setminus \varphi^{-1}(-\infty))$. Since $h^{(R'_2,R_2)}_+$ coincides with $h^{(\varphi)}$ on a neighbourhood of $\overline{\D}_{R_2^\prime}$, it is enough to prove the assertion on $\C\setminus \varphi^{-1}(-\infty)\cap \C\setminus\overline{\D}_{R_2^\prime}$. Since
$h^{>R_2}$ belongs to $W^{1,2}_{\loc}$ and
$\min:\mathbb R^2\to\mathbb R$ is Lipschitz continuous, the chain
rule for Sobolev functions implies that $\min\{h^{(\varphi)}_j,e^c h^{>R_2}_j\}$ belongs to  $W^{1,2}_{\loc}$
for each $j=1,\ldots,n$. Therefore, it follows that 
$h^{(R^\prime_2,R_2)}_+$ belongs to $W^{1,2}_{\loc}(\C)$. The supersolution property on all of $\C$ follows immediately, given that $h^{(R_2^\prime, R_2)}_+$ coincides with $h^{(\varphi)}$ in the neighbourhood of the boundary of $\D_{R_2^\prime}$. This proves assertion (i). We next prove assertion (ii).  From Lemma \ref{lemma33} and \cite[Proposition 27]{Miy4}, we have
\begin{align}
h_j^{<R_1}<h_j^{(\varphi)} \ \text{and} \ h_j^{>R_2}<h_j^{(\varphi)} \label{75}
\end{align}
for all $j=1,\dots, n$. Therefore, on $\D_{R_2^\prime}$, we obtain (\ref{74}) from (\ref{75}). On $\D_{R_1}\setminus \D_{R^\prime_2}$, by (\ref{echR}), (\ref{75}), and $c>0$, we have $h^{<R_1}_j<
 \min\{h^{(\varphi)}_j,e^c h^{>R_2}_j\}$
and $h^{>R_2}_j<\min\{h^{(\varphi)}_j,e^c h^{>R_2}_j\}$.
Hence, 
\begin{align*}
h^{(R_1,R_2)}_{-,j}
=\max\{h^{<R_1}_j,h^{>R_2}_j\}<\min\{h^{(\varphi)}_j,e^c h^{>R_2}_j\}
=h^{(R'_2,R_2)}_{+,j}.
\end{align*}
Finally, on $\C\setminus \D_{R_1}$, we have $h^{(R_1,R_2)}_{-,j}=h^{>R_2}_j.$
Since $c>0$, it follows from (\ref{75}) that
\begin{align*}
h^{>R_2}_j<\min\{h^{(\varphi)}_j,e^c h^{>R_2}_j\}=h^{(R'_2,R_2)}_{+,j}.
\end{align*}
Thus (\ref{74}) holds on $\C$, which proves assertion (ii).
\end{proof}
\begin{lemm}\label{le52}{\it The subsolution $h_-^{(R_1,R_2)}$ is complete. For each positive constant $c$, $h_-^{(R_1,R_2)}$ and $h^{(R^\prime_2,R_2)}_+$ are mutually bounded on $\C\setminus \D_{R_1}$}.
\end{lemm}
\begin{proof} Since $h^{>R_2}$ is complete, $h_-^{(R_1,R_2)}$ is also complete. As in the proof of Lemma \ref{lemma51}, for each $c>0$, we have the following estimate on $\C\setminus \D_{R_1}$:
\begin{align*}
h_j^{>R_2}=h_{-,j}^{(R_1,R_2)}<h_{+,j}^{(R_2^\prime, R_2)}\leq e^ch_j^{>R_2}\ \text{for $j=1,\dots, n$}.
\end{align*}
This proves the latter half of the statement.
\end{proof}
\begin{cor}\label{cor52} {\it Suppose that 
\begin{itemize}
\item $e^{-\varphi}$ is finite and continuous on some closed ball of radius $R_1^\prime>0$, 
\item On $\C\setminus\varphi^{-1}(-\infty)$, $\varphi$ is continuous and belongs to $W^{1,2}_\loc$.
\end{itemize}
Then there exists a real complete solution $h=(h_1,\dots, h_r)$ to equation (\ref{phi}). Furthermore, the complete solution $h$ has the following property concerning its size and rate of divergence: After a translation, we may assume that the closed ball is centered at the origin. Let $R_2, R_2^\prime, R_1,c$ be positive real numbers satisfying the following properties:
\begin{enumerate}
\item $0<R_2<R_2^\prime<R_1^\prime<R_1$. 
\item The following holds on $\{z\in\C\mid R^\prime_2\leq |z|\leq R_1^\prime\}$:
\begin{align*}
h_j^{(\varphi)}<e^ch_j^{>R_2}\ \text{for $j=1,\dots, n$}. 
\end{align*} 
\item The following holds on $\{z\in\C\mid R_2^\prime\leq |z|<R_1\}$:
\begin{align*}
h^{<R_1}_j<e^ch^{>R_2}_j\ \text{for $j=1,\dots, n$.} 
\end{align*}
\end{enumerate}
Then we have the following estimate on $\C\setminus \D_{R_1}$:
\begin{align}
h_j^{>R_2}\leq h_j\leq e^ch_j^{>R_2}\ \text{for $j=1,\dots, n$}. \label{RhcR}
\end{align}
}
\end{cor}
\begin{proof} We may assume from the beginning that the center of the closed ball of radius $R_1^\prime$ is the origin. Take positive real numbers $R_2, R_2^\prime, R_1$ so that $R_2<R_2^\prime<R_1^\prime<R_1$. Take a positive constant $c$ large enough so that (\ref{echj}) and (\ref{echR}) hold. Then the supersolution $h_+^{(R_2^\prime,R_2)}$ and a subsolution $h_-^{(R_1,R_2)}$ associated with these real numbers constructed as in Lemmas \ref{R1R2} and \ref{lemma51} satisfy (\ref{74}), are globally continuous, and belong to $W^{1,2}_\loc$. By the assumption that $\varphi$ is continuous on $\C\setminus \varphi^{-1}(-\infty)$, $e^\varphi$ is continuous. Therefore by Proposition \ref{coffee}, there exists a real solution $h$ to equation (\ref{phi}) satisfying 
\begin{align*}
h_{-,j}^{(R_1,R_2)}\leq h_j\leq h_{+,j}^{(R_2^\prime,R_2)}.
\end{align*}
By the construction, $h$ is real, and in particular $\det(h)=1$. By Lemma \ref{le52} and the above inequalities, $h$ is mutually bounded with the complete metric $h_-^{(R_1,R_2)}$ outside a compact set. Thus $h$ is complete. Since $h$ is real, $\det(h)=1$, and Theorem \ref{main theorem 1} shows that $h$ is independent of the choices of $R_2,R^\prime_2,R_1,c$. Estimate (\ref{RhcR}) follows from the definitions. 
\end{proof}
Let $(\varphi_\epsilon)_{\epsilon>0}$ be a smooth approximation of the weight function $\varphi$ satisfying the condition in Theorem \ref{main theorem 3}. From Corollary \ref{cor52} and Theorem \ref{main theorem 1}, for each $\epsilon>0$, there exists a unique real smooth complete solution $h_\epsilon=(h_{\epsilon,1},\dots, h_{\epsilon,r})$ associated with $\varphi_\epsilon$. We prove the corresponding family of real complete solutions $(h_\epsilon)_{\epsilon>0}$ converges to a real complete solution $h$ associated with $\varphi$. 
For each $\epsilon>0$ and each $j=1,\dots,r-1$, we set $H_{\epsilon,j}\coloneqq h_{\epsilon,j}^{-1}\otimes h_{\epsilon,j+1}$, and we also define $H_{\epsilon,r}\coloneqq (e^{-\varphi_\epsilon}h_\refe)^{-r}\otimes H_{\epsilon, 1}^{-1}\otimes\cdots \otimes H_{\epsilon,r-1}^{-1}$. We also take a positive real number $R_2$. We denote by $h_\epsilon^{>R_2}=(h_{\epsilon,1}^{>R_2},\dots, h_{\epsilon,r}^{>R_2})$ the unique real complete solution associated with $\varphi_\epsilon|_{\C\setminus \overline{\D}_{R_2}}$. We also set $H_{\epsilon,j}^{>R_2}\coloneqq  h_{\epsilon,j}^{>R_2-1}\otimes h_{\epsilon,j+1}^{>R_2}$ for each $j=1,\dots, r-1$, and set $H_{\epsilon,r}^{>R_2}\coloneqq (e^{-\varphi_\epsilon}h_\refe|_{\C\setminus \overline{\D}_{R_2}})^{-r}\otimes H_{\epsilon, 1}^{>R_2-1}\otimes\cdots \otimes H_{\epsilon,r-1}^{>R_2-1}$. We divide the proof of Theorems \ref{main theorem 2} and \ref{main theorem 3} for the complex plane case into the following five lemmas.
\begin{lemm}\label{le54}{\it For each $0<\epsilon<\epsilon^\prime$ and $j=1,\dots, n$, the following holds:
\begin{enumerate}
\item $H_{\epsilon,j}\otimes H_{\epsilon^\prime,j}^{-1}$ and $H_{\epsilon,j}^{>R_2}\otimes H_{\epsilon^\prime,j}^{>R_2-1}$ are bounded above.
\item  It holds that
\begin{align*}
h_{\epsilon^\prime,j}\leq h_{\epsilon,j}\ \text{and} \ h_{\epsilon^\prime,j}^{>R_2}\leq h_{\epsilon,j}^{>R_2}.
\end{align*}
\end{enumerate}
}
\end{lemm}
\begin{proof}Fix $j=1,\dots,n$.  We first prove assertion 1. Since the proofs are the same, we only prove the assertion for $H_{\epsilon,j}$ and $H_{\epsilon^\prime,j}$ for simplicity. Because $e^{\varphi_\epsilon^\prime}h_\refe^{-1}\otimes H_{\epsilon^\prime,j}^{-1}$ is bounded, $e^{\varphi_\epsilon}h_\refe^{-1}\otimes H_{\epsilon^\prime,j}^{-1}$ is also bounded. Then by Corollary \ref{CYM}, we observe that $H_{\epsilon,j}\otimes H_{\epsilon^\prime,j}^{-1}$ is bounded. We next prove assertion 2. Since the proofs are the same, we only prove the assertion for $h_{\epsilon,j}$ and $h_{\epsilon^\prime,j}$ for simplicity. The following holds:
\begin{align*}
(h_{\epsilon^\prime,j}\otimes h_{\epsilon,j}^{-1})^2=\Pi_{k=j}^{r-j}(H_{\epsilon^\prime,k}^{-1}\otimes H_{\epsilon,k})\ \text{for each $j=1,\dots,n$}.
\end{align*}
Therefore, by assertion 1, $h_{\epsilon^\prime,j}\otimes h_{\epsilon,j}^{-1}$ is bounded, and thus by Lemma \ref{cands}, we have the claim.
\end{proof}
\begin{lemm}\label{le55}{\it The following holds for all $\epsilon>0$ and $j=1,\dots, n$:
\begin{align*}
h_{\epsilon,j}\leq h^{(\varphi)}_j=(e^{-\varphi}h_\refe)^{\frac{r-(2j-1)}{2}}.
\end{align*}
}
\end{lemm}
\begin{proof}
From Lemma \ref{lemma33}, \cite[Proposition 27]{Miy4}, and the inequality $\varphi\leq \varphi_\epsilon$, we obtain
\begin{align*}
h_{\epsilon,j}\leq h^{(\varphi_\epsilon)}_j\leq h^{(\varphi)}_j,
\end{align*}
and thus we have the claim.
\end{proof}

\begin{lemm}\label{le56}{\it The following holds:
\begin{enumerate}
\item There exists a positive constant $C(R_2)$ such that on $\C\setminus \overline{\D}_{R_2}$, the following holds for all $\epsilon>0$ and $j=1,\dots, n$:
\begin{align}
h_{\epsilon,j}^{-1}\otimes h_{\epsilon,j}^{>R_2}\leq C(R_2). \label{79}
\end{align}
\item 
For each real number $R_1$ which is greater than $R_2$ and each $\epsilon>0$, there exists a positive constant $C(\epsilon, R_1, R_2)$ such that on $\C\setminus \D_{R_1}$, the following holds for all $j=1,\dots, n$:
\begin{align}
h_{\epsilon,j}\otimes h_{\epsilon,j}^{>R_2-1}\leq C(\epsilon, R_1, R_2). \label{80}
\end{align}
\end{enumerate}
}
\end{lemm}
\begin{proof}We first prove assertion 1. 
By Corollary \ref{ji}, there exists a constant $C$ independent of $\epsilon$ such that
\begin{align*}
e^{\varphi_\epsilon}h_\refe^{-1}\otimes H_{\epsilon,j}^{>R_2-1}\leq C\ \text{for each $j=1,\dots,n$}. 
\end{align*}
Applying Lemma \ref{ephi} to $H=CH_{\epsilon,j}^{>R_2}$, we have
\begin{align}
\inum\Lambda_{H_{\epsilon,j}^{>R_2}}\partial\bar{\partial}\log(\sum_{k=1}^r\vol(H_{\epsilon,k})/\vol(H_{\epsilon,j}^{>R_2})) \geq C_1(\sum_{k=1}^r\vol(H_{\epsilon,k})/\vol(H_{\epsilon,j}^{>R_2}))-CC_2-2.\label{2/C}
\end{align}
By applying Proposition \ref{SL} to inequality (\ref{2/C}), we obtain
\begin{align*}
\sum_{k=1}^r(\vol(H_{\epsilon,k})/\vol(H_{\epsilon,j}^{>R_2}))\leq \frac{CC_2+2}{C_1}.
\end{align*}
In particular, it follows that $H_{\epsilon,k}\otimes H_{\epsilon,j}^{>R_2-1}$ is bounded above by a constant independent of $\epsilon$, $j=1,\dots, r-1$, and $k=1,\dots, r$. As a consequence, it follows that 
\begin{align*}
(h_{\epsilon,j}^{>R_2}\otimes h_{\epsilon,j}^{-1})^2=\Pi_{k=j}^{r-j}(H_{\epsilon,k}^{>R_2-1}\otimes H_{\epsilon,k})
\end{align*}
is bounded above by a constant independent of $\epsilon$ and $j$. This establishes inequality (\ref{79}). We next provide a proof of assertion 2, which is similar to the proof of assertion 1. For each $\epsilon>0$ and $j=1,\dots, r-1$, let $H_{\epsilon,j}^{(R_1,R_2)}$ be a smooth complete Hermitian metric on the anticanonical bundle of $\C\setminus \overline{\D}_{R_2}$ satisfying the following properties:
\begin{enumerate}
\item On $\C\setminus \D_{R_1}$, $H_{\epsilon,j}^{(R_1,R_2)}$ coincides with $H_{\epsilon,j}$.
\item On $\D_{\frac{R_1+R_2}{2}}\setminus \overline{\D}_{R_2}$, $H_{\epsilon,j}^{(R_1,R_2)}$ coincides with the hyperbolic metric on $\C\setminus \overline{\D}_{R_2}$.
\end{enumerate}
Then by definition, we can easily confirm that $e^{\varphi_\epsilon}h_\refe^{-1}|_{\C\setminus \overline{\D}_{R_2}}\otimes H_{\epsilon,j}^{(R_1,R_2)-1}$ is bounded. Since the Gaussian curvature of $H_{\epsilon,j}^{(R_1,R_2)}$ is bounded below, by Lemma \ref{ephi} and Corollary \ref{CYM}, $H^{>R_2}_{\epsilon,j}\otimes H_{\epsilon,j}^{(R_1,R_2)-1}$ is bounded above by a constant that may depend on $\epsilon$. As a consequence, since $H_{\epsilon,j}^{(R_1,R_2)}$ coincides with $H_{\epsilon,j}$ on $\C\setminus \D_{R_1}$, for each $j=1,\dots, n$, it follows that 
\begin{align*}
(h_{\epsilon,j}^{>R_2-1}\otimes h_{\epsilon,j})^2=\Pi_{k=j}^{r-j}(H_{\epsilon,k}^{>R_2}\otimes H_{\epsilon,k}^{-1})
\end{align*}
is bounded above on $\C\setminus \D_{R_1}$. This completes the proof. 
\end{proof}

\begin{lemm}\label{le57}{\it For each $j=1,\dots, n$, define $h_j$ and $\widetilde{h}_j^{>R_2}$ pointwise by
\begin{align*}
h_j(x)&\coloneqq \sup_{\epsilon>0}h_{\epsilon,j}(x)\ \text{for $x\in\C$}, \\
\widetilde{h}_j^{>R_2}(x)&\coloneqq \sup_{\epsilon>0}h_{\epsilon,j}^{>R_2}(x)\ \text{for $x\in \C\setminus \overline{\D}_{R_2}$}.
\end{align*}
For the remaining indices $j=n+1,\dots, r$, we define $h_j$ and $\widetilde{h}_j^{>R_2}$ so that $h=(h_1,\dots, h_r)$ and $\widetilde{h}^{>R_2}=(\widetilde{h}^{>R_2}_1,\dots, \widetilde{h}^{>R_2}_r)$ are real. Then the following holds:
\begin{enumerate}
\item For each $j=1,\dots, n$, $h_j$ is of class $C^{2j-1,1}_\loc$. $h$ solves equation (\ref{phi}) in the sense of distributions. 
\item $\widetilde{h}^{>R_2}=h^{>R_2}$.
\end{enumerate}
}
\end{lemm}
\begin{proof}We first prove assertion 1. We set $H_j\coloneqq h_j^{-1}\otimes h_{j+1}$ for each $j=1,\dots, r-1$ and $H_r\coloneqq (e^{-\varphi}h_\refe)^{-r}\otimes H_1^{-1}\otimes\cdots\otimes H_{r-1}^{-1}$. Then by Lemma \ref{le55}, $h$ is finite outside the polar set of $\varphi$. Take a smooth reference metric $h_\refe$ on $K_\C$ and define weight functions $\phi_{\epsilon,j}\ (\epsilon>0)$ and $\phi_j$ by 
\begin{align*}
h_j=e^{-\phi_j}h_\refe^{\frac{r-(2j-1)}{2}} \ \text{and}\ h_{\epsilon,j}=e^{-\phi_{\epsilon,j}}h_\refe^{\frac{r-(2j-1)}{2}}
\end{align*}
for $j=1,\dots, n$. By Lemma \ref{le55} and the monotonicity of $h_{\epsilon,j}$, for every fixed $\epsilon_0>0$ and $0<\epsilon<\epsilon_0$, we have
\begin{align}
\frac{r-(2j-1)}{2}\varphi\leq \phi_{\epsilon,j}\leq \phi_{\epsilon_0,j}. \label{phi*3}
\end{align}
Since both $\frac{r-(2j-1)}{2}\varphi$ and $\phi_{\epsilon_0,j}$ belong to $L^1_\loc$, the dominated convergence theorem implies that $\phi_{\epsilon,j}$ converges to $\phi_j$ in $L^1_\loc$ for $j=1,\dots, n$. Therefore, the curvature $F_{h_{\epsilon,j}}$ converges to $F_{h_j}$ in the sense of distributions. On the other hand, for each fixed $\epsilon_0>0$ and each $0<\epsilon<\epsilon_0$, it holds that
\begin{align*}
0<H_{\epsilon,n}=h_{\epsilon,n}^{-(2n+2-r)}\leq h_{\epsilon_0,n}^{-(2n+2-r)}.
\end{align*}
Also, by \cite[Proposition 27]{Miy4}, we obtain
\begin{align*}
0\leq H_{\epsilon,j}\leq H_{\epsilon,n}\leq h_{\epsilon_0,n}^{-(2n+2-r)}
\end{align*}
for each $j=0,1,\dots, n-1$ and $0<\epsilon<\epsilon_0$. 
Therefore,  $H_1,\dots, H_r$ belong to $L^\infty_\loc(\C)$ and by the Lebesgue dominated convergence theorem, for each $j=1,\dots, r$, as $\epsilon\searrow0$, $H_{\epsilon,j}$ converges to $H_j$ in the $L^1_\loc$ sense. Therefore, by taking the limit $\epsilon\searrow 0$ in the equation
\begin{align*}
\inum F_{h_{\epsilon,j}}+\vol(H_{\epsilon,j-1})-\vol(H_{\epsilon,j})=0 \ \text{for $j=1,\dots,n$},
\end{align*}
we conclude that 
\begin{align*}
\inum F_{h_j}+\vol(H_{j-1})-\vol(H_j)=0 \ \text{for $j=1,\dots,n$},
\end{align*}
in the sense of distributions. By reality, the remaining equations follow as well. Therefore, $h$ solves equation (\ref{phi}) in the sense of distributions. By Proposition \ref{C11}, for each $j=1,\dots, n$, $h_j$ agrees almost everywhere with a $C^{2j-1,1}_\loc$ Hermitian metric. We claim that the equality holds everywhere. Indeed, since every subharmonic function that is not identically $-\infty$ belongs to $L^p_\loc$ for every $1\leq p<\infty$, by (\ref{phi*3}), for a fixed $\epsilon_0>0$, the $L^p_\loc$-norm of $(\phi_{\epsilon,j})_{0<\epsilon<\epsilon_0}$ is uniformly bounded, and thus, the standard interior $L^p$-estimate together with the uniform $L^p$-bounds of the curvatures $\inum F_{h_{\epsilon,j}}$ gives a uniform bound on the $W^{2,p}_\loc$-norms of $(\phi_{\epsilon,j})_{0<\epsilon<\epsilon_0}$. Taking $p=2$, compact Sobolev embedding and a diagonal argument show that every sequence $\epsilon_i\to0$ admits a subsequence along which the local weights $\phi_{\epsilon_i,j}$, and hence the metrics $h_{{\epsilon_i},j}$, converge locally uniformly. By uniqueness of the $L^1_\loc$-limit of the weight function, the limiting metric is the $C^{1,1}_\loc$-representative of $h_j$. Therefore $h_{\epsilon,j}$ converges to this representative in $C^0_\loc$, and the pointwise limit $h_j$ agrees with it everywhere. This completes the proof of assertion 1. 

We next prove assertion 2. The proof of assertion 2 is essentially the same as that of assertion 1, so we only indicate the necessary modifications. By the same arguments as in Lemmas \ref{le54} and \ref{le55}, together with \cite[Proposition 27]{Miy4}, we can verify that $\widetilde{h}^{>R_2}$ solves equation (\ref{phi}) in the sense of distributions. By Proposition \ref{C11}, for each $j=1,\dots,n$, $\widetilde{h}^{>R_2}_j$ agrees almost everywhere with a $C^{2j-1,1}_\loc$ Hermitian metric. By repeating the argument above, we can verify that $h_{\epsilon}^{>R_2}$ converges to $\widetilde{h}^{>R_2}$ in $C^0_\loc$. Hence, the pointwise limit $\widetilde h_j^{>R_2}$ agrees everywhere with this representative.
We set
\begin{align*}
\widetilde{H}_j^{>R_2}\coloneqq \widetilde{h}_j^{>R_2-1}\otimes \widetilde{h}_{j+1}^{>R_2}\ \text{for $j=1,\dots, r-1$}.
\end{align*}
By Lemma \ref{khn}, $H_{\epsilon,j}^{>R_2}$ is bounded below by the hyperbolic metric with a factor of $1/2$. Therefore, $\widetilde{H}_j^{>R_2}$ is bounded below by the hyperbolic metric with a factor of $1/2$ and is, in particular, complete.  Finally, by  reality, $\det(\widetilde{h}^{>R_2})=1$. Therefore, by the uniqueness of the complete solution, we obtain $\widetilde{h}^{>R_2}=h^{>R_2}$.
\end{proof}

\begin{lemm}\label{le2000}{\it The Hermitian metric $h$ in Lemma \ref{le57} is mutually bounded with $h^{>R_2}$ on $\C\setminus\D_{R_1}$, and in particular, it is complete.}
\end{lemm}  
\begin{proof} 
By Lemma \ref{le57}, as $\epsilon\searrow0$, $h_{\epsilon,j}$ and $h_{\epsilon,j}^{>R_2}$ converge to $h_j$ and $h_j^{>R_2}$, respectively. Therefore, by taking the limit $\epsilon\searrow0$ in inequality (\ref{79}) in Lemma \ref{le56}, we obtain
\begin{align*}
h_j^{-1}\otimes h_j^{>R_2}\leq C(R_2)\ \text{for $j=1,\dots,n$}.
\end{align*}
We prove the converse inequality. Fix $\epsilon_0>0$. Take a constant $c$ so that
\begin{align*}
c>\max\{\max_{1\leq j\leq n}\sup_{z\in \partial \D_{R_1}}\log(h_j(z)\otimes h_{\epsilon_0,j}^{>R_2-1}(z)),0\},
\end{align*}
where $\partial \D_{R_1}$ denotes the boundary $\{z\in\C\mid |z|=R_1\}$ of $\D_{R_1}$. By inequality (\ref{80}) in Lemma \ref{le56}, $h_{\epsilon,j}\otimes h_{\epsilon,j}^{>R_2-1}$ is bounded above. Then by Lemma \ref{le57}, we have the following for each $j=1,\dots, n$ and $0<\epsilon<\epsilon_0$:
\begin{align*}
(h_{\epsilon,j}\otimes h_{\epsilon,j}^{>R_2-1})(z)\leq e^c\ \text{for all $z\in \partial \D_{R_1}$.}
\end{align*}
Then by Example \ref{genkidase} and Lemma \ref{cands}, we obtain 
\begin{align*}
e^{-c}h_{\epsilon,j}\otimes (h_{\epsilon,j}^{>R_2-1})\leq 1.
\end{align*}
Since the constant $c$ does not depend on $0<\epsilon<\epsilon_0$, taking the limit $\epsilon\searrow 0$, we obtain
\begin{align*}
h_j\otimes h_j^{>R_2-1}\leq e^c. 
\end{align*}
This implies the claim. 
\end{proof}
As a corollary, we obtain
\begin{cor}\label{cor1000}{\it On the complex plane, for any subharmonic weight function $\varphi$ that is not identically $-\infty$, there exists a unique real complete solution $h$ associated with $\varphi$. Moreover, for each smooth approximation $(\varphi_\epsilon)_{0<\epsilon<1}$ of $\varphi$ satisfying the condition in Theorem \ref{main theorem 3}, the corresponding family $(h_\epsilon)_{0<\epsilon<1}$ of real complete solutions monotonically converges to the complete solution $h$. Furthermore, the complete solution $h$ has the property in Theorem \ref{main theorem 2}.
}
\end{cor}
\subsection{Proof of Theorem \ref{main theorem 2} for the general case}\label{proof of mc}
\begin{theo}\label{th1000}{\it For any Riemann surface $X$ and any subharmonic weight function $\varphi$ on $K_X$ which is not identically $-\infty$ unless $X$ is hyperbolic, there exists a real complete solution associated with $\varphi$.}
\end{theo}
\begin{proof}
The argument follows \cite[Proof of Propositions 5.7 and 5.10]{LM1}. Let $\pi:\widetilde{X}\rightarrow X$ be the universal covering map. Let $\varphi$ be a subharmonic weight function which is not identically $-\infty$ unless $X$ is hyperbolic. The case $X=\CP^1$ is excluded by the assumption on $\varphi$. Thus, by uniformization, $\widetilde{X}$ is biholomorphic to either $\D$ or $\C$. We denote by $\pi^\ast \varphi$ the pullback of this weight function by $\pi$. As in Section \ref{pb}, by setting $K_{\widetilde{X}}^{1/2}\coloneqq \pi^\ast K_X^{1/2}$, there is a natural isomorphism
\begin{align*}
\pi^\ast\left(\K_{X,r}\right)\simeq \K_{\widetilde{X},r},
\end{align*}
where we have used the notation $\K_{X,r}\coloneqq \bigoplus_{j=1}^rK_X^{\frac{r-(2j-1)}{2}}$ and $\K_{\widetilde{X},r}\coloneqq \bigoplus_{j=1}^rK_{\widetilde{X}}^{\frac{r-(2j-1)}{2}}$ to explicitly denote the underlying Riemann surfaces.
Then by Theorem \ref{main theorem 1} and Corollaries \ref{cor100} and \ref{cor1000}, there exists a unique real complete solution $\widetilde{h}$ on $\K_{\widetilde{X},r}$ associated with $\pi^\ast\varphi$. For each $\gamma\in \pi_1(X)$, we contend that $\widetilde{h}$ is $\gamma$-invariant. Indeed, since $\pi^\ast\varphi$ is $\gamma$-invariant, by Proposition \ref{hsolves}, $\gamma^\ast \widetilde{h}$ also solves equation (\ref{phi}) associated to $\pi^\ast\varphi$, where $\gamma^\ast \K_{\widetilde{X},r}$ and $\K_{\widetilde{X},r}$ are naturally identified. By Theorem \ref{main theorem 1}, we have $\widetilde{h}=\gamma^\ast \widetilde{h}$. Therefore, $\widetilde{h}$ is $\pi_1(X)$-invariant, and thus, there exists a unique Hermitian metric $h$ on $\K_{X,r}$ such that $\widetilde{h}=\pi^\ast h$. By Proposition \ref{hsolves}, $h$ is a solution to equation (\ref{phi}), and moreover, by Proposition \ref{pistarphi}, $h$ is a complete solution. Since $\widetilde{h}$ is real, $h$ is also real. This implies the claim. 
\end{proof}
From Corollaries \ref{cor100} and \ref{cor1000} and Theorem \ref{th1000}, we obtain Theorems \ref{main theorem 2} and \ref{main theorem 3}.
\appendix
\section{Regularity of solutions to equation (\ref{phi})}\label{appA}
The purpose of this section is to prove the following proposition:
\begin{prop}\label{C11}{\it Let $h=(h_1,\dots, h_r)$ be a solution to equation (\ref{phi}) with $\det(h)=1$, possibly incomplete. Then $h$ has the following regularity properties:
\begin{itemize}
\item $h$ is of class $C^{1,1}_\loc$. 
\item $h$ is of class $C^2$ if and only if $e^\varphi$ is continuous. 
\end{itemize}
}
\end{prop}
By invoking well-established general theory for the Poisson equation, it is immediately apparent that the regularity of the solution extends to $C^{1,\alpha}_\loc$ for any $0<\alpha<1$ (cf. \cite{Miy3}). In fact, Proposition \ref{C11} asserts that the regularity can be further improved to $C^{1,1}_\loc$, while, in general, it cannot reach $C^2$. The proof of Proposition \ref{C11} is based on the following lemma:
\begin{lemm}\label{lemapp}{\it Let $U\subseteq\C$ be a connected open subset and $\Delta\coloneqq \frac{\partial^2}{\partial x^2}+\frac{\partial^2}{\partial y^2}=4\frac{\partial^2}{\partial z\bar{\partial}\bar{z}}$ the standard Laplacian. Let $u,f\in L^1_\loc(U)$ satisfy the following in the sense of distributions:
\begin{align*}
\Delta u=f.
\end{align*}
Suppose that $\Delta f$ is a locally finite signed measure. Then the following holds:
\begin{itemize}
\item If $f\in L^\infty_\loc(U)$, then $u$ has a $C^{1,1}_\loc$ representative. 
\item If $f\in C^0(U)$, then $u$ has a $C^2$ representative. 
\end{itemize}
}
\end{lemm}
\begin{proof} We begin by recalling the fundamental solution of the bi-Laplacian $\Delta^2$ (cf. \cite{CF1, Mu1}). Define $F:\C\rightarrow \R$ by
\begin{align*}
F(z)\coloneqq 
\begin{cases}
\frac{1}{8\pi}|z|^2\log|z| \ \text{if $z\neq 0$}\\
0 \ \text{if $z=0$}. 
\end{cases}
\end{align*}
Then one can easily verify that
\begin{align*}
\Delta^2F=\delta_0
\end{align*}
in the sense of distributions, where $\delta_0$ denotes the Dirac measure (cf. \cite{CF1} and \cite[Lemma 3.1]{Mu1}). Furthermore, outside the origin, for the Hessian of $F$, we obtain
\begin{align}D^2F\coloneqq 
\begin{pmatrix}
\frac{\partial^2F}{\partial x^2}&\frac{\partial^2F}{\partial x\partial y} \\
\frac{\partial^2F}{\partial y\partial x} &\frac{\partial^2F}{\partial y^2}
\end{pmatrix}=\frac{1}{8\pi r^2}
\begin{pmatrix}
 (2\log r+1)r^2+2x^2&   2xy  \\       
2xy &    (2\log r+1)r^2+2y^2
\end{pmatrix},\label{D2F}
\end{align}
where $r$ is defined by $r\coloneqq \sqrt{x^2+y^2}$. Since $F\in W^{2,1}_\loc(\C)$ by \cite[Lemma 3.1]{Mu1}, (\ref{D2F}) also holds in the sense of distributions on $\C$. Let $v=(x\ y)^\T$ be a $2\times 1$ column vector with components $x$ and $y$ and $I_2$ the $2\times 2$ unit matrix. Then (\ref{D2F}) can be rewritten as follows:
\begin{align}
D^2F&=\left(\frac{1}{2}\Delta F-\frac{1}{8\pi}\right)I_2+\frac{2}{8\pi r^2}vv^\T \\
&=\frac{1}{2}(\Delta F)I_2+\frac{1}{8\pi}\left(\frac{2}{r^2}vv^\T-I_2\right). \label{D2F2}
\end{align}
Let $U^\prime\subseteq U$ be an arbitrary relatively compact open subset. Let $\chi:U\rightarrow \R$ be a smooth function with compact support such that $\chi|_{U^\prime}\equiv1$. Let $\nu$ be the finite signed measure on $\R^2$ extending $\chi\Delta f$ to the outside of $U$ by setting it to 0. We set $h\coloneqq u-F\ast \nu$, where $\ast$ denotes the convolution. Then $h$ satisfies $\Delta^2h=0$ on $U^\prime$, and thus in particular it is smooth on $U^\prime$. By (\ref{D2F2}), we have the following in the sense of distributions:
\begin{align}
D^2u&=D^2h+D^2(F\ast \nu) \notag \\
&=D^2h+(D^2F)\ast\nu \notag \\
&=D^2h+\left\{\frac{1}{2}(\Delta F)I_2+\frac{1}{8\pi}\left(\frac{2}{r^2}vv^\T-I_2\right)\right\}\ast\nu \notag \\
&=D^2h+\frac{1}{2}\left\{(\Delta u)-(\Delta h)\right\}I_2+\frac{1}{8\pi}\left(\frac{2}{r^2}vv^\T-I_2\right)\ast\nu \notag \\
&=\frac{1}{2}fI_2+\left(D^2h-\frac{1}{2}\left(\Delta h\right)I_2\right)+\frac{1}{8\pi}\left(\frac{2}{r^2}vv^\T-I_2\right)\ast\nu. \label{nonotag}
\end{align}
The matrix-valued function $\frac{2}{r^2}vv^\T-I_2$ is defined everywhere except at the origin; we extend it to the origin by setting its value there to 0, and denote this extended function by $K$. Since $\nu$ is a finite signed measure, its total variation $|\nu|$ is finite and satisfies $|\int_\C bd\nu|\leq \int_\C |b|d|\nu|$ for every $b\in L^1(|\nu|)$. This implies the following estimate for each $x\in\C$:
\begin{align*}
\left|K\ast\nu\right|(x)&=\left|\int_\C K(x-y)d\nu(y)\right| \\
&\leq \int_\C |K(x-y)|d|\nu|(y) \\
&\leq |K|_{L^\infty}\int_\C d|\nu|.
\end{align*}
Suppose that $f\in L^\infty_\loc$. Then from the argument above, the right-hand side of (\ref{nonotag}) belongs to $L^\infty_\loc(U^\prime)$. Since $U^\prime$ is an arbitrary relatively compact open subset, $u$ belongs to $W^{2,\infty}_\loc(U)$ and therefore it has a $C^{1,1}_\loc$ representative. Suppose that $f$ is continuous.  We first show that $\mu\coloneqq \Delta f$ has no atoms. Fix $a\in U$. We decompose $\mu$ as $\mu = \mu^+ - \mu^-$ as a difference of two locally finite nonnegative measures $\mu^+$ and $\mu^-$ (for instance, one may take the Jordan decomposition). Choose local subharmonic potentials $p_\pm$ satisfying $\Delta p_\pm=\mu^\pm$. By Weyl's lemma, $p_+-p_--f$ has a harmonic representative. Since $f$ is continuous, the Lelong numbers of $p_+$ and $p_-$ coincide. The standard characterization of Lelong numbers therefore gives $\mu^+(\{a\})=\mu^-(\{a\})$. Consequently, $\mu(\{a\})=0$. 
As above, we take an arbitrary relatively compact open subset $U^\prime\subseteq U$ and let $\chi:U\rightarrow \R$ be a smooth function with compact support such that $\chi|_{U^\prime}\equiv1$. We define $\nu$ to be the finite signed measure on $\R^2$ extending $\chi\Delta f$ to the outside of $U$ by setting it to 0. Since $\mu$ has no atoms, neither do $\nu$ and $|\nu|$. This implies that $K\ast \nu$ is continuous. Indeed, let $x\in \C$ be an arbitrary point and $(x_n)_{n\geq 1}$ an arbitrary sequence that converges to $x$. Since $|\nu|$ has no atoms and $K$ is continuous except for the origin, $K(x_n-y)$ converges to $K(x-y)$ for $|\nu|$-almost every $y$. Also, we have a uniform estimate $|K(x_n-y)-K(x-y)|\leq 2\sup_\C|K|$ for all $n\geq 1$ and $y\in\C$. This estimate and Lebesgue's dominated convergence theorem imply that
\begin{align*}
|(K\ast\nu)(x_n)-(K\ast\nu)(x)|&\leq \int_\C|K(x_n-y)-K(x-y)|d|\nu|(y) \\
&\rightarrow 0\ \text{as}\ n\rightarrow\infty. 
\end{align*}
Therefore, by (\ref{nonotag}), $u$ has a $C^2$ representative on $U^\prime$, and since $U^\prime$ is an arbitrary relatively compact open subset, $u$ has a $C^2$ representative on $U$. 
\end{proof}
We establish a sufficient condition on $f$ for the requirement in Lemma \ref{lemapp} to hold:
\begin{lemm}\label{sois}{\it Let $U\subseteq\C$ be open, and let $f\in L^\infty_\loc(U)$ be positive almost everywhere with $\log f\in L^1_\loc(U)$. If $\Delta\log f$ is a locally finite signed measure, then $\Delta f$ is also a locally finite signed measure.}
\end{lemm}
\begin{proof} Let $V\subset U$ be an arbitrary relatively compact open subset and $(\rho_\epsilon)_{\epsilon>0}$ a family of standard mollifiers for mollification on $V$. We set $\widetilde{f}\coloneqq\log f$ and for each $\epsilon>0$, we denote by $\widetilde{f}_\epsilon\coloneqq \widetilde{f}\ast \rho_\epsilon$ the mollified function defined on an open set smaller than $V$. We take an arbitrary relatively compact open subset $V^\prime\subseteq V$ and $\epsilon_0>0$ small enough so that $\widetilde{f}_{\epsilon_0}$ is defined on $V^\prime$. Then by the assumption that $\log f$ is locally essentially bounded above, there is a constant $M$ such that $\widetilde{f}_\epsilon\leq M$ on $V^\prime$ for all $0<\epsilon<\epsilon_0$.  By a direct calculation, we have
\begin{align}
\Delta(e^{\widetilde{f}_\epsilon})&= e^{\widetilde{f}_\epsilon}((\Delta \widetilde{f})\ast\rho_\epsilon)+e^{\widetilde{f}_\epsilon}| d\widetilde{f}_\epsilon|^2 \notag \\
&\geq -e^M(|\Delta \widetilde{f}|\ast\rho_\epsilon), \label{totvar}
\end{align}
where $|\Delta \widetilde{f}|$ denotes the total variation of $\Delta \widetilde{f}$. Since $e^{\widetilde{f}_\epsilon}$ converges to $f=e^{\widetilde{f}}$ in $L^1_\loc$, taking the limit as $\epsilon\searrow 0$ in (\ref{totvar}), we obtain 
\begin{align}
\Delta f+e^M|\Delta \widetilde{f}|\geq 0 \label{MDelta}
\end{align}
on $V^\prime$. Since positive distributions are locally finite positive measures, (\ref{MDelta}) shows that $\Delta f+e^M|\Delta\widetilde f|$ is a locally finite positive measure on $V^\prime$. Hence $\Delta f$ is a locally finite signed measure on $V^\prime$. Since $V\subseteq U$ and $V^\prime\subseteq V$ were arbitrary, the conclusion follows. 
\end{proof}
\begin{proof}[Proof of Proposition \ref{C11}] If $\varphi\equiv-\infty$, then $H_r\equiv0$, and standard elliptic bootstrapping shows that $h$ is smooth. We may therefore assume that $\varphi\not\equiv-\infty$. It is enough to prove the assertions only for $h_1$ and $h_r$ since by the general theory of the Poisson equation (cf. \cite{Miy3}) and the standard elliptic regularity, $h_j \ (2\leq j\leq r-1)$ are at least in the class $C_\loc^{3,\alpha}$ for any $0<\alpha<1$. Since the proofs are the same, we prove that $h_1$ belongs to $C^{1,1}_\loc$ for simplicity. We work in a local holomorphic coordinate and identify the metrics with their coefficient functions; Lemma \ref{lemapp} is applied to $\log h_1$. Note that by Lemma \ref{sois}, the Laplacians of $H_1$ and $H_r$ are locally finite because the curvatures of $H_1$ and $H_r$ are locally finite. Therefore we can apply Lemma \ref{lemapp} to the equation
\begin{align}
\inum F_{h_1}+\vol(H_r)-\vol(H_1)=0 \label{thelastlabel}
\end{align}
and it follows that $h_1$ belongs to $C^{1,1}_\loc$. Suppose that $e^\varphi$ is continuous. Then $H_1$ and $H_r$ are continuous, and thus by applying Lemma \ref{lemapp} to (\ref{thelastlabel}), $h_1$ is of class $C^2$. Conversely, suppose that $h_1$ is of class $C^2$. By this assumption, the curvature term in (\ref{thelastlabel}) is continuous. Therefore equality (\ref{thelastlabel}) implies that $e^\varphi$ is continuous, which concludes the proof.
\end{proof}

\noindent
E-mail address 1: natsuo.miyatake.e8@tohoku.ac.jp

\noindent
E-mail address 2: natsuo.m.math@gmail.com \\

\noindent
Mathematical Science Center for Co-creative Society, Tohoku University, 468-1 Aramaki Azaaoba, Aoba-ku, Sendai 980-0845, Japan.

\end{document}